\documentclass[a4paper,english]{article}
\usepackage{lmodern}
\usepackage[T1]{fontenc}
\usepackage[latin9]{inputenc}
\usepackage{color}
\usepackage{amsmath}
\usepackage{amsthm}
\usepackage{amssymb}
\usepackage{esint}
\usepackage[numbers]{natbib}

\makeatletter

  \theoremstyle{remark}
  \newtheorem*{acknowledgement*}{\protect\acknowledgementname}
\theoremstyle{plain}
\newtheorem{thm}{\protect\theoremname}[section]
  \theoremstyle{plain}
  \newtheorem{lem}[thm]{\protect\lemmaname}
  \theoremstyle{plain}
  \newtheorem{assumption}[thm]{\protect\assumptionname}
  \theoremstyle{plain}
  \newtheorem{prop}[thm]{\protect\propositionname}
  \theoremstyle{plain}
  \newtheorem{cor}[thm]{\protect\corollaryname}
  \theoremstyle{definition}
  \newtheorem{defn}[thm]{\protect\definitionname}
  \theoremstyle{definition}
  \newtheorem{example}[thm]{\protect\examplename}
  \theoremstyle{remark}
  \newtheorem{rem}[thm]{\protect\remarkname}

\usepackage{bbm}
\usepackage{a4wide}

 \usepackage[T1]{fontenc}

\makeatother

\usepackage{babel}
  \providecommand{\acknowledgementname}{Acknowledgement}
  \providecommand{\assumptionname}{Assumption}
  \providecommand{\corollaryname}{Corollary}
  \providecommand{\definitionname}{Definition}
  \providecommand{\examplename}{Example}
  \providecommand{\lemmaname}{Lemma}
  \providecommand{\propositionname}{Proposition}
  \providecommand{\remarkname}{Remark}
\providecommand{\theoremname}{Theorem}

\begin{document}
\global\long\def\phi{\varphi}
\global\long\def\epsilon{\varepsilon}
\global\long\def\theta{\vartheta}
\global\long\def\E{\mathbb{E}}
\global\long\def\N{\mathbb{N}}
\global\long\def\Z{\mathbb{Z}}
\global\long\def\R{\mathbb{R}}
\global\long\def\F{\mathcal{F}}
\global\long\def\le{\leqslant}
\global\long\def\ge{\geqslant}
\global\long\def\MT{\ensuremath{\clubsuit}}
\global\long\def\1{\mathbbm1}
\global\long\def\d{\,\mathrm{d}}
\global\long\def\dd{\mathrm{d}}
\global\long\def\leq{\le}
\global\long\def\geq{\ge}
 \global\long\def\subset{\subseteq}
\global\long\def\supset{\supseteq}
\global\long\def\argmin{\arg\,\min}
\global\long\def\bull{{\scriptstyle \bullet}}
\global\long\def\supp{\operatorname{supp}}
\global\long\def\sgn{\operatorname{sign}}

\title{Rough differential equations driven by signals in Besov spaces}

\author{David J. Prömel\thanks{Eidgenössische Technische Hochschule Zürich, Email: david.proemel@math.ethz.ch }
~and Mathias Trabs\thanks{Université Paris-Dauphine, Email: trabs@ceremade.dauphine.fr} }
\maketitle
\begin{abstract}
Rough differential equations are solved for signals in general Besov
spaces unifying in particular the known results in Hölder and p-variation
topology. To this end the paracontrolled distribution approach, which
has been introduced by Gubinelli, Imkeller and Perkowski \citep{Gubinelli2012}
to analyze singular stochastic PDEs, is extended from Hölder to Besov
spaces. As an application we solve stochastic differential equations
driven by random functions in Besov spaces and Gaussian processes
in a pathwise sense. 
\end{abstract}
\noindent \textbf{Key words:} Besov regularity, Itô map, Paradifferential
calculus, Rough differential equation, Geometric Besov rough path,
Stochastic differential equation, Young integration. 

\noindent \textbf{MSC 2010 Classification:} Primary: 34A34, 60H10;
Secondary: 30H25, 46N20, 46N30.

\section{Introduction }

Differential equations belong to the most fundamental objects in numerous
areas of mathematics gaining extra interest from their various fields
of applications. A very important sub-class of classical ordinary
differential equations (ODEs) are controlled ODEs, whose dynamics
are given by

\begin{equation}
\dd u(t)=F(u(t))\xi(t),\quad u(0)=u_{0},\quad t\in\mathbb{R},\label{eq:rde}
\end{equation}
where $u_{0}\in\mathbb{R}^{m}$ is the initial condition, $u\colon\mathbb{R}\to\mathbb{R}^{m}$
is a continuous function, $\dd$ denotes the differential operator
and $F\colon\mathbb{R}^{m}\to\mathcal{L}(\mathbb{R}^{n},\mathbb{R}^{m})$
is a family of vector fields on $\R^{m}$. In such a dynamic $\xi\colon\mathbb{R}\to\mathbb{R}^{n}$
typically models the input signal and $u$ the output. 

If the signal $\xi$ is very irregular, for instance if $\xi$ has
the regularity of white noise, equation (\ref{eq:rde}) is called
rough differential equation (RDE). Starting with the seminal paper
by \citet{Lyons1998}, the theory of rough paths has been developed
to solve and analyze rough differential equations over the last two
decades. A significant insight due to \citet{Lyons1998} was that
the driving signal $\xi$ must be enhanced to a \textquotedblright rough
path\textquotedblright{} in some sense, in order to solve the RDE~(\ref{eq:rde})
and to restore the continuity of the Itô map defined by $\xi\mapsto u$
in a $p$-variation topology, cf. \citep{Lyons2002,Lyons2007,Friz2010}.
In particular, the rough path framework allows for treating important
examples as stochastic differential equations in a non-probabilistic
setting. Parallel to the $p$-variation results, rough differential
equations have been analyzed in the Hölder topology with similar tools,
cf. \citep{Friz2005a,Friz2013}.

One core goal of this article is to unify the approach via the $p$-variation
and the one via the Hölder topology in a common framework. To this
end, we deal with rough differential equations on the very large and
flexible class of Besov spaces $B_{p,q}^{\alpha}$, noting that, loosely
speaking, the space of $\alpha$-Hölder regular functions is given
by the Besov space $B_{\infty,\infty}^{\alpha}$ and that the $p$-variation
scale corresponds to $B_{p,q}^{1/p}$ (see \citep{BourdaudEtAl2006}).
The results by \citet{Zahle1998,Zahle2001,Zahle2005}, who set up
integration for functions in Sobolev\textendash Slobodeckij spaces
via fractional calculus, are covered by our results as well. In fact,
Besov spaces unify numerous function spaces, including also Sobolev
spaces and Bessel-potential spaces, for a comprehensive monograph
we refer to \citet{triebel2010}. Furthermore, different types of
Besov spaces and Besov embeddings already appear naturally in various
applications of rough path theory. Let us mention, for instance, their
use to derive large deviation principles \citep{Ledoux2002,Inahama2015},
a non-Markovian Hörmander theory for RDEs \citep{Cass2010,Cass2015}
and certain embedding results in the context rough path \citep{Hara2007,Friz2006}.

Due to this generality, studying solutions to the RDE~(\ref{eq:rde})
on Besov spaces is a highly interesting, but challenging problem.
In a first step, provided the driving signal $\xi$ is in $B_{p,q}^{\alpha-1}$
for $\alpha>1/2$, $p\ge2$, $q\ge1$, the existence and uniqueness
of a solution $u$ to the RDE~(\ref{eq:rde}) is proven, see Theorem~\ref{thm:Young},
and further it is shown that the corresponding Itô map is locally
Lipschitz continuous with respect to the Besov topology, see Theorem~\ref{thm:lipschitz young}.
In particular, with these results we recover the classical Young integration
\citep{Young1936} on Besov spaces. 

In order to handle a more irregular driving signal $\xi$ in $B_{p,q}^{\alpha-1}$
for $\alpha>1/3$, $p\ge3$, $q\ge1$, the path itself has to be enhanced
with an additional information, say $\pi(\theta,\xi)$, which always
exists for a smooth path $\xi$ and corresponds to the first iterated
integral in rough path theory. In the spirit of the usual notion of
geometric rough path, this leads naturally to the new definition of
the space of geometric Besov rough paths $\mathcal{B}_{p,q}^{0,\alpha}$,
cf. Definition~\ref{def:geometric rough path}. Starting with a smooth
path $\xi$, it is shown that the Itô map associated to the RDE~(\ref{eq:rde})
extends continuously to the space of geometric Besov rough path, cf.
Theorem~\ref{thm:gobalsolution}. As a consequence there exists a
unique pathwise solution to the RDE~(\ref{eq:rde}) driven by a geometric
Besov rough path. Note that due to $\alpha>1/p$ our results are restricted
to continuous solutions, which seems to appear rather naturally, see
Remark~\ref{remark:jumps} for a discussion. Especially, for signals
which are not self-similar like Brownian motion but whose regularity
is determined by rare singularities, we can profit from measuring
regularity in general Besov norms. 

The immediate and highly non-trivial problem appearing in equation
(\ref{eq:rde}) is that the product $F(u)\xi$ is not well-defined
for very irregular signals. While classical approaches as rough path
theory formally integrate equation~(\ref{eq:rde}) and then give
the appearing integral a meaning, the first step of our analysis is
to give a direct meaning to the product in (\ref{eq:rde}). Our analysis
relies on the notion of paracontrolled distributions, very recently
introduced by \citet{Gubinelli2012} on the Hölder spaces $B_{\infty,\infty}^{\alpha}$.
Their key insight is that by applying Bony's decomposition to $F(u)\xi$
the appearing resonant term can be reduced to the resonant term $\pi(\theta,\xi)$
of $\xi$ and its antiderivative $\theta$, using a controlled ansatz
to the solution $u$. The resonant term $\pi(\theta,\xi)$ turns out
to be the necessary additional information to show the existence of
a pathwise solution and corresponds to the first iterated integral
in rough path theory as already mentioned above.

Generalizing the approach from \citep{Gubinelli2012} to Besov spaces
poses severe additional problems, which are solved by using the Besov
space characterizations via Littlewood-Paley blocks as well as the
one via the modulus of continuity. Besov spaces are a Banach algebra
if and only if $p=q=\infty$ such that in general our results can
only rely on pointwise multiplier theorems, Bony's decomposition and
Besov embeddings. We thus need to generalize certain results in \citep{Bahouri2011}
and \citep{Gubinelli2012}, including the commutator lemma, see Lemma~\ref{lem:commutator}.
A second difficulty is that $u\in B_{p,q}^{\alpha}$ imposes an $L^{p}$-integrability
condition on $u$. To overcome this problem, we localize the signal
and consider a weighted Itô(-Lyons) map, both done in a way that does
not change the dynamic of the RDE on a compact interval around the
origin.

The paracontrolled distribution approach \citep{Gubinelli2012} offers
an extension of rough path theory to a multiparameter setting as also
done by the innovative theory of regularity structures developed by
\citet{Hairer2014}. While Hairer's theory presumably has a much wider
range of applicability, both successfully give a meaning to many stochastic
partial differential equations (PDEs) like the KPZ equation \citep{Hairer2013a,Gubinelli2015}
and the dynamical $\Phi_{3}^{4}$ equation \citep{Hairer2014,Catellier2013}
just to name two. Even if the approach of \citet{Gubinelli2012} may
not be a systematic theory as regularity structures, it comprises
some advantages. The approach works with already well-studied tools
like Bony's paraproduct and Littlewood-Paley theory, which leads to
globally defined objects rather than the locally operating \textquotedblleft jets\textquotedblright{}
appearing in the theory of regularity structures. Since for stochastic
PDEs the question about the ``most suitable'' function spaces seems
not to be settled yet, it might be quite promising on its own to extend
\citep{Gubinelli2012} to a more general foundation as we do by working
with general Besov spaces. For instance, let us refer to the very
recent work of \citet{Hairer2015}, where the theory of regularity
structures is adapted to a setting of weighted Besov spaces. 

In probability theory the prototypical example of the differential
equation~(\ref{eq:rde}) is a stochastic differential equation driven
by a fractional Brownian motion $B^{H}$ with Hurst index $H>0$.
It is well-known that the Besov regularity of such a fractional Brownian
motion is $B_{p,\infty}^{H}$ for $p\in[1,\infty)$ and thus the results
of the present paper are applicable. For our Besov setting, an even
more interesting example coming from stochastic analysis, recalling
for example the Karhunen-Loève theorem, are Gaussian processes and
stochastic processes given by a basis expansion with random coefficients,
see e.g. \citet{Friz2013a}. The Besov regularity of such random functions
can be determined sharply and they are well-studied for instance when
investigating the regularity of solutions for certain stochastic PDEs
\citep{Chioica2012} or in non-parametric Bayesian statistics \citep{Abramovich1998,Bochkina2013}.
In order to make our results about RDEs accessible for these examples,
we prove all the required sample path properties in Section~\ref{sec:SDE},
especially the existence of the resonant term is provided.

This work is organized as follows. Section~\ref{sec:functional analysis}
introduces the functional analytic framework and gives some preliminary
results. In Section~\ref{sec:Young} we recover Young integration
on Besov spaces and deal with differential equations driven by paths
with regularity $\alpha>1/2$. The analytic foundation of the paracontrolled
distribution approach on general Besov spaces is presented in Section~\ref{sec:LinCom}.
The application of the paracontrolled ansatz to rough differential
equations is developed in Section~\ref{sec:paracontrolled ansatz}
and in Section~\ref{sec:SDE} it is used to solve certain stochastic
differential equations. In Appendix~\ref{sec:appendix} some known
results about Besov spaces are recalled and the proof for the local
Lipschitz continuity of the Itô map is given.
\begin{acknowledgement*}
The authors would like to thank Peter Imkeller and Nicolas Perkowski
for fruitful discussions. D.J.P. gratefully acknowledges the financial
support by the DFG Research Training Group 1845 \textquotedbl{}Stochastic
Analysis with Applications in Biology, Finance and Physics\textquotedbl{}
and the Swiss National Foundation under Grant No. 200021\_163014.
M.T. gratefully acknowledges the financial support by a DFG research
fellowship TR 1349/1-1. The main part of the paper was carried out
while M.T. was employed and D.J.P. was a Ph.D. student at the Humboldt-Universität
zu Berlin.
\end{acknowledgement*}

\section{Functional analytic preliminaries\label{sec:functional analysis}}

For our analysis we need to recall the definition of Besov spaces,
some elements of the Littlewood-Paley theory and Bony's paraproduct.
For the properties of Besov spaces we refer to \citet{triebel2010}.
The calculus of Bony's paraproduct is comprehensively studied by \citet{Bahouri2011},
from which we also borrow most of our notation. 

For the sake of clarification let us mention that $L^{p}(\R^{d},\R^{m})$
denotes the space of Lebesgue $p$-integrable functions for $p\in(0,\infty)$
and $L^{\infty}(\R^{d},\R^{m})$ denotes the space of bounded functions
with the (quasi-)norms $\|\cdot\|_{L^{p}}$, $p\in(0,\infty]$. The
space of $\alpha$-Hölder continuous functions $f\colon\R^{d}\to\R^{m}$
is denoted by $C^{\alpha}$ equipped with the Hölder norm 
\[
\|f\|_{\alpha}:=\sum_{|k|<\lfloor\alpha\rfloor}\|f^{(k)}\|_{L^{\infty}}+\sum_{|k|=\lfloor\alpha\rfloor}\sup_{x\neq y}\frac{|f^{(k)}(x)-f^{(k)}(y)|}{|x-y|^{\alpha-\lfloor\alpha\rfloor}},
\]
where $k$ denotes multi-indices with usual conventions and where
$\lfloor\alpha\rfloor$ denotes the integer part of $\alpha>0$. For
operator valued functions $F\colon\mathbb{R}^{m}\to\mathcal{L}(\mathbb{R}^{n},\mathbb{R}^{m})$
we write $F\in C^{n}$, $n\in\mathbb{N}$, if $F$ is bounded, continuous
and $n$-times differentiable with bounded and continuous derivatives,
and we use the abbreviation $C:=C^{0}$. The first and second derivative
are denoted by $F^{\prime}$ and $F^{\prime\prime}$, respectively,
and higher derivatives by $F^{(n)}$. On the space $C^{n}$ we introduce
the norm 
\[
\|F\|_{\infty}:=\sup_{x\in\R^{m}}\|F(x)\|\quad\mbox{and}\quad\|F\|_{C^{n}}:=\|F\|_{\infty}+\sum_{j=1}^{n}\|F^{(n)}\|_{\infty},
\]
for $n\geq1$, where $\|\cdot\|$ denotes the corresponding operator
norms.

The presumably most fundamental way to define \emph{Besov spaces}
is given via the modulus of continuity of a function $f\in L^{p}(\R^{d},\R^{m})$
\begin{equation}
\omega_{p}(f,\delta):=\sup_{0<|h|<\delta}\|f(\cdot)-f(\cdot-h)\|_{L^{p}}\quad\text{for}\quad p,\delta>0.\label{eq:ModCont}
\end{equation}
For $p,q\in[1,\infty]$ and $\alpha\in(0,1)$ Besov spaces are defined
as 
\begin{align*}
B_{p,q}^{\alpha}(\R^{d}):=B_{p,q}^{\alpha}(\R^{d},\R^{m}) & :=\big\{ f\in L^{p}(\R^{d},\R^{m}):\|f\|_{\omega:\alpha,p,q}<\infty\big\}\quad\text{with}\\
\|f\|_{\omega:\alpha,p,q} & :=\|f\|_{L^{p}}+\Big(\int_{\R^{d}}|h|^{-\alpha q}\omega_{p}(f,|h|)^{q}\frac{\mathrm{d}h}{|h|^{d}}\Big)^{1/q}
\end{align*}
and the usual modification if $q=\infty$. If $d=1$ (and no confusion
arises from the dimension $m$) we subsequently abbreviate $L^{p}:=L^{p}(\R,\R^{m})$
and $B_{p,q}^{\alpha}:=B_{p,q}^{\alpha}(\R,\R^{m})$. In $B_{p,q}^{\alpha}(\R^{d})$
the regularity $\alpha$ is measured in the $L^{p}$-norm while $q$
is basically a fine tuning parameter in view of the embedding $B_{p,q_{1}}^{\alpha}(\R^{d})\subset B_{p,q_{2}}^{\beta}(\R^{d})$
for $\beta<\alpha$ and any $q_{1},q_{2}\ge1$. The classical Hölder
spaces and Sobolev spaces are recovered as the special cases $B_{\infty,\infty}^{\alpha}(\R^{d})$
(for non-integer $\alpha$) and $B_{2,2}^{\alpha}(\R^{d})$, respectively.
Alternatively, Besov spaces can be characterized in terms of a Littlewood-Paley
decomposition. Since our analysis mainly relies on this latter characterization,
we describe it subsequently.

\medskip{}

We write $\mathcal{S}(\R^{d}):=\mathcal{S}(\R^{d},\R^{m})$ for the
space of Schwartz functions on $\R^{d}$ and denote its dual by $\mathcal{S}^{\prime}(\R^{d})$,
which is the space of tempered distributions. For a function $f\in L^{1}$
the Fourier transform is defined by

\[
\mathcal{F}f(z):=\int_{\mathbb{R}^{d}}e^{-i\langle z,x\rangle}f(x)\d x
\]
and so the inverse Fourier transform is given by $\mathcal{F}^{-1}f(z):=(2\pi)^{-d}\mathcal{F}f(-z)$.
If $f\in\mathcal{S}^{\prime}(\R^{d})$, then the usual generalization
of the Fourier transform is considered. The Littlewood-Paley theory
is based on localization in the frequency domain. Let $\chi$ and
$\rho$ be non-negative infinitely differentiable radial functions
on $\mathbb{R}^{d}$ such that
\begin{enumerate}
\item there is a ball $\mathcal{B}\subset\R^{d}$ and an annulus $\mathcal{A}\subset\R^{d}$
satisfying $\supp\chi\subset\mathcal{B}$ and $\supp\rho\subset\mathcal{A}$,
\item $\chi(z)+\sum_{j\geq0}\rho(2^{-j}z)=1$ for all $z\in\mathbb{R}^{d}$,
\item \label{enu:suppRho}$\textup{supp}(\chi)\cap\textup{supp}(\rho(2^{-j}\cdot))=\emptyset$
for $j\geq1$ and $\textup{supp}(\rho(2^{-i}\cdot))\cap\textup{supp}(\rho(2^{-j}\cdot))=\emptyset$
for $\vert i-j\vert>1$.
\end{enumerate}
We say a pair $(\chi,\rho)$ with these properties is a\textit{\textcolor{black}{{}
dyadic partition of unity}}\textcolor{red}{{} }and we throughout use
the notation 
\[
\rho_{-1}:=\chi\quad\text{and}\quad\rho_{j}:=\rho(2^{-j}\cdot)\quad\text{for }j\ge0.
\]
For the existence of such a partition we refer to \citep[Prop. 2.10]{Bahouri2011}.
Taking a dyadic partition of unity $(\chi,\rho),$ the \emph{Littlewood-Paley
blocks} are defined as 
\[
\Delta_{-1}f:=\mathcal{F}^{-1}(\rho_{-1}\mathcal{F}f)\quad\text{and}\quad\Delta_{j}f:=\mathcal{F}^{-1}(\rho_{j}\mathcal{F}f)\quad\text{for }j\geq0.
\]
Note that $\Delta_{j}f$ is a smooth function for every $j\geq-1$
and for every $f\in\mathcal{S}^{\prime}(\R^{d})$ we have 
\[
f=\sum_{j\geq-1}\Delta_{j}f:=\lim_{j\to\infty}S_{j}f\quad\text{with}\quad S_{j}f:=\sum_{i\leq j-1}\Delta_{i}f.
\]
For $\alpha\in\mathbb{R}$ and $p,q\in(0,\infty]$ the Besov space
can be characterized in full generality as 
\[
B_{p,q}^{\alpha}(\mathbb{R}^{d},\mathbb{R}^{m})=\bigg\{ f\in\mathcal{S}^{\prime}(\mathbb{R}^{d},\mathbb{R}^{m}):\|f\|_{\alpha,p,q}<\infty\bigg\}\quad\text{with}\quad\|f\|_{\alpha,p,q}:=\Big\|\big(2^{j\alpha}\|\Delta_{j}f\|_{L^{p}}\big)_{j\ge-1}\Big\|_{\ell^{q}}.
\]
According to \citep[Thm. 2.5.12]{triebel2010}, the norms $\|\cdot\|_{\omega:\alpha,p,q}$
and $\|\cdot\|_{\alpha,p,q}$ are equivalent for $p,q\in(0,\infty]$
and $\alpha\in(\frac{d}{\min\{p,1\}}-d,1)$. $B_{p,q}^{\alpha}(\R^{d})$
is a quasi-Banach space and if $p,q\ge1$, it is Banach space, cf.
\citep[Thm. 2.3.3]{triebel2010}. Although the (quasi-)norm $\|\cdot\|_{\alpha,p,q}$
depends on the dyadic partition $(\chi,\rho)$, different dyadic partitions
of unity lead to equivalent norms. 

\medskip{}

We will frequently use the notation $A_{\theta}\lesssim B_{\theta}$,
for a generic parameter $\theta$, meaning that $A_{\theta}\le CB_{\theta}$
for some constant $C>0$ independent of $\theta$. We write $A_{\theta}\sim B_{\theta}$
if $A_{\theta}\lesssim B_{\theta}$ and $B_{\theta}\lesssim A_{\theta}$.
For integers $j_{\theta},k_{\theta}\in\Z$ we write $j_{\theta}\lesssim k_{\theta}$
if there is some $N\in\N$ such that $j_{\theta}\le k_{\theta}+N$,
and $j_{\theta}\sim k_{\theta}$ if $j_{\theta}\lesssim k_{\theta}$
and $k_{\theta}\lesssim j_{\theta}.$

\medskip{}

In view of the RDE~(\ref{eq:rde}) we need to study the product of
two distributions. The standard estimate, cf. \citet[(24) on p. 143]{triebel2010},
\begin{equation}
\|fg\|_{\alpha,p,q}\lesssim\|f\|_{\alpha,\infty,q}\|g\|_{\alpha,p,q}\label{eq:pointwiseMulti}
\end{equation}
applies only for $\alpha>0$ and $p,q\ge1$. However, in the context
of RDEs the regularity $\alpha$ of the involved product will typically
be negative. Given $f\in B_{p_{1,}q_{1}}^{\alpha}(\R^{d})$ and $g\in B_{p_{2},q_{2}}^{\beta}(\R^{d})$,
at least formally we can decompose the product $fg$ in terms of Littlewood-Paley
blocks as 
\[
fg=\sum_{j\geq-1}\sum_{i\geq-1}\Delta_{i}f\Delta_{j}g=T_{f}g+T_{g}f+\pi(f,g),
\]
where 
\begin{equation}
T_{f}g:=\sum_{j\geq-1}S_{j-1}f\Delta_{j}g,\quad\text{and}\quad\pi(f,g):=\sum_{\vert i-j\vert\leq1}\Delta_{i}f\Delta_{j}g.\label{eq:bony decomposition}
\end{equation}
We call $\pi(f,g)$ the \textit{resonant term}. This decomposition
was introduced by \citet{Bony1981} and it comes with the following
estimates: 
\begin{lem}
\label{lem:paraproduct}Let $\alpha,\beta\in\mathbb{R}$ and $p_{1},p_{2},q_{1},q_{2}\in[1,\infty]$
and suppose that 
\[
\frac{1}{p}:=\frac{1}{p_{1}}+\frac{1}{p_{2}}\le1\quad\text{and}\quad\frac{1}{q}:=\frac{1}{q_{1}}+\frac{1}{q_{2}}\le1.
\]

\begin{enumerate}
\item For any $f\in L^{p_{1}}(\R^{d})$ and $g\in B_{p_{2},q}^{\beta}(\R^{d})$
we have
\[
\|T_{f}g\|_{\beta,p,q}\lesssim\|f\|_{L^{p_{1}}}\|g\|_{\beta,p_{2},q}.
\]

\item If $\alpha<0$, then for any $(f,g)\in B_{p_{1},q_{1}}^{\alpha}(\R^{d})\times B_{p_{2},q_{2}}^{\beta}(\R^{d})$
we have
\begin{eqnarray*}
\|T_{f}g\|_{\alpha+\beta,p,q} & \lesssim & \|f\|_{\alpha,p_{1},q_{1}}\|g\|_{\beta,p_{2},q_{2}}.
\end{eqnarray*}

\item If $\alpha+\beta>0$, then for any $(f,g)\in B_{p_{1},q_{1}}^{\alpha}(\R^{d})\times B_{p_{2},q_{2}}^{\beta}(\R^{d})$
we have 
\[
\|\pi(f,g)\|_{\alpha+\beta,p,q}\lesssim\|f\|_{\alpha,p_{1},q_{1}}\|g\|_{\beta,p_{2},q_{2}}.
\]

\end{enumerate}
\end{lem}
\begin{proof}
The last claim is Theorem~2.85 in \citep{Bahouri2011}. For the first
claim and the second one we slightly generalize their Theorem~2.82.
Since $\rho_{j}$ is supported on $2^{j}$ times an annulus and the
Fourier transform of $S_{k-1}f\Delta_{k}g$ is supported on $2^{k}$
times another annulus, it holds $\Delta_{j}T_{f}g=\Delta_{j}\sum_{j\sim k}S_{k-1}f\Delta_{k}g$.
Using that $\Delta_{j}$ is a convolution with $\F^{-1}\rho_{j}=2^{jd}\F^{-1}\rho(2^{j}\cdot)$,
$j\ge0$, Young's inequality yields for any function $h\in L^{p}(\R^{d})$
that $\|\Delta_{j}h\|_{L^{p}}\lesssim\|\F^{-1}\rho\|_{L^{1}}\|h\|_{L^{p}}$.
Together with Hölder's inequality we obtain for any $j\ge-1$
\[
\Big\|\Delta_{j}\big(\sum_{k\ge-1}S_{k-1}f\Delta_{k}g\big)\Big\|_{L^{p}}\lesssim\sum_{k\sim j}\|S_{k-1}f\Delta_{k}g\|_{L^{p}}\le\sum_{k\sim j}\|S_{k-1}f\|_{L^{p_{1}}}\|\Delta_{k}g\|_{L^{p_{2}}}.
\]
Since $\lim_{k\to\infty}\|S_{k-1}f\|_{L^{p_{1}}}=\|f\|_{L^{p_{1}}}$,
assertion (i) follows from
\begin{align*}
\|T_{f}g\|_{\beta,p,q} & \lesssim\Big\|2^{j\beta}\sum_{j\sim k}\|S_{k-1}f\|_{L^{p_{1}}}\|\Delta_{k}g\|_{L^{p_{2}}}\Big\|_{\ell^{q}}\\
 & \lesssim\|f\|_{L^{p_{1}}}\big\|2^{j\beta}\|\Delta_{j}g\|_{L^{p_{2}}}\big\|_{\ell^{q}}=\|f\|_{L^{p_{1}}}\|g\|_{\beta,p_{2},q}.
\end{align*}
For (ii) another application of Hölder's inequality yields
\begin{align*}
\|T_{f}g\|_{\alpha+\beta,p,q}\lesssim & \Big\|2^{j(\alpha+\beta)}\sum_{j\sim k}\|S_{k-1}f\|_{\ell^{p_{1}}}\|\Delta_{k}g\|_{L^{p_{2}}}\Big\|_{\ell^{q}}\\
\lesssim & \big\|2^{j\alpha}\|S_{j-1}f\|_{L^{p_{1}}}\big\|_{\ell^{q_{1}}}\big\|2^{j\beta}\|\Delta_{j}g\|_{L^{p_{2}}}\big\|_{\ell^{q_{2}}}\le\big\|2^{j\alpha}\|S_{j-1}f\|_{L^{p_{1}}}\big\|_{\ell^{q_{1}}}\|g\|_{\beta,p_{2},q_{2}}.
\end{align*}
Finally, we apply Lemma~\ref{lem:bahouri} to conclude that $(2^{j\alpha}\|S_{j-1}f\|_{L^{p_{1}}})_{j}\in\ell^{q_{1}}$
and that
\[
\|(2^{j\alpha}\|S_{j-1}f\|_{L^{p_{1}}})_{j}\|_{\ell^{q_{1}}}\lesssim\|f\|_{\alpha,p_{1},q_{1}.}\qedhere
\]

\end{proof}
We finish this section with two elementary lemmas, which seem to be
non-standard (cf. Lemma~A.4 and A.10 in \citep{Gubinelli2012} for
the Hölder case). To control the norm of an antiderivative with respect
to the function itself will play an import role, naturally restricted
to the case $d=1$. The following lemma provides the counterpart to
the well-known estimate $\|F'\|_{\alpha-1,p,q}\lesssim\|F\|_{\alpha,p,q}$
for any $F\in B_{p,q}^{\alpha}$, cf. \citet[Thm. 2.3.8]{triebel2010}.
For $p<\infty$ the antiderivative will in general have no finite
$L^{p}$-norm such that we have to apply a weighting function to ensure
integrability. 
\begin{lem}
\label{lem:antiderivative} Let $p\in(1,\infty]$ and $\alpha\in(1/p,1)$.
For every $f\in B_{p,q}^{\alpha-1}(\R)$ there exits a unique function
$F\colon\R\to\R^{m}$ such that $F'=f$ and $F(0)=0$. Moreover, it
holds for any fixed $\psi\in C^{1}$ satisfying $C_{\psi}:=\|\psi\|_{C^{1}}+\sum_{j,k\in\{0,1\}}\|t^{j}\psi^{(k)}(t)\|_{L^{p}}<\infty$
that 
\[
\|\psi F\|_{\alpha,p,q}\lesssim C_{\psi}\|f\|_{\alpha-1,p,q}.
\]
In particular, for any smooth $\psi$ with $\supp\psi\subset[-\mathcal{T},\mathcal{T}]$
for some $\mathcal{T}>0$ one has 
\[
\|\psi F\|_{\alpha,p,q}\lesssim(1\vee\mathcal{T}^{2})\|\psi\|_{C^{1}}\|f\|_{\alpha-1,p,q}.
\]
\end{lem}
\begin{proof}
Since differentiating in spatial domain corresponds to multiplication
in Fourier domain, we set 
\[
G(t):=\sum_{j\geq0}\mathcal{F}^{-1}\Big[\frac{1}{iu}\rho_{j}(u)\mathcal{F}f(u)\Big](t)\quad\text{and}\quad H(t):=\int_{0}^{t}\Delta_{-1}f(s)\d s,\quad t\in\mathbb{R}.
\]
Provided 
\begin{equation}
\|\psi G\|_{\alpha,p,q}\lesssim\|\psi\|_{C^{1}}\|G\|_{\alpha,p,q}\lesssim\|\psi\|_{C^{1}}\|f\|_{\alpha-1,p,q},\quad\|\psi H\|_{\alpha,p,q}\le C_{\psi}\|f\|_{\alpha-1,p,q}\label{eq:antiderivative}
\end{equation}
and noting that $B_{p,q}^{\alpha}\subset C(\R)$ for $\alpha>1/p$,
the function $F:=G+H-G(0)$ satisfies $F'=f$ and the asserted norm
estimate. Uniqueness follows because any distribution with zero derivative
is constant. 

It remains to verify (\ref{eq:antiderivative}). Concerning $G$,
we obtain for each Littlewood-Paley block, using $\textup{supp}(\rho_{j})\cap\textup{supp}(\rho_{k})=\emptyset$
for all $j,k\geq-1$ with $\vert k-j\vert>1$, 
\begin{align*}
\Delta_{k}G & =\sum_{j=(k-1)\vee0}^{k+1}\mathcal{F}^{-1}\Big[\frac{1}{iu}\rho_{k}(u)\rho_{j}(u)\mathcal{F}f(u)\Big]=\bigg(\sum_{j=(k-1)\vee0}^{k+1}\mathcal{F}^{-1}\Big[\frac{1}{iu}\rho_{j}(u)\Big]\bigg)*\Delta_{k}f.
\end{align*}
Using twice a substitution, we have for $j\ge0$ 
\[
\Big\|\mathcal{F}^{-1}\Big[\frac{\rho_{j}(u)}{iu}\Big]\Big\|_{L^{1}}=\Big\|\mathcal{F}^{-1}\Big[\frac{\rho(u)}{iu}\Big](2^{j}\cdot)\Big\|_{L^{1}}=2^{-j}\Big\|\mathcal{F}^{-1}\Big[\frac{\rho(u)}{iu}\Big]\Big\|_{L^{1}}.
\]
Hence, Young's inequality yields
\begin{align*}
\|G\|_{\alpha,p,q}= & \Big\|\big(2^{\alpha k}\|\Delta_{k}G\|_{L^{p}}\big)_{k}\Big\|_{\ell^{q}}\lesssim\Big\|\big(2^{(\alpha-1)k}\big\|\mathcal{F}\big[\rho(u)/(iu)\big]\big\|_{L^{1}}\|\Delta_{k}f\|_{L^{p}}\big)\Big\|_{\ell^{q}}\lesssim\|f\|_{\alpha-1,p,q}.
\end{align*}

To show the second part of (\ref{eq:antiderivative}), we use $\|\psi H\|_{\alpha,p,q}\lesssim\|\psi H\|_{1,p,\infty}\lesssim\|\psi H\|_{L^{p}}+\|(\psi H)'\|_{L^{p}}$
due to $\alpha<1$. Hölder's inequality yields for $\bar{p}:=\frac{p}{p-1}$
with the usual modification for $p=\infty$ that
\[
\|\psi H\|_{L^{p}}\le\|\Delta_{-1}f\|_{L^{p}}\big\|\psi(t)t^{1/\bar{p}}\big\|_{L^{p}}\lesssim\|(1\vee t)\psi(t)\|_{L^{p}}\|f\|_{\alpha-1,p,q}
\]
and similarly 
\begin{align*}
\|(\psi H)'\|_{L^{p}}\le\|\psi'H\|_{L^{p}}+\|\psi\Delta_{-1}f\|_{L^{p}} & \lesssim\|\Delta_{-1}f\|_{L^{p}}\big(\big\|\psi'(t)t^{1/\bar{p}}\big\|_{L^{p}}+\|\psi\|_{\infty}\big)\\
 & \lesssim\big(\|\psi\|_{\infty}+\|(1\vee t)\psi'(t)\|_{L^{p}}\big)\|f\|_{\alpha-1,p,q}.\mbox{\tag*{{\qedhere}}}
\end{align*}
 
\end{proof}
For later reference we finally investigate the \emph{scaling operator}
$\Lambda_{\lambda}$, given by $\Lambda_{\lambda}f(\cdot):=f(\lambda\cdot)$
for any $\lambda>0$ and any function $f$, on Besov spaces. 
\begin{lem}
\label{lem:scaling} For $\alpha\neq0$, $p,q\ge1$ and all $f\in B_{p,q}^{\alpha}(\R^{d})$
we have 
\[
\|\Lambda_{\lambda}f\|_{\alpha,p,q}\lesssim(1+\lambda^{\alpha}|\log\lambda|)\lambda^{-d/p}\|f\|_{\alpha,p,q}.
\]
\end{lem}
\begin{proof}
Using $\Lambda_{\kappa}(\mathcal{F}f)=\kappa^{-d}\mathcal{F}[\Lambda_{\kappa^{-1}}f]$
for $\kappa>0$, $f\in B_{p,q}^{\alpha}(\R^{d})$, we first deduce
\begin{align*}
\Delta_{j}(\Lambda_{\lambda}f) & =\lambda^{-d}\mathcal{F}^{-1}[\rho_{j}\Lambda_{\lambda^{-1}}(\mathcal{F}f)]=\mathcal{F}^{-1}[\rho_{j}(\lambda\cdot)\mathcal{F}f](\lambda\cdot)\quad\text{and}\\
\Lambda_{\lambda}(\Delta_{j}f) & =\lambda^{-d}\mathcal{F}^{-1}[\rho_{j}(\lambda^{-1}\cdot)(\mathcal{F}f)(\lambda^{-1}\cdot)]=\mathcal{F}^{-1}[\rho_{j}(\lambda^{-1}\cdot)\mathcal{F}[\Lambda_{\lambda}f]]
\end{align*}
for all $\lambda>0$. For $j\ge0$ the Fourier transform of $\Lambda_{\lambda}(\Delta_{j}f)$
is consequently supported in $\lambda2^{j}\mathcal{A}$, where $\mathcal{A}$
is the annulus containing the support of $\rho$, and we have $\Delta_{k}(\Lambda_{\lambda}\Delta_{j}f)\neq0$
only if $2^{k}\sim\lambda2^{j}$. Together with $\|\Delta_{k}f\|_{L^{p}}\le\|\F^{-1}\rho_{k}\|_{L^{1}}\|f\|_{L^{p}}\lesssim\|f\|_{L^{p}}$
by Young's inequality we obtain 
\[
\|\Delta_{k}\Lambda_{\lambda}f\|_{L^{p}}\le\sum_{j:2^{k}\sim\lambda2^{j}}\|\Delta_{k}\Lambda_{\lambda}(\Delta_{j}f)\|_{L^{p}}\lesssim\lambda^{-d/p}\sum_{j:2^{k}\sim\lambda2^{j}}\|\Delta_{j}f\|_{L^{p}}\quad\text{for }k\geq0.
\]
Applying again Young's inequality to the sequences $a:=(\1_{[-|\log\lambda|,|\log\lambda|]}(k))_{k}$
and $(2^{j\alpha}\|\Delta_{j}f\|_{L^{_{p}}})_{j}$, we infer
\begin{align*}
\Big\|\big(2^{k\alpha}\|\Delta_{k}\Lambda_{\lambda}f\|_{L^{p}}\big)_{k\ge0}\Big\|_{\ell^{q}}\lesssim & \lambda^{-d/p}\Big\|\Big(\sum_{j:2^{k}\sim\lambda2^{j}}\lambda^{\alpha}2^{j\alpha}\|\Delta_{j}f\|_{L^{p}}\Big)_{k\ge0}\Big\|_{\ell^{q}}\\
\lesssim & \lambda^{\alpha-d/p}\|a\|_{\ell^{1}}\|f\|_{\alpha,p,q}\lesssim|\log\lambda|\lambda^{\alpha-d/p}\|f\|_{\alpha,p,q}.
\end{align*}
Finally, we obtain analogously for $k=-1$ that 
\[
\|\Delta_{-1}\Lambda_{\lambda}f\|_{L^{p}}\lesssim\lambda^{-d/p}\sum_{j:\lambda2^{j}\lesssim1}\|\Delta_{j}f\|_{L^{p}}\lesssim\lambda^{-d/p}\|f\|_{\alpha,p,q}\sum_{j:\lambda2^{j}\lesssim1}2^{-\alpha j}\lesssim(1+\lambda^{\alpha})\lambda^{-d/p}\|f\|_{\alpha,p,q}.\qedhere
\]

\end{proof}

\section{Young integration revisited\label{sec:Young} }

In the present section we start to consider the differential equation\,(\ref{eq:rde}),
which was given by 

\[
\dd u(t)=F(u(t))\xi(t),\quad u(0)=u_{0},\quad t\in\R,
\]
where $u_{0}\in\mathbb{R}^{m}$, $u\colon\mathbb{R}\to\mathbb{R}^{m}$
is a continuous function and $F\colon\mathbb{R}^{m}\to\mathcal{L}(\mathbb{R}^{n},\mathbb{R}^{m})$.
Assuming our driving signal $\xi\colon\mathbb{R}\to\mathbb{R}^{n}$
is smooth enough, the differential equation\,(\ref{eq:rde}) is well-defined
and can be equivalently written in its integral form 
\begin{equation}
u(t)=u_{0}+\int_{0}^{t}F(u(s))\xi(s)\d s,\quad t\in[0,\infty),\label{eq:rdeintgral-1}
\end{equation}
and analogously for $t\in(-\infty,0)$. According to \citet{Young1936},
the involved integral can be defined as limit of Riemann sums as long
as the driving signal $\xi$ is the derivative of a path $\theta$
which is of finite $p$-variation for $p<2$. Then, equation\,(\ref{eq:rdeintgral-1})
admits a unique solution on every bounded interval $[-\mathcal{T},\mathcal{T}]\subset\mathbb{R}$
if $F\in C^{2}$ (see modern books as \citep[Theorem 1.28]{Lyons2007}
or \citep[Theorem 1]{Lejay2009}). This result was first proven by
\citet{Lyons1994} using a Picard iteration. The case of a $1/p$-Hölder
continuous driving path $\theta$ was treated by \citet{Ruzmaikina2002}.
Since then it is still of great interest to find new approaches to
(\ref{eq:rdeintgral-1}): \citet{Gubinelli2004} has introduced the
notion of controlled paths, \citet{Davie2007} has shown the convergence
of an Euler scheme, \citet{Hu2007} have used techniques from fractional
calculus and \citet{Lejay2010} has developed a simple approach similar
to \citep{Ruzmaikina2002}.

In this section we recover the analogous results on Besov spaces with
a special focus on the situation when $F$ is a linear functional.
For a discussion of the importance of linear RDEs we refer to \citet{Coutin2014}
and references therein. 

We first note that the function $F(u)$ inherits its regularity from
the regularity of $u$. More precisely, \citep[Thm. 2.87]{Bahouri2011}
shows for $u\in B_{p,q}^{\alpha}$ satisfying $\|u\|_{\infty}<\infty$
and a family of sufficient regular vector fields $F$ with $F(0)=0$
(or $p=\infty$) that 
\begin{equation}
\|F(u)\|_{\alpha,p,q}\lesssim\Big(\sum_{k=1}^{\lceil\alpha\rceil}\sup_{|x|\le\|u\|_{\infty}}\|F^{(k)}(x)\|\Big)\|u\|_{\alpha,p,q}\lesssim\|F\|_{C^{\lceil\alpha\rceil}}\|u\|_{\alpha,p,q},\label{eq:F(u)}
\end{equation}
denoting the smallest integer larger or equal than $\alpha>0$ by
$\lceil\alpha\rceil$ and provided the norms on the right-hand side
are finite. If the product $F(u)\xi$ is regular enough, we can understand
the differential equation~(\ref{eq:rde}) in its integral form~(\ref{eq:rdeintgral-1})
where the integral is given by the antiderivative of the product,
i.e.
\[
\dd\big(\int_{0}^{t}F(u(s))\xi(s)\d s\big)=F(u(t))\xi(t)\quad\text{and}\quad\int_{0}^{0}F(u(s))\xi(s)\d s=0.
\]

In view of Lemma~\ref{lem:antiderivative} the solution $u$ of (\ref{eq:rde})
cannot be expected to be contained in $B_{p,q}^{\alpha}$. Therefore,
we consider instead a localized version of the differential equation.
Alternatively, the solution of the RDE~(\ref{eq:rde}) could be studied
in homogenous or weighted Besov spaces, which can only lead to very
similar results. In order to provide our results in the most commonly
used notion of Besov spaces, we focus on localized equations. We impose
the following standing assumption:
\begin{assumption}
\label{ass:phi}Let $\phi\colon\R\to\R_{+}$ be fixed smooth function
with support $[-2,2]$ and equal to $1$ on $[-1,1]$. Denote $\phi_{\mathcal{T}}(x):=\phi(x/\mathcal{T})$
for $\mathcal{T}>0$. \end{assumption}
\begin{thm}
\label{thm:Young} Let $\mathcal{T}>0$, $\alpha\in(1/2,1]$ and assume
that $\xi\in B_{p,q}^{\alpha-1}$ for $p\in[2,\infty]$ and $q\in[1,\infty]$.
If $F\colon\mathbb{R}^{m}\to\mathcal{L}(\mathbb{R}^{n},\mathbb{R}^{m})$
is a linear mapping, then for every $u_{0}\in\mathbb{R}^{d}$ there
exists a unique global solution $u\in B_{p,q}^{\alpha}$ to the Cauchy
problem 
\begin{equation}
u(t)=\phi_{\mathcal{T}}(t)u_{0}+\phi_{\mathcal{T}}(t)\int_{0}^{t}F(u(s))\xi(s)\d s,\quad t\in\mathbb{R},\label{eq:lip-1}
\end{equation}
with the usual convention for $t<0$. This result extends to nonlinear
$F\in C^{2}$ if $p=\infty$. \end{thm}
\begin{proof}
\emph{Step 1:} First we establish a contraction principle under the
assumption that $\|F'\|_{C^{1}}$ is sufficiently small. Without loss
of generality we may assume $u_{0}=0$. Following a fixed point argumentation,
we consider the solution map
\begin{align*}
\Phi\colon B_{p,q}^{\alpha}\to B_{p,q}^{\alpha},\quad u & \mapsto\tilde{u}:=\phi_{\mathcal{T}}\int_{0}^{\cdot}F(u(s))\xi(s)\d s,\quad t\in\mathbb{R}.
\end{align*}
 In order to verify that $\Phi$ is indeed well-defined, we use Lemma~\ref{lem:antiderivative}
to observe
\[
\|\phi_{\mathcal{T}}F\|_{\alpha,p,q}\lesssim(1\vee\mathcal{T}^{2})(1\vee\mathcal{T}^{-1})\|\phi\|_{C^{1}}\|f\|_{\alpha-1,p,q}\lesssim C_{\mathcal{T},\phi}\|f\|_{\alpha-1,p,q},
\]
where $C_{\mathcal{T},\phi}:=(\mathcal{T}^{-1}\vee\mathcal{T}^{2})\|\phi\|_{C^{1}}$,
for any given $f\in B_{p,q}^{\alpha-1}$ with $\dd F=f$ and $F(0)=0$.
We thus have

\begin{align*}
\|\Phi(u)\|_{\alpha,p,q} & =\Big\|\phi_{\mathcal{T}}\Big(\int_{0}^{\cdot}F(u(s))\xi(s)\d s\Big)\Big\|_{\alpha,p,q}\lesssim C_{\mathcal{T},\phi}\|F(u)\xi\|_{\alpha-1,p,q}.
\end{align*}
Applying Bony's decomposition, the Besov embedding $B_{p/2,q}^{2\alpha-1}\subset B_{p,q}^{\alpha-1}$
(cf. \citep[Thm. 2.7.1]{triebel2010}) for $p>1/\alpha$ and Lemma~\ref{lem:paraproduct},
we obtain
\begin{align*}
\|\Phi(u)\|_{\alpha,p,q} & \lesssim C_{\mathcal{T},\phi}\big(\|T_{F(u)}\xi\|_{\alpha-1,p,q}+\|\pi(F(u),\xi)\|_{2\alpha-1,p/2,q}+\|T_{\xi}(F(u))\|_{\alpha-1,p,q}\big)\\
 & \lesssim C_{\mathcal{T},\phi}\big(\|F(u)\|_{\infty}\|\xi\|_{\alpha-1,p,q}+\|F(u)\|_{\alpha,p,2q}\|\xi\|_{\alpha-1,p,2q}+\|\xi\|_{\alpha-1,p,q}\|F(u)\|_{0,\infty,\infty}\big).
\end{align*}
Using the embeddings $B_{p,q}^{\alpha}\subset L^{\infty}$ and $B_{p,q}^{\alpha}\subset B_{\infty,\infty}^{0}$
for $\alpha>1/p$ and (\ref{eq:F(u)}), we deduce that 
\begin{equation}
\|\Phi(u)\|_{\alpha,p,q}\lesssim C_{\mathcal{T},\phi}\|F'\|_{\infty}\|\xi\|_{\alpha-1,p,q}\|u\|_{\alpha,p,q}.\label{eq:bound young}
\end{equation}
To apply Banach's fixed point theorem, it remains to show that $\Phi$
is a contraction. For $u,\tilde{u}\in B_{p,q}^{\alpha}$ Lemma~\ref{lem:antiderivative}
again yields 
\begin{align*}
\|\Phi(u)-\Phi(\tilde{u})\|_{\alpha,p,q} & \lesssim C_{\mathcal{T},\phi}\|\big(F(u)-F(\tilde{u})\big)\xi\|_{\alpha-1,p,q}\\
 & \lesssim C_{\mathcal{T},\phi}\int_{0}^{1}\|F^{\prime}(u+t(u-\tilde{u}))(u-\tilde{u})\xi\|_{\alpha-1,p,q}\d t.
\end{align*}
Denoting by $v_{t}:=F^{\prime}(u+t(u-\tilde{u}))(u-\tilde{u})$, we
conclude as above
\begin{align*}
\|\Phi(u)-\Phi(\tilde{u})\|_{\alpha,p,q} & \lesssim C_{\mathcal{T},\phi}\int_{0}^{1}\big(\|T_{v_{t}}\xi\|_{\alpha-1,p,q}+\|\pi(v_{t},\xi)\|_{2\alpha-1,p/2,q/2}+\|T_{\xi}v_{t}\|_{\alpha-1,p,q}\big)\d t\\
 & \lesssim C_{\mathcal{T},\phi}\int_{0}^{1}\big(\|v_{t}\|_{\alpha,p,q}\|\xi\|_{\alpha-1,p,q}\big)\d t.
\end{align*}
By the standard estimate (\ref{eq:pointwiseMulti}), we obtain 
\begin{equation}
\|\Phi(u)-\Phi(\tilde{u})\|_{\alpha,p,q}\lesssim C_{\mathcal{T},\phi}\Big(\int_{0}^{1}\|F^{\prime}(u+t(\tilde{u}-u))\|_{\alpha,\infty,q}\d t\Big)\|\xi\|_{\alpha-1,p,q}\|u-\tilde{u}\|_{\alpha,p,q}.\label{eq:youngContraction}
\end{equation}
Hence, if $F$ is linear and $\|F^{\prime}\|_{\infty}$ is small enough,
$\Phi$ is a contraction. Provided $p=\infty$ and $F\in C^{2}$,
it suffices if $\|F^{\prime}\|_{C^{1}}$ is sufficiently small: 
\begin{equation}
\|\Phi(u)-\Phi(\tilde{u})\|_{\alpha,p,q}\lesssim C_{\mathcal{T},\phi}\|F^{\prime}\|_{C^{1}}\big(\|u\|_{\alpha,\infty,q}+\|\tilde{u}\|_{\alpha,\infty,q}\big)\|\xi\|_{\alpha-1,\infty,q}\|u-\tilde{u}\|_{\alpha,\infty,q}.\label{eq:youngContraction2}
\end{equation}

\emph{Step 2:} In order to ensure that $\|F'\|_{C^{1}}$ is small
enough, we scale $\xi$ as follows: For some fixed $\epsilon\in(0,\alpha-1/p)$
and for some $\lambda\in(0,1)$ to be chosen later we set 
\begin{equation}
\xi^{\lambda}:=\lambda^{1-\alpha+1/p+\epsilon}\Lambda_{\lambda}\xi,\label{eq:scaling young}
\end{equation}
where we recall the scaling operator $\Lambda_{\lambda}f=f(\lambda\cdot)$
for $f\in\mathcal{S}^{\prime}$. Lemma~\ref{lem:scaling} yields
\begin{align*}
\|\xi^{\lambda}\|_{\alpha-1,p,q} & =\lambda^{1-\alpha+1/p+\epsilon}\|\Lambda_{\lambda}\xi\|_{\alpha-1,p,q}\lesssim(\lambda^{\epsilon}|\log\lambda|+\lambda^{1-\alpha+\epsilon})\|\xi\|_{\alpha-1,p,q}\le\|\xi\|_{\alpha-1,p,q}.
\end{align*}
For $\lambda>0$ sufficiently small Step~1 provides a unique global
solution $u^{\lambda}\in B_{p,q}^{\alpha}$ to the (localized) differential
equation 
\begin{equation}
u^{\lambda}(t)=\phi_{\mathcal{T}}(t)u_{0}+\phi_{\mathcal{T}}(t)\int_{0}^{t}\lambda^{\alpha-1/p-\epsilon}F(u^{\lambda}(s))\xi^{\lambda}(s)\d s,\label{eq:rde scaled young}
\end{equation}
for all $u_{0}\in\mathbb{R}$. Setting now $u:=\Lambda_{\lambda^{-1}}u^{\lambda}$,
we have constructed a unique solution to 
\[
u(t)=\Lambda_{\lambda^{-1}}u^{\lambda}(t)=\phi_{\lambda\mathcal{T}}(t)u_{0}+\phi_{\lambda\mathcal{T}}(t)\int_{0}^{t}F(u(s))\xi(s)\d s,
\]
which coincides with (\ref{eq:lip-1}) on $[-\lambda\mathcal{T},\lambda\mathcal{T}]$.

\emph{Step 3:} Since the choice of $\lambda$ does not depend on $u_{0}$,
we can iteratively apply Step~2 on intervals of length $2\lambda\mathcal{T}$
to construct a unique global solution $u\in B_{p,q}^{\alpha}$ to
equation~(\ref{eq:lip-1}).
\end{proof}
In this simple setting it turns out that the Itô map $S$ defined
by
\begin{align}
S\colon\mathbb{R}^{d}\times B_{p,q}^{\alpha-1} & \to B_{p,q}^{\alpha}\quad\mbox{via}\quad(u_{0},\xi)\mapsto u,\label{eq:itomap}
\end{align}
where $u$ denotes the solution of the (localized) Cauchy problem~(\ref{eq:lip-1}),
is a locally Lipschitz continuous map with respect to the Besov norm. 
\begin{thm}
\label{thm:lipschitz young} Let $\alpha\in(1/2,1]$, $q\in[1,\infty]$
and $F\colon\mathbb{R}^{m}\to\mathcal{L}(\mathbb{R}^{n},\mathbb{R}^{m})$.
If either $F$ is a linear mapping and $p\in[2,\infty]$ or $F\in C^{2}$
and $p=\infty$, then the Itô map $S$ given by (\ref{eq:itomap})
is locally Lipschitz continuous. \end{thm}
\begin{proof}
Let $u_{0}^{i}\in\mathbb{R}^{d}$, $\xi^{i}\in B_{p,q}^{\alpha-1}$
be such that $\|\xi^{i}\|_{\alpha-1,p,q}\leq R$ and $|u_{0}^{i}|\leq R$
for some $R>0$ and denote by $u^{i}$ the unique solution to corresponding
Cauchy problems~(\ref{eq:lip-1}) for $i=1,2$, which exists thanks
to Theorem~\ref{thm:Young}. In order to avoid repetition, we just
consider a linear mapping $F$. The non-linear case works analogously. 

\emph{Step 1:} Suppose that $\|F'\|_{\infty}$ is sufficiently small.
Recalling $C_{\mathcal{T},\phi}=(\mathcal{T}^{-1}\vee\mathcal{T}^{2})\|\phi\|_{C^{1}}$,
we deduce similarly to (\ref{eq:bound young}) that 
\[
\|u^{i}\|_{\alpha,p,q}\lesssim\|\phi_{\mathcal{T}}\|_{\alpha,p,q}|u_{0}^{i}|+C_{\mathcal{T},\phi}\|F'\|_{\infty}\|\xi^{i}\|_{\alpha-1,p,q}\|u^{i}\|_{\alpha,p,q},
\]
which, provided $\|F'\|_{\infty}$ is small enough, depending only
on $R$, $\phi$ and $\mathcal{T}$, leads to 
\[
\|u^{i}\|_{\alpha,p,q}\lesssim\|\phi_{\mathcal{T}}\|_{\alpha,p,q}R,\quad\text{for}\quad i=1,2.
\]
For the difference $u^{1}-u^{2}$ we have
\begin{align*}
\|u^{1}-u^{2}\|_{\alpha,p,q} & \leq\|\phi_{\mathcal{T}}(u_{0}^{1}-u_{0}^{2})\|_{\alpha,p,q}+\Big\|\phi_{\mathcal{T}}\Big(\int_{0}^{\cdot}F(u^{1}(s))\xi^{1}(s)\d s-\int_{0}^{\cdot}F(u^{2}(s))\xi^{2}(s)\d s\Big)\Big\|_{\alpha,p,q}\\
 & \lesssim\|\phi_{\mathcal{T}}\|_{\alpha,p,q}|u_{0}^{1}-u_{0}^{2}|+\Big\|\phi_{\mathcal{T}}\int_{0}^{\cdot}\big(F(u^{1}(s))-F(u^{2}(s))\big)\xi^{1}(s)\d s\Big\|_{\alpha,p,q}\\
 & \qquad+C_{\mathcal{T},\phi}\|F(u^{2})(\xi^{1}-\xi^{2})\|_{\alpha-1,p,q}.
\end{align*}
The second term can be estimated as in (\ref{eq:youngContraction})
and for the last one Bony's decomposition, Lemma~\ref{lem:paraproduct}
and (\ref{eq:F(u)}) yield
\[
\|F(u^{2})(\xi^{1}-\xi^{2})\|_{\alpha-1,p,q}\lesssim\|F(u^{2})\|_{\alpha,p,q}\|\xi^{1}-\xi^{2}\|_{\alpha-1,p,q}\le\|F'\|_{\infty}\|u^{2}\|_{\alpha,p,q}\|\xi^{1}-\xi^{2}\|_{\alpha-1,p,q}.
\]
Therefore, we can combine the above estimates to 
\begin{align*}
\|u^{1}-u^{2}\|_{\alpha,p,q}\lesssim & C_{\mathcal{T},\phi}\Big(|u_{0}^{1}-u_{0}^{2}|+\|\phi_{\mathcal{T}}\|_{\alpha,p,q}\|F'\|_{\infty}R\|\xi^{1}-\xi^{2}\|_{\alpha-1,p,q}\\
 & +\Big(\int_{0}^{1}\|(F'(u^{1}+t(u^{2}-u^{1}))\|_{\alpha-1,\infty,q}\d t\Big)R\|u^{1}-u^{2}\|_{\alpha,p,q}\Big).
\end{align*}
If $F$ is linear with sufficiently small $\|F\|_{C^{1}}$, we obtain
the desired estimate by rearranging:
\[
\|u^{1}-u^{2}\|_{\alpha,p,q}\lesssim C_{\mathcal{T},\phi}\big(|u_{0}^{1}-u_{0}^{2}|+\|\phi_{\mathcal{T}}\|_{\alpha,p,q}\|F\|_{C^{1}}R\|\xi^{1}-\xi^{2}\|_{\alpha-1,p,q}\big).
\]

\emph{Step 2:} The assumption on $\|F'\|_{\infty}$ can be translated
to an assumption on $\mathcal{T}$ using the same scaling argument
as in Step 2 in the proof of Theorem~\ref{thm:Young}. More precisely,
we define $\xi^{\lambda,1}$ and $\xi^{\lambda,2}$ for $\lambda>0$
as in (\ref{eq:scaling young}) and note $\|\xi^{\lambda,i}\|_{\alpha,p,q}\lesssim R$
for $i=1,2$. Therefore, for sufficiently small $\lambda$ there exists
a unique solution $u^{\lambda,i}$ to (\ref{eq:rde scaled young})
for $i=1,2$. Setting again $u^{i}:=\Lambda_{\lambda^{-1}}u^{\lambda}$
and applying twice Lemma~\ref{lem:scaling} together with Step 1
gives 
\begin{align*}
\|u^{1}-u^{2}\|_{\alpha,p,q} & \lesssim\big(1+\lambda^{-\alpha}|\log\lambda^{-1}|\big)\lambda^{1/p}\|u^{\lambda,1}-u^{\lambda,2}\|_{\alpha,p,q}\\
 & \lesssim C_{\mathcal{T},\phi}\big(1+\lambda^{-\alpha}|\log\lambda^{-1}|\big)\lambda^{1/p}\big(|u_{0}^{1}-u_{0}^{2}|+\|\phi_{\mathcal{T}}\|_{\alpha,p,q}\|F'\|_{\infty}R\|\xi^{\lambda,1}-\xi^{\lambda,2}\|_{\alpha-1,p,q}\big)\\
 & \lesssim C_{\mathcal{T},\phi}\big(1+\lambda^{-\alpha}|\log\lambda^{-1}|\big)\lambda^{1/p}\big(|u_{0}^{1}-u_{0}^{2}|+\|\phi_{\mathcal{T}}\|_{\alpha,p,q}\|F'\|_{\infty}R\|\xi^{1}-\xi^{2}\|_{\alpha-1,p,q}\big).
\end{align*}
In conclusion, the Itô map is locally Lipschitz continuous given $\mathcal{T}>0$
is sufficiently small because $u^{i}$ is a solution to 
\[
u^{i}(t)=\phi_{\lambda\mathcal{T}}(t)u_{0}^{i}+\phi_{\lambda\mathcal{T}}(t)\int_{0}^{t}F(u^{i}(s))\xi^{i}(s)\d s,\quad i=1,2.
\]

\emph{Step 3:} The local Lipschitz continuity for arbitrary $\mathcal{T}$
follows by a pasting argument. For this purpose choose a partition
of unity $(\mu_{j})_{j\in\mathbb{Z}}\subset C^{\infty}$ satisfying
$\mu_{j}(t_{j}+\epsilon)=1$, $\epsilon\in[-\frac{1}{2}\lambda\mathcal{T},\frac{1}{2}\lambda\mathcal{T}]$,
for anchor points $t_{j}\in\mathbb{R}$ with $t_{0}=0$ and $|t_{j}-t_{j-1}|\le\lambda\mathcal{T}/2$
and fulfilling 
\[
|\supp\mu_{j}|:=\sup\{|x-y|\ :\ x,y\in\supp\mu_{j}\}\leq\lambda\mathcal{T}\quad\text{and}\quad\sum_{j\in\mathbb{Z}}\mu_{j}(x)=1\quad\text{for all }x\in\mathbb{R}.
\]
Since the $u^{i}$ for $i=1,2$ have compact support, there is some
$N\in\N$ such that one has, using (\ref{eq:pointwiseMulti}), 

\begin{align*}
\|u^{1}-u^{2}\|_{\alpha,p,q} & \leq\sum_{j=-N}^{N}\|\mu_{j}\big(u^{1}-u^{2}\big)\|_{\alpha,p,q}\lesssim\sum_{j=-N}^{N}\|\mu_{j}\|_{C^{1}}\|u_{j}^{1}-u_{j}^{2}\|_{\alpha,p,q},
\end{align*}
where $u_{j}^{i}$ is the unique solution to 
\[
u_{j}^{i}(t)=\phi_{\lambda\mathcal{T}}(t-t_{j})u_{t_{j}}^{i}+\phi_{\lambda\mathcal{T}}(t-t_{j})\int_{t_{j}}^{t}F(u_{j}^{i}(s))\xi^{i}(s)\d s
\]
with initial condition $u_{t_{j}}^{i}:=u^{i}(t_{j})$ for $i=1,2$.
Noting that $|u_{t_{j}}^{1}-u_{t_{j}}^{2}|\lesssim\|u_{j-1}^{1}-u_{j-1}^{2}\|_{\alpha,p,q}$
for $j\ge1$ and similarly for negative $j$, Step~2 yields 
\[
\|u^{1}-u^{2}\|_{\alpha,p,q}\lesssim C_{\mathcal{T},\phi}\big(|u_{0}^{1}-u_{0}^{2}|+\|\phi_{\mathcal{T}}\|_{\alpha,p,q}\|F'\|_{\infty}R\|\xi^{1}-\xi^{2}\|_{\alpha-1,p,q}\big).\qedhere
\]

\end{proof}
To extend these results to nonlinear functions $F$ for $p<\infty$
and to less regular driving signals $\xi$, more precisely $\xi\in B_{p,q}^{\alpha-1}$
for $\alpha\in(1/3,1/2)$, is the aim of the following two sections.

\section{Linearization and commutator estimate\label{sec:LinCom} }

In order to deal with more irregular driving signals $\xi\in B_{p,q}^{\alpha-1}$,
we shall apply Bony's decomposition to rigorously define the product
$F(u)\xi$, which appears in the RDE~(\ref{eq:rde}). Let us first
formally decompose $F(u)\xi$ and analyze the Besov regularity of
the different terms as follows 
\begin{align}
F(u)\xi & =\underbrace{T_{F(u)}\xi}_{\in B_{p,q}^{\alpha-1}}+\underbrace{\pi(F(u),\xi)}_{\in B_{p/2,q/2}^{2\alpha-1}\text{ if }2\alpha-1>0}+\underbrace{T_{\xi}(F(u))}_{\in B_{p/2,q/2}^{2\alpha-1}}.\label{eq:decomp}
\end{align}
The first term $T_{F(u)}\xi$ is in $B_{p,q}^{\alpha-1}$ due to Lemma~\ref{lem:paraproduct}
and the boundedness of $F$. The regularity of the third term $T_{\xi}F(u)\in B_{p/2,q/2}^{2\alpha-1}$
for $\alpha<1$ can also be deduced from Lemma~\ref{lem:paraproduct}
since naturally the solution $u$ has regularity $B_{p,q}^{\alpha}$
and thus $F(u)\in B_{p,q}^{\alpha}$ by (\ref{eq:F(u)}). The regularity
estimate of the resonant term can be applied only if $2\alpha-1>0$.
This is the main reason, why it was possible for $\alpha\in(1/2,1]$
to show the existence of a solution to the (localized) RDE~(\ref{eq:rde})
in Section~\ref{sec:Young} without taking any additional information
about $\xi$ into account. However, this high Besov regularity assumption
on $\xi$ is violated in most of the basic examples from stochastic
analysis as for instance for stochastic differential equations driven
by Brownian motion or martingales. The aim of this section is to reduce
the resonant term $\pi(F(u),\xi)$ to $\pi(u,\xi)$:
\begin{prop}
\label{prop:reduction} Let $\alpha\in(\frac{1}{3},\frac{1}{2})$,
$p\in[3,\infty]$ and $F\in C^{2+\gamma}(\R)$ for some $\gamma\in(0,1]$
satisfying $F(0)=0$. Then there is a map $\Pi_{F}\colon B_{p,\infty}^{\alpha}(\R)\times B_{p,\infty}^{\alpha-1}(\R)\to B_{p/3,\infty}^{3\alpha-1}(\R)$
such that for any $u\in B_{p,\infty}^{\alpha}(\R)$ and $\xi\in B_{p,\infty}^{\alpha-1}(\R)$
we have 
\begin{equation}
\pi(F(u),\xi)=F'(u)\pi(u,\xi)+\Pi_{F}(u,\xi)\label{eq:RDEresonant}
\end{equation}
with
\begin{equation}
\|\Pi_{F}(u,\xi)\|_{3\alpha-1,p/3,\infty}\lesssim\|F\|_{C^{2}}\|u\|_{\alpha,p,\infty}^{2}\|\xi\|_{\alpha-1,p,\infty}.\label{eq:boundResonant}
\end{equation}
Moreover, $\Pi_{F}$ is locally Hölder continuous satisfying for any
$u^{1},u^{2}\in B_{p,q}^{\alpha}(\R)$ and $\xi^{1},\xi^{2}\in B_{p,q}^{\alpha-1}(\R)$
\begin{align*}
 & \|\Pi_{F}(u^{1},\xi^{1})-\Pi_{F}(u^{2},\xi^{2})\|_{3\alpha-1,p/3,\infty}\\
 & \quad\lesssim\|F\|_{C^{2+\gamma}}C(u^{1},u^{2},\xi^{1},\xi^{2})\Big(\|u^{1}-u^{2}\|_{\infty}^{\gamma}+\|u^{1}-u^{2}\|_{\alpha,p,\infty}+\|\xi^{1}-\xi^{2}\|_{\alpha-1,p,\infty}\Big)
\end{align*}
where $C(u^{1},u^{2},\xi^{1},\xi^{2}):=\|u^{1}\|_{\alpha,p,\infty}^{2}\wedge\|u^{2}\|_{\alpha,p,\infty}^{2}+\big(\|u^{1}\|_{\alpha,p,\infty}+\|u^{2}\|_{\alpha,p,\infty}\big)\big(1+\|\xi^{1}\|_{\alpha-1,p,\infty}\wedge\|\xi^{2}\|_{\alpha-1,p,\infty}\big).$
\end{prop}
As we will see in the next section, it suffices to consider only $q=\infty$
in Proposition~\ref{prop:reduction}. Taking into account the embedding
$B_{p,q}^{\alpha}\subset B_{p,\infty}^{\alpha}$ for any $q\in[1,\infty]$,
this case corresponds to the weakest Besov norm for fixed $\alpha$
and $p$. 

In order to prove this proposition, we need the subsequent lemmas.
As the first step, we show the following paralinearization result,
which is a slight generalization of Theorem 2.92 in \citep{Bahouri2011}.
Our proof is inspired by \citep[Lem. 2.6]{Gubinelli2012} and relies
on the characterization of Besov spaces via the modulus of continuity.
We obtain that the composition $F(u)$ can be written as a paraproduct
of $F'(u)$ and $u$ up to some more regular remainder. 
\begin{lem}
\label{lem:linearization} Let $0<\beta\le\alpha<1$ and $F\in C^{1+\beta/\alpha}(\R^{m})$.
Let $p\ge\beta/\alpha+1$ and define $p':=\alpha p/(\alpha+\beta)$.
Then for any $g\in B_{p,\infty}^{\alpha}(\R^{d})\cap L^{\infty}(\R^{d})$
there is some $R_{F}(g)\in B_{p',\infty}^{\alpha+\beta}(\R^{d})$
satisfying 
\[
F(g)-F(0)=T_{F'(g)}g+R_{F}(g)\quad\text{and}\quad\|R_{F}(g)\|_{\alpha+\beta,p',\infty}\lesssim\|F\|_{C^{1+\beta/\alpha}}^{2-\beta/\alpha}\|g\|_{\alpha,p,\infty}^{1+\beta/\alpha}.
\]
Moreover, if $F\in C^{2+\gamma}$ for some $\gamma\in(0,1]$ and if
$p>2\vee1/\alpha$ then the map
\[
R_{F}\colon B_{p,\infty}^{\alpha}(\R^{d})\cap L^{\infty}(\R^{d})\to B_{p/2,\infty}^{2\alpha}(\R^{d})
\]
is locally Hölder continuous with
\begin{align*}
 & \|R_{F}(g)-R_{F}(h)\|_{2\alpha,p/2,\infty}\\
 & \quad\lesssim\|F\|_{C^{2+\gamma}}\Big(\|g\|_{\alpha,p,\infty}^{2}\wedge\|h\|_{\alpha,p,\infty}^{2}+\|g\|_{\alpha,p,\infty}+\|h\|_{\alpha,p,\infty}\Big)\big(\|g-h\|_{\infty}^{\gamma}+\|g-h\|_{\alpha,p,\infty}\big).
\end{align*}
\end{lem}
\begin{proof}
The remainder $R_{F}(g)$ is given by
\[
R_{F}(g)=F(g)-F(0)-T_{F'(g)}g=\sum_{j\ge-1}F_{j}\quad\text{with}\quad F_{j}:=\Delta_{j}(F(g)-F(0))-S_{j-1}(F'(g))\Delta_{j}g.
\]
For $j\le0$ Young's inequality and the Lipschitz continuity of $F$
yield 
\begin{align*}
\|F_{j}\|_{L^{p}}= & \|\Delta_{j}(F(g)-F(0))\|_{L^{p}}\le\|\F^{-1}\rho_{j}\|_{L^{1}}\|F(g)-F(0)\|_{L^{p}}\lesssim\|F\|_{C^{1}}\|g\|_{L^{p}}
\end{align*}
and we have $\|F_{j}\|_{L^{p'}}\le\|F_{j}\|_{L^{p}}\|F_{j}\|_{L^{\alpha p/\beta}}\lesssim\|F_{j}\|_{\infty}^{1-\beta/\alpha}\|F_{j}\|_{L^{p}}^{1+\beta/\alpha}.$

For $j>0$ we have $\Delta_{j}F(0)=0$ and the Fourier transform of
$F_{j}$ is supported in $2^{j}$ times some annulus. Defining the
kernel functions $K_{j}:=\F^{-1}\rho_{j}$ and $K_{<j-1}:=\sum_{k<j-1}K_{k}$
and using that $\int K_{j}(x)\d x=\rho_{j}(0)=0$, the blocks $F_{j}$
can be written as convolution
\begin{align}
F_{j}(x) & =\int_{\mathbb{R}^{2}}K_{j}(x-y)K_{<j-1}(x-z)\big(F(g(y))-F'(g(z))g(y)\big)\d y\d z\nonumber \\
= & \int_{\mathbb{R}^{2}}K_{j}(x-y)K_{<j-1}(x-z)\Big(F(g(y))-F(g(z))-F'(g(z))\big(g(y)-g(z)\big)\Big)\d y\d z\nonumber \\
= & \int_{\mathbb{R}^{2}}K_{j}(x-y)K_{<j-1}(x-z)\Big(\big(F'\big(g(z)+\xi_{yz}(g(y)-g(z))\big)-F'(g(z))\big)\big(g(y)-g(z)\big)\Big)\d y\d z,\label{eq:Fj}
\end{align}
where we used in the in last equality the mean value theorem for intermediate
points $\xi_{yz}\in[0,1]$. By the Hölder continuity of $F'$ the
above display can be estimated by
\begin{align*}
|F_{j}(x)|\le & \|F\|_{C^{1+\beta/\alpha}}\int_{\mathbb{R}^{2}}|K_{j}(x-y)K_{<j-1}(x-z)|\xi_{yz}^{\beta/\alpha}|g(y)-g(z)|^{\beta/\alpha+1}\d y\d z\\
\le & \|F\|_{C^{1+\beta/\alpha}}\int_{\mathbb{R}^{2}}|K_{j}(y)K_{<j-1}(z)||g(x-y)-g(x-z)|^{\beta/\alpha+1}\d y\d z.
\end{align*}
Now we can estimate the $L^{p'}$-norm of the integral by the integral
of the $L^{p'}$-norm, which yields
\begin{align*}
\|F_{j}\|_{L^{p'}}\le & \|F\|_{C^{1+\beta/\alpha}}\int_{\mathbb{R}^{2}}|K_{j}(y)K_{<j-1}(z)|\,\big\||g(x-(y-z))-g(x)|^{\beta/\alpha+1}\big\|_{L^{p'}}\d y\d z\\
\le & \|F\|_{C^{1+\beta/\alpha}}\int_{\mathbb{R}^{2}}|K_{j}(y)K_{<j-1}(z)|\sup_{|h|\le|y-z|}\big\||g(x)-g(x-h)|^{\beta/\alpha+1}\big\|_{L^{p'}}\d y\d z\\
= & \|F\|_{C^{1+\beta/\alpha}}\int_{\mathbb{R}^{2}}|K_{j}(y)K_{<j-1}(z)|\sup_{|h|\le|y-z|}\big\| g(x)-g(x-h)\big\|_{L^{p}}^{1+\beta/\alpha}\d y\d z.
\end{align*}
Recalling the modulus of continuity from (\ref{eq:ModCont}) and the
corresponding representation of the Besov norm, we obtain with Hölder's
inequality for any $q\in[1,\infty]$ with $q^{*}=q/(q-1)$
\begin{align}
\|F_{j}\|_{L^{p'}}\le & \|F\|_{C^{1+\beta/\alpha}}\int_{\mathbb{R}^{2}}|K_{j}(y)K_{<j-1}(y-h)|\omega_{p}(g,|h|)^{1+\beta/\alpha}\d y\d h\nonumber \\
\lesssim & \|F\|_{C^{1+\beta/\alpha}}\|g\|_{\alpha,p,(1+\beta/\alpha)q}^{1+\beta/\alpha}\Big(\int\Big(|h|^{\alpha+\beta+d/q}\int|K_{j}(y)K_{<j-1}(y-h)|\d y\Big)^{q^{*}}\d h\Big)^{1/q^{*}}\label{eq:FjLp}
\end{align}
(with $d/q:=0$ for $q=\infty$ and the usual modification for $q^{*}=\infty$).
Abbreviating $\delta:=\alpha+\beta+d/q$, the last integral can be
written as
\begin{align*}
 & \big\||h|^{\delta}\big(|K_{j}|\ast|K_{<j-1}(-\cdot)|\big)(h)\big\|_{L^{q^{*}}}\\
 & \quad\le\big\|\big(|h|^{\delta}|K_{j}|(h)\big)\ast|K_{<j-1}(-\cdot)|\big\|_{L^{q^{*}}}+\big\||K_{j}|\ast\big(|h|^{\delta}|K_{<j-1}(-h)|\big)\big\|_{L^{q^{*}}}\\
 & \quad\le\||h|^{\delta}|K_{j}|(h)\|_{L^{q^{*}}}\|K_{<j-1}\|_{L^{1}}+\|K_{j}\|_{L^{1}}\||h|^{\delta}|K_{<j-1}(-h)|\|_{L^{q^{*}}},
\end{align*}
where we apply Young's inequality in the last estimate. Due to $K_{j}=\F^{-1}\rho_{j}=(2\pi)^{-d}2^{jd}\F\rho(2^{j}\cdot)$,
we see easily that $\||h|^{\delta}|K_{j}|(h)\|_{L^{q^{*}}}\lesssim2^{-j(\alpha+\beta)}$
and $\|K_{j}\|_{L^{1}}\lesssim1$. To bound similarly the norms of
$K_{<j-1}$ note that $\F K_{<j-1}$ is uniformly bounded and supported
on a ball with radius of order $2^{j}$. We conclude
\[
\|F_{j}\|_{L^{p'}}\lesssim2^{-j(\alpha+\beta)}\|F\|_{C^{1+\beta/\alpha}}\|g\|_{\alpha,p,(1+\beta/\alpha)q}^{1+\beta/\alpha}.
\]
The claimed bound $\|R_{F}(g)\|_{\alpha+\beta,p,\infty}$ thus follows
from Lemma~\ref{lem:BesovAnnulus} and choosing $q=\infty$.

To show the Hölder continuity, we will apply similar arguments. For
convenience we define $\Delta f(y,z):=f(y)-f(z)$ for any function
$f$. Using the additional regularity of $F$, we obtain from (\ref{eq:Fj})
that 
\begin{align*}
F_{j}(x)= & \int_{\R^{2}}K_{j}(x-y)K_{<j-1}(x-z)\int_{0}^{1}\Big(F'\big(g(z)+s\Delta g(y,z)\big)-F'(g(z))\Big)\Delta g(y,z)\d s\d y\d z\\
= & \int_{\R^{2}}K_{j}(x-y)K_{<j-1}(x-z)\int_{0}^{1}\int_{0}^{1}sF''\big(g(z)+rs\Delta g(y,z)\big)\Delta g(y,z)^{2}\d r\d s\d y\d z.
\end{align*}
Hence, we can write
\[
R_{F}(g)-R_{F}(h)=\sum_{j\ge-1}G_{j}
\]
with
\begin{align*}
G_{j}(x)= & \int_{\R^{2}}\int_{0}^{1}\int_{0}^{1}K_{j}(x-y)K_{<j-1}(x-z)s\Big(F''\big(g(z)+rs\Delta g(y,z)\big)\Delta g(y,z)^{2}\\
 & \qquad\hspace{40bp}-F''\big(h(z)+rs\Delta h(y,z)\big)\Delta h(y,z)^{2}\Big)\d r\d s\d y\d z\\
= & \int_{\R^{2}}\int_{0}^{1}\int_{0}^{1}K_{j}(x-y)K_{<j-1}(x-z)s\Big(\Big(F''\big(g(z)+rs\Delta g(y,z)\big)-F''\big(h(z)+rs\Delta h(y,z)\big)\Big)\\
 & \qquad\hspace{40bp}\times\Delta g(y,z)^{2}+F''\big(h(z)+rs\Delta h(y,z)\big)\big(\Delta g(y,z)^{2}-\Delta h(y,z)^{2}\big)\Big)\d r\d s\d y\d z.
\end{align*}
The Hölder continuity of $F''$ yields
\begin{align*}
|G_{j}(x)|\le & \|F\|_{C^{2+\gamma}}\int_{\R^{2}}\int_{0}^{1}\int_{0}^{1}\big|K_{j}(x-y)K_{<j-1}(x-z)\big|\Big(\Big|(g-h)(z)+rs\Delta(g-h)(y,z)\Big|^{\gamma}\\
 & \qquad\hspace{45bp}\times\big|\Delta g(y,z)\big|^{2}+\big|\Delta(g-h)(y,z)\big|\big(|\Delta g(y,z)|+|\Delta h(y,z)|\big)\Big)\d r\d s\d y\d z\\
\le & \|F\|_{C^{2+\gamma}}\int_{\R^{2}}\big|K_{j}(x-y)K_{<j-1}(x-z)\big|\Big(\|g-h\|_{\infty}^{\gamma}\big|\Delta g(y,z)\big|^{2}\\
 & \qquad\hspace{45bp}+\big|\Delta(g-h)(y,z)\big|\big(|\Delta g(y,z)|+|\Delta h(y,z)|\big)\Big)\d y\d z.
\end{align*}
Using the inequalities by Minkowski and Cauchy-Schwarz, we obtain
analogously to (\ref{eq:FjLp})
\begin{align*}
\|G_{j}\|_{L^{p/2}}\le & \|F\|_{C^{2+\gamma}}\int_{\R^{2}}\big|K_{j}(y)K_{<j-1}(z)\big|\Big(\|g-h\|_{\infty}^{\gamma}\|\Delta g(x-y,x-z)\|_{L^{p}}^{2}\\
 & +\|\Delta(g-h)(x-y,x-z)\|_{L^{p}}\big(\|\Delta g(x-y,x-z)\|_{L^{p}}+\|\Delta h(x-y,x-z)\|_{L^{p}}\big)\Big)\d y\d z\\
\le & \|F\|_{C^{2+\gamma}}\int_{\R}\big(\big|K_{j}|\ast|K_{<j-1}(-\cdot)|\big)(z)\\
 & \qquad\qquad\times\Big(\|g-h\|_{\infty}^{\gamma}\omega_{p}(g,|z|)^{2}+\omega_{p}(g-h,|z|)\big(\omega_{p}(g,|z|)+\omega_{p}(h,|z|)\big)\Big)\d z\\
\le & \|F\|_{C^{2+\gamma}}\Big(\|g-h\|_{\infty}^{\gamma}\|g\|_{\alpha,p,2q}^{2}+\|g-h\|_{\alpha,p,2q}\big(\|g\|_{\alpha,p,2q}+\|h\|_{\alpha,p,2q}\big)\Big)2^{-j2\alpha}.
\end{align*}
The claimed bound follows again from Lemma~\ref{lem:BesovAnnulus}
and the symmetry in $g$ and $h$.
\end{proof}
In the situation of Proposition~\ref{prop:reduction} we conclude
\[
F(u)=T_{F'(u)}u+R_{F}(u)\quad\text{with}\quad\|R_{F}(u)\|_{2\alpha,p/2,\infty}\lesssim\|u\|_{\alpha,p,\infty}^{2}.
\]
Due to this linearization it remains to study $\pi(T_{F'(u)}u,\xi)$.
For Hölder continuous functions \citet[Lem. 2.4]{Gubinelli2012} have
shown that the terms $\pi(T_{F'(u)}u,\xi)$ and $F'(u)\pi(u,\xi)$
only differ by a smoother remainder. To find an estimate of the regularity
for the commutator 
\begin{equation}
\Gamma(f,g,h):=\pi(T_{f}g,h)-f\pi(g,h)\label{eq:Commutator}
\end{equation}
in general Besov norms, we first prove the following auxiliary lemma,
cf. \citep[Lem. 2.97]{Bahouri2011}.
\begin{lem}
\label{lem:comAux} Let $p,p_{1},p_{2}\ge1$ such that $\frac{1}{p}=\frac{1}{p_{1}}+\frac{1}{p_{2}}\le1$.
Then for $\alpha\in(0,1)$ and for any $f\in B_{p_{1},\infty}^{\alpha}(\R^{d})$
and $g\in L^{p_{2}}(\R^{d})$ the operator $[\Delta_{j},f]g:=\Delta_{j}(fg)-f\Delta_{j}g$
satisfies
\[
\|[\Delta_{j},f]g\|_{L^{p}}\lesssim2^{-j\alpha}\|f\|_{\alpha,p_{1},\infty}\|g\|_{L^{p_{2}}}.
\]
\end{lem}
\begin{proof}
Since $\Delta_{j}f=(\F^{-1}\rho_{j})\ast f$ , we observe
\begin{align*}
[\Delta_{j},f]g(x)= & \F^{-1}\rho_{j}\ast(fg)(x)-f(\F^{-1}\rho_{j}\ast g)(x)\\
= & \int_{\mathbb{R}}\F^{-1}\rho_{j}(y)\big(f(x-y)-f(x)\big)g(x-y)\d y,\quad x\in\R^{d}.
\end{align*}
Minkowski's and Hölder's inequalities yield 
\begin{align*}
\|[\Delta_{j},f]g\|_{L^{p}}\le & \int_{\mathbb{R}}\big\|\F^{-1}\rho_{j}(y)\big(f(\cdot-y)-f\big)g(\cdot-y)\big\|_{L^{p}}\d y\\
\le & \|g\|_{L^{p_{2}}}\int_{\mathbb{R}}|\F^{-1}\rho_{j}(y)|\|f(\cdot-y)-f\|_{L^{p_{1}}}\d y.
\end{align*}
With the modulus of continuity (\ref{eq:ModCont}) and the corresponding
Besov norm, we obtain
\begin{align*}
\|[\Delta_{j},f]g\|_{L^{p}}\le & \|g\|_{L^{p_{2}}}\int_{\mathbb{R}}|\F^{-1}\rho_{j}(y)\omega_{p_{1}}(f,|y|)|\d y\\
\le & \|g\|_{L^{p_{2}}}\sup_{y\in\R^{d}}\big\{|y|^{-\alpha}\omega_{p_{1}}(f,|y|)\big\}\int_{\mathbb{R}}|y|^{\alpha}|\F^{-1}\rho_{j}(y)|\d y\\
\sim & \|f\|_{\alpha,p_{1},\infty}\|g\|_{L^{p_{2}}}\big\||y|^{\alpha}|\F^{-1}\rho_{j}(y)|\big\|_{L^{1}}.
\end{align*}
For $j=-1$ the previous $L^{1}$-norm is finite because $\chi$ is
smooth and compactly supported. For $j\ge0$ we additionally note
that $\F^{-1}\rho_{j}=2^{jd}\F\rho(2^{j}\cdot)$ implies
\[
\big\||y|^{\alpha}|\F^{-1}\rho_{j}(y)|\big\|_{L^{1}}=2^{-j\alpha}\big\||y|^{\alpha}|\F^{-1}\rho(y)|\big\|_{L^{1}}\lesssim2^{-j\alpha}.\qedhere
\]
\end{proof}
\begin{lem}
\label{lem:commutator} Let $\alpha\in(0,1)$, $\beta,\gamma\in\mathbb{R}$
such that $\alpha+\beta+\gamma>0$ and $\beta+\gamma<0$. Moreover,
let $p_{1},p_{2},p_{3}\ge1$ satisfy $\frac{1}{p_{1}}+\frac{1}{p_{2}}+\frac{1}{p_{3}}=\frac{1}{p}\le1$
and let $q\ge1$. Then for $f,g,h\in\mathcal{S}(\R^{d})$ the commutator
operator from (\ref{eq:Commutator}) satisfies 
\[
\|\Gamma(f,g,h)\|_{\alpha+\beta+\gamma,p,q}\lesssim\|f\|_{\alpha,p_{1},q}\|g\|_{\beta,p_{2},q}\|h\|_{\gamma,p_{3},q}.
\]
Therefore, $\Gamma$ can be uniquely extended to a bounded trilinear
operator 
\[
\Gamma\colon B_{p_{1},q}^{\alpha}(\R^{d})\times B_{p_{2},q}^{\beta}(\R^{d})\times B_{p_{3},q}^{\gamma}(\R^{d})\to B_{p,q}^{\alpha+\beta+\gamma}(\R^{d}).
\]
\end{lem}
\begin{proof}
Let $f,g,h\in\mathcal{S}(\R^{d})$. Using $T_{f}g=\sum_{k\ge-1}\sum_{l\ge k+2}\Delta_{k}f\Delta_{l}g=\sum_{k\ge-1}\Delta_{k}f(g-S_{k+2}g)$,
we decompose
\begin{align}
 & \Gamma(f,g,h)=\pi(T_{f}g,h)-f\pi(g,h)\nonumber \\
 & \quad=\sum_{j\ge-1}\sum_{i:|i-j|\le1}\big(\Delta_{i}(T_{f}g)\Delta_{j}h-f\Delta_{i}g\Delta_{j}h\big)\nonumber \\
 & \quad=\sum_{j,k\ge-1}\sum_{i:|i-j|\le1}\Big(\Delta_{i}\big((\Delta_{k}f)(g-S_{k+2}g)\big)-\Delta_{k}f\Delta_{i}g\Big)\Delta_{j}h\nonumber \\
 & \quad=-\sum_{k\ge-1}\underbrace{\sum_{j\ge-1}\sum_{i:|i-j|\le1}\Delta_{k}f\Delta_{i}(S_{k+2}g)\Delta_{j}h}_{=:a_{k}}+\sum_{j\ge-1}\underbrace{\sum_{k\ge-1}\sum_{i:|i-j|\le1}\big([\Delta_{i},\Delta_{k}f](g-S_{k+2}g)\big)\Delta_{j}h}_{=:b_{j}}.\label{eq:proofComLem}
\end{align}
We will separately estimate both sums in the following.

For $k\ge-1$ we have $\Delta_{i}(S_{k+2}g)=0$ for $i>k+2$ due to
property \ref{enu:suppRho} of the dyadic partition of unity. Consequently,
\[
a_{k}=\sum_{i=-1}^{k+2}\sum_{j:|i-j|\le1}\Delta_{k}f\Delta_{i}(S_{k+2}g)\Delta_{j}h
\]
and its Fourier transform satisfies $\supp\F a_{k}\subset2^{k}\mathcal{B}$
for some ball $\mathcal{B}$. Hölder's inequality yields 
\begin{align*}
\|a_{k}\|_{L^{p}}\le & \|\Delta_{k}f\|_{L^{p_{1}}}\sum_{i=-1}^{k+2}\sum_{j:|i-j|\le1}\|\Delta_{i}(S_{k+2}g)\|_{L^{p_{2}}}\|\Delta_{j}h\|_{L^{p_{3}}}.
\end{align*}
Owing to $\Delta_{i}(S_{k+2}g)=\Delta_{i}g$ for $i\le k$ and $\|\Delta_{i}\Delta_{k}g\|_{L^{p_{2}}}\le\|\F^{-1}\rho_{i}\|_{L^{1}}\|\Delta_{k}g\|_{L^{p_{2}}}\lesssim\|\Delta_{k}g\|_{L^{p_{2}}}$
by Young's inequality, we have
\begin{align*}
\|a_{k}\|_{L^{p}} & \lesssim\|\Delta_{k}f\|_{L^{p_{1}}}\sum_{i=-1}^{k+2}\sum_{j:|i-j|\le1}\|\Delta_{i}g\|_{L^{p_{2}}}\|\Delta_{j}h\|_{L^{p_{3}}}\\
 & \lesssim\|\Delta_{k}f\|_{L^{p_{1}}}\|g\|_{\beta,p_{2},\infty}\|h\|_{\gamma,p_{3},\infty}\sum_{i=-1}^{k+2}2^{-i(\beta+\gamma)}\lesssim2^{-k(\beta+\gamma)}\|\Delta_{k}f\|_{L^{p_{1}}}\|g\|_{\beta,p_{2},\infty}\|h\|_{\gamma,p_{3},\infty},
\end{align*}
using $\beta+\gamma<0$ in the last estimate. Since $2^{k\alpha}\|\Delta_{k}f\|_{L^{p_{1}}}\in\ell^{q}$,
Lemma~\ref{lem:BesovBall} yields 
\[
\Big\|\sum_{k\ge-1}a_{k}\Big\|_{\alpha+\beta+\gamma,p,q}\lesssim\|f\|_{\alpha,p_{1},q}\|g\|_{\beta,p_{2},\infty}\|h\|_{\gamma,p_{3},\infty}.
\]

Now, let us consider the second sum in (\ref{eq:proofComLem}). Note
that
\begin{align*}
b_{j}= & \sum_{i:|i-j|\le1}\sum_{k\ge-1}\sum_{l\ge k+2}\big([\Delta_{i},\Delta_{k}f]\Delta_{l}g\big)\Delta_{j}h=\sum_{i:|i-j|\le1}\sum_{l\ge-1}\big([\Delta_{i},S_{l-1}f]\Delta_{l}g\big)\Delta_{j}h.
\end{align*}
Since the support of the Fourier transform of $S_{l-1}f\Delta_{l}g$
is of the form $2^{l}\mathcal{A}$ for some annulus $\mathcal{A}$,
we have that 
\[
[\Delta_{i},S_{l-1}f]\Delta_{l}g=\Delta_{i}(S_{l-1}f\Delta_{l}g)-(S_{l-1}f)(\Delta_{i}\Delta_{l}g)
\]
vanishes if $|i-l|\ge N$ for some $N\in\N$. Therefore, $b_{j}=\sum_{i:|i-j|\le1}\sum_{l\sim i}\big([\Delta_{i},S_{l-1}f]\Delta_{l}g\big)\Delta_{j}h$
has a Fourier transform supported on $2^{j}$ times some annulus.
Using Hölder's inequality and Lemma~\ref{lem:comAux}, we estimate
\begin{align*}
\|b_{j}\|_{L^{p}}\lesssim & \sum_{i:|i-j|\le1}\sum_{l\sim i}2^{-i\alpha}\|S_{k-1}f\|_{\alpha,p_{1},\infty}\|\Delta_{l}g\|_{L^{p_{2}}}\|\Delta_{j}h\|_{L^{p_{3}}}\\
\lesssim & 2^{-j(\alpha+\beta+\gamma)}\|f\|_{\alpha,p_{1},\infty}(2^{j\beta}\sum_{l\sim j}\|\Delta_{l}g\|_{L^{p_{2}}})2^{j\gamma}\|\Delta_{j}h\|_{L^{p_{3}}}.
\end{align*}
For any $q_{2},q_{3}\ge q$ satisfying $\frac{1}{q}=\frac{1}{q_{2}}+\frac{1}{q_{3}}$
Hölder's inequality and Lemma~\ref{lem:BesovBall} yield then
\[
\Big\|\sum_{j\ge-1}b_{j}\Big\|_{\alpha+\beta+\gamma,p,q}\lesssim\|f\|_{\alpha,p_{1},\infty}\|g\|_{\beta,p_{2},q_{2}}\|h\|_{\gamma,p_{3},q_{3}}.
\]
To obtain the claimed norm bound, recall that $B_{p,q}^{\alpha}(\R^{d})$
continuously embeds into $B_{p,q'}^{\alpha}(\R^{d})$ for any $q\le q'$.

For $p,q<\infty$ the Schwartz space $\mathcal{S}(\R^{d})$ is dense
$B_{p,q}^{\alpha}(\R^{d})$ for any $\alpha\in\R$ such that there
is a unique extension of $C$ on $B_{p_{1},q}^{\alpha}(\R^{d})\times B_{p_{2},q}^{\beta}(\R^{d})\times B_{p_{3},q}^{\gamma}(\R^{d})$.
For $p=\infty$ or $q=\infty$ a similar argument as in \citep[Lem. 2.4]{Gubinelli2012}
applies.
\end{proof}
Combining the previous results, we obtain the following corollary,
cf. \citep[Lem. 2.7]{Gubinelli2012}, which immediately implies Proposition~\ref{prop:reduction}
due to the embedding $B_{p,q}^{\alpha}\subset L^{\infty}$ for $\alpha>1/p$
and $d=1$.
\begin{cor}
\label{cor:resonant} Let $p_{1},p_{2}\in[1,\infty]$ satisfy $\frac{2}{p_{1}}+\frac{1}{p_{2}}=:\frac{1}{p}\le1$.
Let $\alpha\in(0,1)$ and \textup{$\beta<0$} such that $2\alpha+\beta>0$
and $\alpha+\beta<0$. Further, suppose $F\in C^{2+\gamma}(\R^{m})$
for some $\gamma\in(0,1]$ satisfying $F(0)=0$. Then there exists
a map $\Pi_{F}\colon B_{p_{1},\infty}^{\alpha}(\R^{d})\times B_{p_{2},\infty}^{\beta}(\R^{d})\to B_{p,\infty}^{2\alpha+\beta}(\R^{d})$
such that
\[
\pi(F(f),g)=F'(f)\pi(f,g)+\Pi_{F}(f,g)
\]
and 
\[
\|\Pi_{F}(f,g)\|_{2\alpha+\beta,p,\infty}\lesssim\|F\|_{C^{2}}\|f\|_{\alpha,p_{1},\infty}^{2}\|g\|_{\beta,p_{2},\infty}.
\]
For $f_{1},f_{2}\in B_{p_{1},\infty}^{\alpha}(\R^{d})\cap L^{\infty}(\R^{d})$
and $g_{1},g_{2}\in B_{p_{2},\infty}^{\beta}(\R^{d})$ we have furthermore
\begin{align*}
 & \|\Pi_{F}(f_{1},g_{1})-\Pi_{F}(f_{2},g_{2})\|_{2\alpha+\beta,p,\infty}\\
 & \quad\lesssim\|F\|_{C^{2+\gamma}}\Big(\|f_{1}\|_{\alpha,p_{1},q}^{2}\wedge\|f_{2}\|_{\alpha,p_{1},\infty}^{2}+\big(\|f_{1}\|_{\alpha,p_{1},\infty}+\|f_{2}\|_{\alpha,p_{1},\infty}\big)\\
 & \quad\qquad\times\big(1+\|g_{1}\|_{\beta,p_{2},\infty}\wedge\|g_{1}\|_{\beta,p_{2},\infty}\big)\Big)\Big(\|f_{1}-f_{2}\|_{\infty}^{\gamma}+\|f_{1}-f_{2}\|_{\alpha,p_{1},\infty}+\|g_{1}-g_{2}\|_{\beta,p_{2},\infty}\Big).
\end{align*}
\end{cor}
\begin{proof}
Setting $\Pi_{F}(f,g):=\Gamma(F'(f),f,g)+\pi(R_{F}(f),g)$, we can
write 
\[
\pi(F(f),g)=F'(f)\pi(f,g)+\Gamma(F'(f),f,g)+\pi(R_{F}(f),g)=F'(f)\pi(f,g)+\Pi_{F}(f,g).
\]
Lemmas~\ref{lem:paraproduct}, \ref{lem:linearization} and \ref{lem:commutator}
yield
\begin{align*}
\|\Pi_{F}(f,g)\|_{2\alpha+\beta,p,\infty}\le & \|\Gamma(F'(f),f,g)\|_{2\alpha+\beta,p,\infty}+\|\pi(R_{F}(f),g)\|_{2\alpha+\beta,p,\infty}\\
\lesssim & \|F'(f)\|_{\alpha,p_{1},\infty}\|f\|_{\alpha,p_{1},\infty}\|g\|_{\beta,p_{2},\infty}+\|R_{F}(f)\|_{2\alpha,p_{1}/2,\infty}\|g\|_{\beta,p_{2},\infty}\\
\lesssim & \big(\|F'(f)\|_{\alpha,p_{1},\infty}+\|F\|_{C^{2}}\|f\|_{\alpha,p_{1},\infty}\big)\|f\|_{\alpha,p_{1},\infty}\|g\|_{\beta,p_{2},\infty},
\end{align*}
where we again used Besov embeddings. Finally, we apply (\ref{eq:F(u)}). 

The bound of $\|\Pi_{F}(f_{1},g_{1})-\Pi_{F}(f_{2},g_{2})\|_{2\alpha+\beta,p,\infty}$
follows from analogous estimates, using the argument-wise linearity
of $\Gamma$ and $\pi$, the Hölder continuity of $R_{F}$ from Lemma~\ref{lem:linearization}
and 
\begin{align}
\|F'(f_{1})-F'(f_{2})\|_{\alpha,p_{1},q}= & \Big\|\int_{0}^{1}F''(f_{1}+s(f_{2}-f_{1}))(f_{1}-f_{2})\d s\Big\|_{\alpha,p_{1},q}\nonumber \\
\le & \int_{0}^{1}\|F''(f_{1}+s(f_{2}-f_{1}))(f_{1}-f_{2})\|_{\alpha,p_{1},q}\d s\nonumber \\
\le & \|F''\|_{\infty}\|f_{1}-f_{2}\|_{\alpha,p_{1},q}\label{eq:estFu}
\end{align}
for any $q\in[1,\infty]$.
\end{proof}

\section{The paracontrolled ansatz\label{sec:paracontrolled ansatz}}

Assuming that the driving signal $\xi$ satisfies $\xi\in B_{p,q}^{\alpha}$
for $\alpha>1/3$, we come back to the RDE\,(\ref{eq:rde}). Recall
that it was given by

\[
\dd u(t)=F(u(t))\xi(t),\quad u(0)=u_{0},\quad t\in\R,
\]
where $u_{0}\in\mathbb{R}^{m}$, $u\colon\mathbb{R}\to\mathbb{R}^{m}$
is a continuous function and $F\colon\mathbb{R}^{m}\to\mathcal{L}(\mathbb{R}^{n},\mathbb{R}^{m})$
is a family of vector fields on $\R^{m}$. In Section~\ref{sec:Young}
we have already considered the case $\alpha>1/2$. The classical way
to continuously extend Young's approach to more irregular driving
signals is Lyons' rough path theory, which additionally to the signal
$\xi$ needs to handle the corresponding ``iterated integral''.

As an alternative, we use in the present section a new paracontrolled
ansatz similar to \citet{Gubinelli2012}. We postulate that the solution
$u$ of the RDE\,(\ref{eq:rde}) is of the form
\[
u=T_{u^{\theta}}\theta+u^{\#}
\]
with $\vartheta,u^{\vartheta}\in B_{p,q}^{\alpha}$ and a remainder
$u^{\#}\in B_{p/2,q}^{2\alpha}$. Decomposing $F(u)\xi$ in terms
of Littlewood-Paley blocks and linearizing $F$ by Proposition~\ref{prop:reduction},
we have 
\[
F(u)\xi=T_{F(u)}\xi+\pi(F(u),\xi)+T_{\xi}(F(u))=T_{F(u)}\xi+F^{\prime}(u)\pi(u,\xi)+\Pi_{F}(u,\xi)+T_{\xi}(F(u)).
\]
The presumed controlled structure yields that understanding the (problematic)
term $\pi(u,\xi)$ reduces further to the analysis of $\pi(\theta,\xi)$
owing to the commutator from (\ref{eq:Commutator}):
\[
\pi(u,\xi)=\pi(T_{u^{\vartheta}}\vartheta,\xi)+\pi(u^{\#},\xi)=u^{\vartheta}\pi(\theta,\xi)+\underbrace{\Gamma(u^{\theta},\theta,\xi)}_{\in B_{p/3,q}^{3\alpha-1}}+\underbrace{\pi(u^{\#},\xi)}_{\in B_{p/3,q}^{3\alpha-1}}.
\]
Plugging the paracontrolled ansatz into the RDE~(\ref{eq:rde}),
the Leibniz rule and the above observation yield
\begin{align*}
T_{u^{\vartheta}}\dd\vartheta+T_{\dd u^{\vartheta}}\theta+\dd u^{\#}=\dd u & =T_{F(u)}\xi+F'(u)\pi(u,\xi)+\Pi_{F}(u,\xi)+T_{\xi}(F(u)).
\end{align*}
Comparing the least regular terms on the left-hand and on the right-hand
side, we choose $\vartheta$ as the solution to  $\dd\vartheta=\xi$
with $\theta(0)=0$ and $u^{\vartheta}=F(u)$. 

As already noted in Section~\ref{sec:Young}, we cannot expect $\theta$
to be contained in any Besov space (cf. Lemma~\ref{lem:antiderivative}).
This requirement would especially be violated in most interesting
examples from probability theory, for instance, $\theta$ being Brownian
motion or a martingale. In order to circumvent this issue, we use
again the localizing function $\phi$ from Assumption~\ref{ass:phi}.
Still relying on $\dd\theta=\xi$ and $\theta(0)=0$, we introduce
the local version of the signal
\[
\theta_{\mathcal{T}}:=\phi_{\mathcal{T}}\theta\quad\text{and}\quad\xi_{\mathcal{T}}:=\dd\theta_{\mathcal{T}}=\phi_{\mathcal{T}}\xi+\phi_{\mathcal{T}}'\theta.
\]
The corresponding localized RDE is then given by 
\begin{equation}
\dd u=F(u)\xi_{\mathcal{T}},\quad u(0)=u_{0}.\label{eq:localRDE}
\end{equation}
This differential equation coincides with the original one on the
interval $[-\mathcal{T},\mathcal{T}]$ due to $\phi(t)=1$ and $\phi'(t)=0$
for $|t|\le\mathcal{T}$. 

Summarizing briefly the above discussion, we need two additional pieces
of information about very irregular signals. Namely, $\xi_{\mathcal{T}}$
has to be the derivative of a path $\theta_{\mathcal{T}}$ with compact
support and the resonant term $\pi(\theta_{\mathcal{T}},\xi_{\mathcal{T}})$
has to be well-defined. This precisely corresponds to the classical
rough path theory, where a path $\theta$ defined on some compact
interval is enhanced with the information of the iterated integral
$\int\theta_{s}\d\theta_{s}$. 

Analogously to the notion of geometric rough path (cf. for example
Section 2.2. in \citep{Friz2013}), we introduce now the notion of
geometric Besov rough path:
\begin{defn}
\label{def:geometric rough path} Let $\mathcal{T}>0$ and let $C_{\mathcal{T}}^{\infty}$
be the space of smooth functions $\theta_{\mathcal{T}}\colon\mathbb{R}\to\mathbb{R}^{n}$
with support $\supp\,\theta_{\mathcal{T}}\subset[-2\mathcal{T},2\mathcal{T}]$
and $\theta_{\mathcal{T}}(0)=0$. The closure of the set $\{(\theta_{\mathcal{T}},\pi(\theta_{\mathcal{T}},\dd\theta_{\mathcal{T}}))\,:\,\theta_{\mathcal{T}}\in C_{\mathcal{T}}^{\infty}\}\subset B_{p,q}^{\alpha}\times B_{p/2,q}^{2\alpha-1}$
with respect to the norm $\|\cdot\|_{\alpha,p,q}+\|\cdot\|_{2\alpha-1,p/2,q}$
is denoted by $\mathcal{B}_{p,q}^{0,\alpha}$ and $(\theta_{\mathcal{T}},\eta_{\mathcal{T}})\in\mathcal{B}_{p,q}^{0,\alpha}$
is called \textit{geometric Besov rough path}.
\end{defn}
Even with the driving signal $(\theta,\eta)\in\mathcal{B}_{p,q}^{0,\alpha}$
we unfortunately cannot expect in general that the solution $u$ to
the Cauchy problem~(\ref{eq:localRDE}) with $\xi_{\mathcal{T}}=\dd\theta_{\mathcal{T}}$
lies in any Besov spaces $B_{p,q}^{\alpha}$ for finite $p$ and $q$.
On the other hand, Besov spaces on the compact domain $[-\mathcal{T},\mathcal{T}]$
seem not be convenient for the paraproduct approach since Littlewood-Paley
theory and Bony's paraproduct are from their very nature constructed
on the whole real line. It appears to be natural to instead consider
a weighted version of the Itô-Lyons $\hat{S}$ map given by
\begin{align}
\hat{S}\colon\mathbb{R}^{m}\times\mathcal{B}_{p,q}^{0,\alpha} & \to B_{p,q}^{\alpha}\quad\text{via}\quad(u_{0},\theta_{\mathcal{T}},\pi(\theta_{\mathcal{T}},\dd\theta_{\mathcal{T}}))\mapsto\psi u,\label{eq:localito}
\end{align}
where $u$ solves (\ref{eq:localRDE}) with $\xi_{\mathcal{T}}=\dd\theta_{\mathcal{T}}$
and $\psi\colon\R\to(0,\infty)$ is a regular weight function being
constant one on $[-2\mathcal{T},2\mathcal{T}]$. Consequently, provided
$\theta_{\mathcal{T}}\in C_{\mathcal{T}}^{\infty}$ with $\xi_{\mathcal{T}}=\dd\theta_{\mathcal{T}}$
the weighted solution $\tilde{u}:=\psi u$ possesses the dynamic 
\begin{equation}
\dd\tilde{u}=\psi\dd u+\psi^{\prime}u=F(\tilde{u})\xi_{\mathcal{T}}+\frac{\psi^{\prime}}{\psi}\tilde{u},\quad\tilde{u}(0)=u_{0}.\label{eq:weightedRDE}
\end{equation}
Let us emphasize that also this weighted differential equation still
coincides with the original RDE~(\ref{eq:rde}) restricted to the
interval $[-\mathcal{T},\mathcal{T}]$. While the very recently developed
semigroup approach to paracontrolled calculus by \citet{Bailleul2015}
might allow for working without the weight $\psi$, this would lead
to non-standard Littlewood-Paley blocks and Besov spaces.

The aim is now to continuously extend the weighted Itô-Lyons map $\hat{S}$
from smooth functions with support in $[-2\mathcal{T},2\mathcal{T}]$
to the geometric Besov rough paths or more precisely from the domain
$\mathbb{R}^{d}\times\{(\theta_{\mathcal{T}},\pi(\theta_{\mathcal{T}},\dd\theta_{\mathcal{T}}))\,:\,\theta_{\mathcal{T}}\in C_{\mathcal{T}}^{\infty}\}$
to $\mathbb{R}^{d}\times\mathcal{B}_{p,q}^{0,\alpha}$. For this purpose
we specify our assumptions on the weight function $\psi$ as follows:
\begin{assumption}
\label{ass:weight} For any $\mathcal{T}>0$ let $\psi=\psi_{\mathcal{T}}\in B_{p,q}^{\alpha}\cap C^{1}$
be a strictly positive function which is equal to one on $[-2\mathcal{T},2\mathcal{T}]$
and suppose that there exist two constants $C_{\psi},c_{\psi}>0$
such that $\|\psi'/\psi\|_{\infty}\lesssim C_{\psi}$ and $\max\{\psi(2\mathcal{T}+1),\psi(-2\mathcal{T}-1)\}>c_{\psi}$.
\end{assumption}
The conditions on $\psi$ are quite weak and allow for a large variety
of weight functions as illustrated by the following examples.
\begin{example}
Let $\alpha\in(0,1)$, $\mathcal{T}>0$ and $\kappa\in(0,1)$.\end{example}
\begin{enumerate}
\item The function 
\[
\psi_{\mathcal{T}}(t):=\begin{cases}
1, & |t|\le2\mathcal{T},\\
\exp\Big(-\frac{\kappa(|t|-2\mathcal{T})^{2}}{1+|t|-2\mathcal{T}}\Big),\quad & |t|>2\mathcal{T},
\end{cases}
\]
satisfies Assumption~\ref{ass:weight} for $C_{\psi}=\kappa$ and
$c_{\psi}=e^{-1/2}$.
\item The function 
\[
\psi_{\mathcal{T}}(t):=\begin{cases}
1, & |t|\le2\mathcal{T},\\
\big(1+\kappa(|t|-2\mathcal{T})^{2}\big)^{-2},\quad & |t|>2\mathcal{T},
\end{cases}
\]
satisfies Assumption~\ref{ass:weight} for $C_{\psi}=\sqrt{\kappa}$
and $c_{\psi}=1/4$.
\end{enumerate}
For later reference let us remark a property which makes weight functions
fulfilling Assumption~\ref{ass:weight} so suitable in our context. 
\begin{rem}
\label{rem:weights} For any two weight functions $\psi$ and $\text{\ensuremath{\tilde{\psi}}}$
satisfying Assumption~\ref{ass:weight}, the resulting weighted Besov
norms of the solution $u$ are equivalent. More precisely, it is elementary
to show
\[
\|\psi u\|_{\alpha,p,q}\lesssim\big(1+c_{\tilde{\psi}}^{-1}\|\tilde{\psi}-\psi\|_{\alpha,p,q}\big)\|\tilde{\psi}u\|_{\alpha,p,q}
\]
for any $u\in B_{p,q}^{\alpha}$ which is constant on $(-\infty,-2\mathcal{T}]$
and on $[2\mathcal{T},\infty)$.
\end{rem}
In order to analyze the weighted RDE~(\ref{eq:weightedRDE}), we
modify our ansatz to
\[
\tilde{u}=T_{F(\tilde{u})}\theta_{\mathcal{T}}+u^{\#},\quad\text{where}\quad u^{\#}\in B_{p/2,q}^{2\alpha},\,\theta_{\mathcal{T}}\in C_{\mathcal{T}}^{\infty}.
\]
Roughly speaking, in the terminology of \citep{Gubinelli2012} the
pair $(\tilde{u},F(\tilde{u}))\in(B_{p,q}^{\alpha})^{2}$ is said
to be \emph{paracontrolled} by $\theta_{\mathcal{T}}\in B_{p,q}^{\alpha}$.
The dynamic of $u^{\#}$ is characterized in the next lemma.
\begin{lem}
\label{lem:characterization} Let $u_{0}\in\mathbb{R}^{m}$, let $\theta_{\mathcal{T}}\in C_{\mathcal{T}}^{\infty}$
with derivative $\xi_{\mathcal{T}}=\dd\vartheta_{\mathcal{T}}$ and
suppose that $\psi$ satisfies Assumption~\ref{ass:weight}. Then
the following conditions are equivalent: 
\begin{enumerate}
\item $u$ is the solution to  the ODE~(\ref{eq:localRDE}),
\item $u$ can be written as $u=\psi^{-1}\tilde{u}$ where $\tilde{u}$
solves the ODE~(\ref{eq:weightedRDE}), 
\item $\tilde{u}$ can be written as $\tilde{u}=T_{F(\tilde{u})}\theta_{\mathcal{T}}+u^{\#}$
where $u^{\#}$ solves
\begin{equation}
\dd u^{\#}=F(\tilde{u})\xi_{\mathcal{T}}-\dd(T_{F(\tilde{u})}\theta_{\mathcal{T}})+\frac{\psi^{\prime}}{\psi}\tilde{u},\qquad u^{\#}(0)=u_{0}-T_{F(\tilde{u})}\theta_{\mathcal{T}}(0).\label{eq:equiRDE-1}
\end{equation}

\end{enumerate}
\end{lem}
\begin{proof}
For the equivalence between (i) and (ii) note that $u=\psi^{-1}\tilde{u}$
is well-defined by Assumption~\ref{ass:weight} and that we have
by the Leibniz rule 
\[
\dd u=\dd(\psi^{-1}\tilde{u})=\psi^{-1}\dd\tilde{u}-\frac{\psi^{\prime}}{\psi^{2}}\tilde{u}=F(u)\xi_{\mathcal{T}},\qquad u(0)=\psi^{-1}(0)\tilde{u}(0)=u_{0}.
\]

The equivalence between (ii) and (iii) follows by combining $\tilde{u}=T_{F(\tilde{u})}\theta_{\mathcal{T}}+u^{\#}$
and (\ref{eq:weightedRDE}), which yields 
\begin{align*}
\dd u^{\#} & =\dd\tilde{u}-\dd(T_{F(\tilde{u})}\theta_{\mathcal{T}})=F(\tilde{u})\xi_{\mathcal{T}}-\dd(T_{F(\tilde{u})}\theta_{\mathcal{T}})+\frac{\psi^{\prime}}{\psi}\tilde{u}
\end{align*}
and due to $\tilde{u}(0)=u(0)=u_{0}$ the initial condition satisfies
$u^{\#}(0)=u_{0}-T_{F(\tilde{u})}\theta_{\mathcal{T}}(0).$
\end{proof}
As we have seen in the discussion at the beginning of the present
section, we want to reduce the resonant term $\pi(F(\tilde{u}),\xi_{\mathcal{T}})$
to the resonant term $\pi(\theta_{\mathcal{T}},\xi_{\mathcal{T}})$.
Indeed, this is possible as proven in the following proposition. The
specific form of $u$ allows to improve the quadratic estimate (\ref{eq:boundResonant})
in Proposition~\ref{prop:reduction} to a linear one. Its proof is
inspired by Lemma~5.2 by \citet{Gubinelli2012}.
\begin{prop}
\label{prop:BoundResonant} Let $\alpha\in(\frac{1}{3},\frac{1}{2})$,
$p\ge3$, $q\ge1$, and $F\in C^{2}$ with $F(0)=0$. If $\theta_{\mathcal{T}}\in C_{\mathcal{T}}^{\infty}$
with derivative $\xi_{\mathcal{T}}=\dd\theta_{\mathcal{T}}$, then
for $\tilde{u}=T_{F(\tilde{u})}\theta_{\mathcal{T}}+u^{\#}$ with
$\tilde{u}\in B_{p,q}^{\alpha}$ and $u^{\#}\in B_{p/2,q}^{2\alpha}$
one has 
\begin{align*}
\|\pi(F(\tilde{u}),\xi_{\mathcal{T}})\|_{2\alpha-1,p/2,q} & \lesssim\big(\|F\|_{C^{2}}\vee\|F\|_{C^{2}}^{2}\big)\big(\|\tilde{u}\|_{\alpha,p,q}+\|u^{\#}\|_{2\alpha,p/2,q}\big)\\
 & \qquad\times\big(\|\theta_{\mathcal{T}}\|_{\alpha,p,q}+\|\theta_{\mathcal{T}}\|_{\alpha,p,q}^{2}+\|\pi(\theta_{\mathcal{T}},\xi_{\mathcal{T}})\|_{2\alpha-1,p/2,q}\big).
\end{align*}
\end{prop}
\begin{proof}
\emph{Step 1:} To avoid the quadratic estimate, we first need a modified
version of Lemma~\ref{lem:linearization}. We will borrow some notation
from the proof of this former lemma. For brevity we define $v_{u}:=T_{F(\tilde{u})}\theta_{\mathcal{T}}$
and recall $\tilde{u}:=\psi u$ such that $\tilde{u}=v_{u}+u^{\#}.$
We write
\[
F(\tilde{u})-F(0)=T_{F'(\tilde{u})}\tilde{u}+R_{F}(\tilde{u})
\]
with
\[
R_{F}(\tilde{u})=\sum_{j\ge-1}F_{j}\quad\text{with}\quad F_{j}:=\Delta_{j}(F(\tilde{u})-F(0))-S_{j-1}(F'(\tilde{u}))\Delta_{j}\tilde{u}.
\]
For $j\le0$, we saw in Lemma~\ref{lem:linearization} that $\|F_{j}\|_{L^{p/2}}\lesssim\|F\|_{C^{1}}\|\tilde{u}\|_{L^{p/2}}$
which yields
\begin{align*}
\|F_{j}\|_{L^{p/2}}\lesssim & \|F\|_{C^{1}}(\|v_{u}\|_{L^{p/2}}+\|u^{\#}\|_{L^{p/2}})\\
\le & \|F\|_{C^{1}}(\|T_{F(\tilde{u})}\theta_{\mathcal{T}}\|_{L^{p/2}}+\|u^{\#}\|_{L^{p/2}}).
\end{align*}
For $j>0$, we deduce from (\ref{eq:Fj}) and our ansatz that
\begin{align*}
|F_{j}|= & \Big|\int_{\mathbb{R}^{2}}K_{j}(x-y)K_{<j-1}(x-z)\Big(\big(F'\big(\tilde{u}(z)+\xi_{yz}(\tilde{u}(y)-\tilde{u}(z))\big)-F'(\tilde{u}(z))\big)\\
 & \qquad\times\big(v_{u}(y)-v_{u}(z)+u^{\#}(y)-u^{\#}(z)\big)\Big)\d y\d z\Big|,\\
\le & \|F\|_{C^{2}}\int_{\mathbb{R}^{2}}|K_{j}(x-y)K_{<j-1}(x-z)||\tilde{u}(y)-\tilde{u}(z)||v_{u}(y)-v_{u}(z)|\d y\d z\\
 & \qquad+2\|F\|_{C^{1}}\int|K_{j}(x-y)K_{<j-1}(x-z)||u^{\#}(y)-u^{\#}(z)|\d y\d z.
\end{align*}
Proceeding as in proof of Lemma~\ref{lem:linearization} and applying
Hölder's inequality, we obtain for $q^{*}=q/(q-1)$
\begin{align*}
\|F_{j}\|_{L^{p/2}}\le & \|F\|_{C^{2}}\int_{\mathbb{R}^{2}}|K_{j}(y)K_{<j-1}(z)|\big\|\tilde{u}(x-(y-z))-\tilde{u}(x)\big\|_{L^{p}}\\
 & \qquad\qquad\quad\times\big\|(v_{u}(x-(y-z))-v_{u}(x)\big\|_{L^{p}}\d y\d z\\
 & +2\|F\|_{C^{1}}\int_{\mathbb{R}^{2}}|K_{j}(y)K_{<j-1}(z)|\big\| u^{\#}(x-(y-z))-u^{\#}(x)\big\|_{L^{p/2}}\d y\d z\\
\le & \|F\|_{C^{2}}\int_{\mathbb{R}^{2}}|K_{j}(y)K_{<j-1}(y-h)|\omega_{p}(\tilde{u},|h|)\omega_{p}(v_{u},|h|)\d y\d h\\
 & +2\|F\|_{C^{1}}\int_{\mathbb{R}^{2}}|K_{j}(y)K_{<j-1}(y-h)|\omega_{p/2}(u^{\#},|h|)\d y\d h\\
\le & \|F\|_{C^{2}}\Big\||h|^{2\alpha+1/q}\big(|K_{j}|\ast|K_{<j-1}(-\cdot)|\big)(h)\Big\|_{L^{q^{*}}}\\
 & \times\Big(\Big\||h|^{-\alpha}\omega_{p}(v_{u},|h|)\Big\|_{\infty}\Big\|(|h|^{-\alpha-1/q}\omega_{p}(\tilde{u},|h|)\Big\|_{L^{q}}+2\Big\||h|^{-2\alpha-1/q}\omega_{p/2}(u^{\#},|h|)\Big\|_{L^{q}}\Big)\\
\lesssim & 2^{-j2\alpha}\|F\|_{C^{2}}\big(\|v_{u}\|_{\alpha,p,\infty}\|\tilde{u}\|_{\alpha,p,q}+\|u^{\#}\|_{2\alpha,p/2,q}\big).
\end{align*}
Due to Lemma~\ref{lem:paraproduct} one further has 
\[
\|v_{u}\|_{\alpha,p,\infty}=\|T_{F(\tilde{u})}\theta_{\mathcal{T}}\|_{\alpha,p,\infty}\lesssim\|T_{F(\tilde{u})}\theta_{\mathcal{T}}\|_{\alpha,p,q}\lesssim\|F\|_{\infty}\|\theta_{\mathcal{T}}\|_{\alpha,p,q}
\]
and thus Lemma~\ref{lem:BesovAnnulus} gives 
\begin{align}
\|R_{F}(\tilde{u})\|_{2\alpha,p/2,\infty} & \lesssim\|F\|_{C^{2}}(1+\|F\|_{\infty}\|\theta_{\mathcal{T}}\|_{\alpha,p,q})(\|\tilde{u}\|_{\alpha,p,q}+\|u^{\#}\|_{2\alpha,p/2,q}).\label{eq:RF(u)}
\end{align}

\emph{Step 2:} Plugging in the ansatz once again and keeping the definition
of our commutator (\ref{eq:Commutator}) in mind, we decompose
\begin{align}
\pi(F(\tilde{u}),\xi_{\mathcal{T}})= & \pi(T_{F'(\tilde{u})}\tilde{u},\xi_{\mathcal{T}})+\pi(R_{F}(\tilde{u}),\xi_{\mathcal{T}})\nonumber \\
= & \pi(T_{F'(\tilde{u})}T_{F(\tilde{u})}\theta_{\mathcal{T}},\xi_{\mathcal{T}})+\pi(T_{F'(\tilde{u})}u^{\#},\xi_{\mathcal{T}})+\pi(R_{F}(\tilde{u}),\xi_{\mathcal{T}})\nonumber \\
= & F'(\tilde{u})\pi(T_{F(\tilde{u})}\theta_{\mathcal{T}},\xi_{\mathcal{T}})+\Gamma(F'(\tilde{u}),T_{F(\tilde{u})}\theta_{\mathcal{T}},\xi_{\mathcal{T}})+\pi(T_{F'(\tilde{u})}u^{\#},\xi_{\mathcal{T}})+\pi(R_{F}(\tilde{u}),\xi_{\mathcal{T}})\nonumber \\
= & F'(\tilde{u})F(\tilde{u})\pi(\theta_{\mathcal{T}},\xi_{\mathcal{T}})+F'(\tilde{u})\Gamma(F(\tilde{u}),\theta_{\mathcal{T}},\xi_{\mathcal{T}})+\Gamma(F'(\tilde{u}),T_{F(\tilde{u})}\theta_{\mathcal{T}},\xi_{\mathcal{T}})\nonumber \\
 & \qquad+\pi(T_{F'(\tilde{u})}u^{\#},\xi_{\mathcal{T}})+\pi(R_{F}(\tilde{u}),\xi_{\mathcal{T}}).\label{eq:pi decomposition}
\end{align}
Therefore, we can bound $\|\pi(F(\tilde{u}),\xi_{\mathcal{T}})\|_{2\alpha-1,p/2,q}$
by estimating these five terms separately.\emph{ }We will apply the
following bound which holds owing to the Besov embedding $B_{p/3,q/2}^{3\alpha-1}\subset B_{p/2,q/2}^{2\alpha-1}$
due to $\alpha>1/p$ and which uses Bony's estimates and $2\alpha-1<0$:
for $f\in L^{\infty}\cup B_{p,\infty}^{\alpha}$ and $g\in B_{p/2,q/2}^{2\alpha-1}$
it holds 
\begin{align}
\|fg\|_{2\alpha-1,p/2,q/2}\lesssim & \|T_{f}g\|_{2\alpha-1,p/2,q/2}+\|\pi(f,g)\|_{3\alpha-1,p/3,q/2}+\|T_{g}f\|_{2\alpha-1,p/2,q/2}\nonumber \\
\lesssim & \|f\|_{\infty}\|g\|_{2\alpha-1,p/2,q/2}+\big(\|f\|_{0,\infty,\infty}\|g\|_{3\alpha-1,p/3,q/2}\wedge\|f\|_{\alpha,p,\infty}\|g\|_{2\alpha-1,p/2,q/2}\big)\nonumber \\
 & \qquad+\|g\|_{2\alpha-1,p/2,q/2}\|f\|_{0,\infty,\infty}\nonumber \\
\lesssim & \big(\|f\|_{\infty}\|g\|_{3\alpha-1,p/3,q/2}\big)\wedge\big(\|f\|_{\alpha,p,\infty}\|g\|_{2\alpha-1,p/2,q/2}\big).\label{eq:estimateProd}
\end{align}
Furthermore, note for the following estimates that $\|\xi_{\mathcal{T}}\|_{\alpha-1,p,q}\lesssim\|\theta_{\mathcal{T}}\|_{\alpha,p,q}$
thanks to the lifting property of Besov spaces, cf. \citep[Thm. 2.3.8]{triebel2010}.

Applying (\ref{eq:estimateProd}) and (\ref{eq:F(u)}) to $\tilde{F}:=F'F$,
we obtain for the first summand 
\begin{align*}
\|F'(\tilde{u})F(\tilde{u})\pi(\theta_{\mathcal{T}},\xi_{\mathcal{T}})\|_{2\alpha-1,p/2,q}\lesssim & \|\tilde{F}(\tilde{u})\|_{\alpha,p,\infty}\|\pi(\theta_{\mathcal{T}},\xi_{\mathcal{T}})\|_{2\alpha-1,p/2,q}\\
\lesssim & \|F\|_{C^{1}}\|F\|_{C^{2}}\|\tilde{u}\|_{\alpha,p,q}\|\pi(\theta_{\mathcal{T}},\xi_{\mathcal{T}})\|_{2\alpha-1,p/2,q}.
\end{align*}
For the second term the above estimate~(\ref{eq:estimateProd}) and
Lemma~\ref{lem:commutator} yield
\begin{align*}
\|F'(\tilde{u})\Gamma(F(\tilde{u}),\theta_{\mathcal{T}},\xi_{\mathcal{T}})\|_{2\alpha-1,p/2,q} & \lesssim\|F'\|_{\infty}\|\Gamma(F(\tilde{u}),\theta_{\mathcal{T}},\xi_{\mathcal{T}})\|_{3\alpha-1,p/3,q}\\
 & \lesssim\|F'\|_{\infty}\|F(\tilde{u})\|_{\alpha,p,q}\|\theta_{\mathcal{T}}\|_{\alpha,p,q}\|\xi_{\mathcal{T}}\|_{\alpha-1,p,q}\\
 & \lesssim\|F\|_{C^{1}}^{2}\|\tilde{u}\|_{\alpha,p,q}\|\theta_{\mathcal{T}}\|_{\alpha,p,q}^{2},
\end{align*}
where (\ref{eq:F(u)}) is used in the last line. Lemmas~\ref{lem:paraproduct}
and \ref{lem:commutator} again together with (\ref{eq:F(u)}) gives
for the third term
\begin{align*}
\|\Gamma(F'(\tilde{u}),T_{F(\tilde{u})}\theta_{\mathcal{T}},\xi_{\mathcal{T}})\|_{2\alpha-1,p/2,q} & \lesssim\|F'(\tilde{u})\|_{\alpha,p,q}\|T_{F(\tilde{u})}\theta_{\mathcal{T}}\|_{\alpha,p,q}\|\xi_{\mathcal{T}}\|_{\alpha-1,p,q}\\
 & \lesssim\|F\|_{C^{1}}^{2}\|\tilde{u}\|_{\alpha,p,q}\|\theta_{\mathcal{T}}\|_{\alpha,p,q}^{2}.
\end{align*}
The second last term in (\ref{eq:pi decomposition}) can be estimated
by 
\begin{align*}
\|\pi(T_{F'(\tilde{u})}u^{\#},\xi_{\mathcal{T}})\|_{2\alpha-1,p/2,q} & \lesssim\|T_{F'(\tilde{u})}u^{\#}\|_{2\alpha,p/2,q}\|\xi_{\mathcal{T}}\|_{\alpha-1,p,q}\\
 & \lesssim\|F'\|_{\infty}\|u^{\#}\|_{2\alpha,p/2,q}\|\theta_{\mathcal{T}}\|_{\alpha,p,q}
\end{align*}
where a Besov embedding, Lemma~\ref{lem:paraproduct} and (\ref{eq:F(u)})
are used. Finally, for the last term, note that there is some $\epsilon\in(0,\alpha-\frac{1}{p})$
such that $3\alpha-1-\epsilon>0$. Applying Lemma~\ref{lem:paraproduct},
Step 1 and Besov embeddings, we get 
\begin{align*}
\|\pi(R_{F}(\tilde{u}),\xi_{\mathcal{T}})\|_{2\alpha-1,p/2,q} & \lesssim\|\pi(R_{F}(\tilde{u}),\xi_{\mathcal{T}})\|_{3\alpha-1-\epsilon,p/3,q}\\
 & \lesssim\|R_{F}(\tilde{u})\|_{2\alpha-\epsilon,p/2,q}\|\xi_{\mathcal{T}}\|_{\alpha-1,p,q}\\
 & \lesssim\|F\|_{C^{2}}(1+\|F\|_{\infty}\|\theta_{\mathcal{T}}\|_{\alpha,p,q})(\|\tilde{u}\|_{\alpha,p,q}+\|u^{\#}\|_{2\alpha,p/2,q})\|\theta_{\mathcal{T}}\|_{\alpha,p,q}.
\end{align*}
These five estimates combined lead to the asserted bound. \end{proof}
\begin{rem}
The requirement $F(0)=0$ seems to be a purely technical assumption.
In view of Lemma~\ref{lem:linearization}, we can decompose in general
$\pi(F(\tilde{u}),\xi_{\mathcal{T}})=\pi(F(\tilde{u})-F(0),\xi_{\mathcal{T}})+\pi(F(0),\xi_{\mathcal{T}})$.
If $p=\infty$ the additional term can be easily estimated. If $p<\infty$,
it seems more reasonable to decompose $F(\tilde{u})\xi_{\mathcal{T}}=(F(\tilde{u})-F(0))\xi_{\mathcal{T}}+F(0)\xi_{\mathcal{T}}$
at the beginning. Hence, we decided to assume the condition $F(0)=0$.
Otherwise, all estimates would become even more involved by keeping
track of the additional term due to $F(0)\neq0$ without needing conceptional
new ideas.
\end{rem}
Having established a linear upper bound for the resonant term $\pi(F(\tilde{u}),\xi_{\mathcal{T}})$,
we deduce the boundedness of the solution to  the localized RDE~(\ref{eq:localRDE})
in the weighted Besov norm.
\begin{cor}
\label{cor:bound} Let $\alpha\in(1/3,1/2)$, $p\ge3$, $q\ge1$ and
$F\in C^{2}$ with $F(0)=0$. Let $\theta_{\mathcal{T}}\in C_{\mathcal{T}}^{\infty}$
with derivative $\xi_{\mathcal{T}}=\dd\vartheta_{\mathcal{T}}$. If
the bound 
\begin{align*}
\|F\|_{C^{2}}\vee\|F\|_{C^{2}}^{2} & <c(\mathcal{T}^{3}\vee1)\big(\|\theta_{\mathcal{T}}\|_{\alpha-1,p,q}+\|\theta_{\mathcal{T}}\|_{\alpha,p,q}^{2}+\|\pi(\theta_{\mathcal{T}},\xi_{\mathcal{T}})\|_{2\alpha-1,p/2,q}\big)^{-1}
\end{align*}
holds for a universal constant $c>0$, independent of $\theta$, $F$,
$u_{0}$ and if $\psi$ satisfies Assumption~\ref{ass:weight} for
some sufficiently small $C_{\psi}$, then the solution $u$ to (\ref{eq:localRDE})
satisfies
\begin{align*}
\|\psi u\|_{\alpha,p,q} & \lesssim(\mathcal{T}^{2}\vee1)\big(|u(0)|+(\|F\|_{C^{2}}\vee\|F\|_{C^{2}}^{3})(\|\theta_{\mathcal{T}}\|_{\alpha,p,q}+1)\\
 & \qquad\times\big(\|\theta_{\mathcal{T}}\|_{\alpha,p,q}+\|\theta_{\mathcal{T}}\|_{\alpha,p,q}^{2}+\|\pi(\theta_{\mathcal{T}},\xi_{\mathcal{T}})\|_{2\alpha-1,p/2,q}\big)\big).
\end{align*}
\end{cor}
\begin{proof}
We recall the characterization of $\tilde{u}=\psi u$ from Lemma~\ref{lem:characterization}.
In order to obtain the desired estimate of the norm, we apply Bony's
decomposition and calculate 
\begin{align}
\dd u^{\#} & =F(\tilde{u})\xi_{\mathcal{T}}-\dd(T_{F(\tilde{u})}\theta_{\mathcal{T}})+\frac{\psi^{\prime}}{\psi}\tilde{u}\nonumber \\
 & =T_{F(\tilde{u})}\xi_{\mathcal{T}}+\pi(F(\tilde{u}),\xi_{\mathcal{T}})+T_{\xi_{\mathcal{T}}}(F(\tilde{u}))-\dd(T_{F(\tilde{u})}\theta_{\mathcal{T}})+\frac{\psi^{\prime}}{\psi}\tilde{u}\nonumber \\
 & =\pi(F(\tilde{u}),\xi_{\mathcal{T}})+T_{\xi_{\mathcal{T}}}(F(\tilde{u}))-T_{\dd(F(\tilde{u}))}\theta_{\mathcal{T}}+\frac{\psi^{\prime}}{\psi}\tilde{u}.\label{eq:phi1}
\end{align}
We bound the $B_{p/2,q}^{2\alpha-1}$-norm of these four terms separately.
The first term is bounded by Proposition~\ref{prop:BoundResonant}.
To estimate the second term in (\ref{eq:phi1}), Lemma~\ref{lem:paraproduct},
(\ref{eq:F(u)}) and a Besov embedding yield
\begin{align*}
\|T_{\xi_{\mathcal{T}}}(F(\tilde{u}))\|_{2\alpha-1,p/2,q} & \lesssim\|F\|_{C^{1}}\|\xi_{\mathcal{T}}\|_{\alpha-1,p,2q}\|\tilde{u}\|_{\alpha,p,2q}\\
 & \lesssim\|F\|_{C^{1}}\|\theta_{\mathcal{T}}\|_{\alpha,p,q}\|\tilde{u}\|_{\alpha,p,q}.
\end{align*}
The third term in (\ref{eq:phi1}) can be estimated with the lifting
property of Besov spaces, Lemma~\ref{lem:paraproduct}, (\ref{eq:F(u)})
and a Besov embedding 
\begin{align*}
\|T_{\dd(F(\tilde{u}))}\theta_{\mathcal{T}}\|_{2\alpha-1,p/2,q} & \lesssim\|\dd F(\tilde{u})\|_{\alpha-1,p,2q}\|\theta_{\mathcal{T}}\|_{\alpha,p,2q}\\
 & \lesssim\|F(\tilde{u})\|_{\alpha,p,2q}\|\theta_{\mathcal{T}}\|_{\alpha,p,2q}\lesssim\|F\|_{C^{1}}\|\tilde{u}\|_{\alpha,p,q}\|\theta_{\mathcal{T}}\|_{\alpha,p,q}.
\end{align*}
For the last term in (\ref{eq:phi1}) we note the norm equivalence
$\|\psi u\|_{L^{p/2}}\sim\|\tilde{\psi}u\|_{L^{p/2}}$ with for $u$
being constant outside of $[-2\mathcal{T},2\mathcal{T}]$, where we
set $\tilde{\psi}:=\psi\psi_{2}$ for another weight function $\psi_{2}$
satisfying Assumption~\ref{ass:weight}. Hence, $\|\tilde{u}\|_{L^{p/2}}\lesssim\|\psi_{2}\tilde{u}\|_{L^{p/2}}\le\|\text{\ensuremath{\psi}}_{2}\|_{L^{p}}\|\tilde{u}\|_{L^{p}}$
by Hölder's inequality. Since $2\alpha-1<0$, a Besov embedding yields
\[
\|\frac{\psi^{\prime}}{\psi}\tilde{u}\|_{2\alpha-1,p/2,q}\lesssim\|\frac{\psi^{\prime}}{\psi}\tilde{u}\|_{L^{p/2}}\lesssim\|\frac{\psi'}{\psi}\|_{\infty}\|\tilde{u}\|_{L^{p/2}}\lesssim(\mathcal{T}\vee1)\|\frac{\psi^{\prime}}{\psi}\|_{\infty}\|\tilde{u}\|_{L^{p}}.
\]
Combining all the above estimates, we obtain 
\begin{align*}
\|\dd u^{\#}\|_{2\alpha-1,p/2,q} & \lesssim C_{\xi,\mathcal{\theta}}(\|F\|_{C^{2}}\vee\|F\|_{C^{2}}^{2})\big(\|\tilde{u}\|_{\alpha,p,q}+\|u^{\#}\|_{2\alpha,p/2,q}\big)+(\mathcal{T}\vee1)\|\frac{\psi^{\prime}}{\psi}\|_{\infty}\|\tilde{u}\|_{\alpha,p,q}
\end{align*}
with 
\begin{align*}
C_{\xi,\mathcal{\theta}} & :=\|\theta_{\mathcal{T}}\|_{\alpha,p,q}+\|\theta_{\mathcal{T}}\|_{\alpha,p,q}^{2}+\|\pi(\theta_{\mathcal{T}},\xi_{\mathcal{T}})\|_{2\alpha-1,p/2,q}.
\end{align*}
Applying again the lifting property of Besov spaces \citep[Thm. 2.3.8]{triebel2010}
together with the definition of $u^{\#}$, $\|\tilde{u}\|_{L^{p/2}}\lesssim(\mathcal{T}\vee1)\|\tilde{u}\|_{L^{p}}$
and the compact support of $\theta_{\mathcal{T}}$, we have 
\begin{align}
\|u^{\#}\|_{2\alpha,p/2,q}\lesssim & \|u^{\#}\|_{L^{p/2}}+\|\dd u^{\#}\|_{2\alpha-1,p/2,q}\nonumber \\
\le & \|T_{F(\tilde{u})}\theta_{\mathcal{T}}\|_{L^{p/2}}+\|\tilde{u}\|_{L^{p/2}}+\|\dd u^{\#}\|_{2\alpha-1,p/2,q}\nonumber \\
\lesssim & (\mathcal{T}\vee1)\big(\|F\|_{\infty}\|\theta_{\mathcal{T}}\|_{L^{p}}+\|\tilde{u}\|_{L^{p}}\big)+\|\dd u^{\#}\|_{2\alpha-1,p/2,q}.\label{eq:u raute 1}
\end{align}
Hence, combining the last two inequalities leads to 
\begin{align*}
\|\dd u^{\#}\|_{2\alpha-1,p/2,q} & \lesssim C_{\xi,\mathcal{\theta}}(\|F\|_{C^{2}}\vee\|F\|_{C^{2}}^{2})\big(\|\tilde{u}\|_{\alpha,p,q}+\|\dd u^{\#}\|_{2\alpha-1,p/2,q}\big)\\
 & \quad+(\mathcal{T}\vee1)\big(C_{\xi,\theta}(\|F\|_{C^{2}}\vee\|F\|_{C^{2}}^{2})(\|F\|_{\infty}\|\theta_{\mathcal{T}}\|_{L^{p}}+\|\tilde{u}\|_{L^{p}})+\|\frac{\psi^{\prime}}{\psi}\|_{\infty}\|\tilde{u}\|_{\alpha,p,q}\big).
\end{align*}
If $C_{\xi,\mathcal{\theta}}(\|F\|_{C^{2}}\vee\|F\|_{C^{2}}^{2})$
is sufficiently small, we thus obtain 
\begin{align}
 & \|\dd u^{\#}\|_{2\alpha-1,p/2,q}\nonumber \\
 & \quad\lesssim(\mathcal{T}\vee1)C_{\xi,\mathcal{\theta}}(\|F\|_{C^{2}}\vee\|F\|_{C^{2}}^{2})\big(\|\tilde{u}\|_{\alpha,p,q}+\|F\|_{\infty}\|\theta_{\mathcal{T}}\|_{\alpha,p,q}\big)+(\mathcal{T}\vee1)\|\frac{\psi^{\prime}}{\psi}\|_{\infty}\|\tilde{u}\|_{\alpha,p,q}.\label{eq:u raute 2}
\end{align}
In combination with the ansatz and the bounds from above, Lemma~\ref{lem:paraproduct}
reveals
\begin{align*}
\|\dd\tilde{u}\|_{\alpha-1,p,q}\le & \|\dd(T_{F(\tilde{u})}\theta_{\mathcal{T}})\|_{\alpha-1,p,q}+\|\dd u^{\#}\|_{\alpha-1,p,q}\\
\lesssim & \|T_{\dd F(\tilde{u})}\theta_{\mathcal{T}}\|_{2\alpha-1,p/2,q}+\|T_{F(\tilde{u})}\xi_{\mathcal{T}}\|_{\alpha-1,p,q}+\|\dd u^{\#}\|_{2\alpha-1,p/2,q}\\
\lesssim & (\mathcal{T}\vee1)\Big(C_{\xi,\mathcal{\theta}}(\|F\|_{C^{2}}\vee\|F\|_{C^{2}}^{2})\big(\|\tilde{u}\|_{\alpha,p,q}+\|F\|_{\infty}\|\theta_{\mathcal{T}}\|_{\alpha,p,q}+1\big)+\|\frac{\psi^{\prime}}{\psi}\|_{\infty}\|\tilde{u}\|_{\alpha,p,q}\Big).
\end{align*}
Due to Remark~\ref{rem:weights} applied to $\tilde{\psi}=\psi\psi_{2}$,
we can apply Lemma~\ref{lem:antiderivative} to obtain
\begin{align*}
\|\tilde{u}\|_{\alpha,p,q} & \lesssim\|\psi_{2}\tilde{u}\|_{\alpha,p,q}\le(\mathcal{T}^{2}\vee1)\big(|u(0)|+\|\dd\tilde{u}\|_{\alpha-1,p,q}\big).\\
 & \lesssim\underbrace{\big((\mathcal{T}^{3}\vee1)C_{\xi,\mathcal{\theta}}(\|F\|_{C^{2}}\vee\|F\|_{C^{2}}^{2})+\|\frac{\psi'}{\psi}\|_{\infty}\big)}_{=:D}\|\tilde{u}\|_{\alpha,p,q}\\
 & \qquad+(\mathcal{T}^{2}\vee1)\big(|u(0)|+C_{\xi,\mathcal{\theta}}(\|F\|_{C^{2}}\vee\|F\|_{C^{2}}^{2})(\|F\|_{\infty}\|\theta_{\mathcal{T}}\|_{\alpha,p,q}+1)\big).
\end{align*}
For $D$ smaller than some universal constant we conclude the assertion.
\end{proof}
For any $F\in C^{\text{3}}$ and $\|F\|_{C^{3}}$ small enough, the
following lemma reveals that the weighted Itô-Lyons map $\hat{S}$
as introduced in (\ref{eq:localito}) is locally Lipschitz continuous
with respect to the Besov norms on $\mathbb{R}^{d}\times B_{p,q}^{\alpha-1}\times B_{p/2,q}^{2\alpha-1}$
and thus it can be uniquely extended in a continuous way. 
\begin{lem}
\label{lem:lipschitz} Let $\alpha\in(1/3,1/2)$, $p\ge3$, $q\ge1$
and let $F\in C^{3}$ with $F(0)=0$. Assume $\psi$ is a weight function
satisfying Assumption~\ref{ass:weight} and let $\theta_{\mathcal{T}}\in C_{0}^{\infty}$
with derivative $\xi_{\mathcal{T}}=\dd\theta_{\mathcal{T}}$. Then
there exits a polynomial on $\mathbb{R}^{3}$ such that, provided
the bound 
\[
\|F\|_{C^{3}}+\|F\|_{C^{2}}^{3}\le P(\mathcal{T}\vee1,\|\theta_{\mathcal{T}}\|_{\alpha,p,q},\|\pi(\xi_{\mathcal{T}},\theta_{\mathcal{T}})\|_{2\alpha-1,p,q})^{-1},
\]
holds and $C_{\psi}$ is sufficiently small, there exists for every
$u_{0}\in\mathbb{R}^{d}$ a unique global solution $u\in\mathcal{S}^{\prime}$
with $\psi u\in B_{p,q}^{\alpha}$ to the Cauchy problem~(\ref{eq:localRDE}).
Furthermore, for fixed $\mathcal{T}$, $\psi$ and $F$ the weighted
Itô-Lyons map $\hat{S}$ is local Lipschitz continuous on $\mathbb{R}^{d}\times C_{\mathcal{T}}^{\infty}$
around $(u_{0},\theta_{\mathcal{T}},\pi(\theta_{\mathcal{T}},\xi_{\mathcal{T}}))$.
\end{lem}
The local Lipschitz continuity is the key ingredient to extend the
weighted Itô-Lyons map from smooth paths to irregular ones. The proof
works similarly to the proofs of Proposition~\ref{prop:BoundResonant}
and Corollary~\ref{cor:bound} with an additional application of
the Lipschitz result in Proposition~\ref{prop:reduction}. Due to
the necessary, but quite lengthy estimations, we postpone the proof
to Appendix~\ref{app:proofLipschitz} with the hope to increase the
readability of the paper. 

Finally, we can state our main result: There exist a continuous extension
of the weighted Itô-Lyons map $\hat{S}$ from $\mathbb{R}^{d}\times C_{\mathcal{T}}^{\infty}$
to the domain $\mathbb{R}^{d}\times\mathcal{B}_{p,q}^{0,\alpha}$.
Similarly to Theorem~\ref{thm:Young} we use a dilation argument
together with a localization procedure to circumvent the assumption
that $\|F\|_{C^{3}}$ has to be small. Allowing for general Besov
spaces, this theorem generalizes Lyons' celebrated Universal Limit
Theorem \citep[Thm. 6.2.2]{Lyons2002} and in particular \citep[Thm. 3.3]{Gubinelli2012}.
\begin{thm}
\label{thm:gobalsolution} Let $\mathcal{T}>0$, $\alpha\in(1/3,1/2)$,
$p\geq3$, $q\ge1$ and $F\in C^{3}$ with $F(0)=0$. If the weight
function $\psi$ satisfies Assumption~\ref{ass:weight} with $C_{\psi}$
sufficiently small, then the weighted Itô-Lyons map $\hat{S}$ as
introduced in (\ref{eq:localito}) can be continuously extended from
$\mathbb{R}^{d}\times C_{\mathcal{T}}^{\infty}$ to \textup{the domain
$\mathbb{R}^{d}\times\mathcal{B}_{p,q}^{0,\alpha}$}. In particular,
there exists a unique solution to (\ref{eq:localito}) for any geometric
Besov rough path $(\theta_{\mathcal{T}},\pi(\theta_{\mathcal{T}},\dd\theta_{\mathcal{T}}))\in\mathcal{B}_{p,q}^{0,\alpha}$.
\end{thm}
An elementary formulation of Theorem~\ref{thm:gobalsolution} is
presented in the next lemma. The proof of Theorem~\ref{thm:gobalsolution}
is then an immediate consequence. 
\begin{lem}
Assume the weight function $\psi$ satisfies Assumption~\ref{ass:weight}
with $C_{\psi}$ sufficiently small. Let $\mathcal{T}>0$, $\alpha\in(1/3,1/2)$,
$p\geq3$, $q\ge1$ and $F\in C^{3}$ with $F(0)=0$. Let further
$u_{0}\in\R^{m}$ be an initial condition and $(\theta_{\mathcal{T}},\eta_{\mathcal{T}})\in\mathcal{B}_{p,q}^{0,\alpha}$
be a geometric Besov rough path. Let $(\theta_{\mathcal{T}}^{n})\subset C_{\mathcal{T}}^{\infty}$
be a sequence of functions with corresponding derivatives $(\xi_{\mathcal{T}}^{n})$
and $(u_{0}^{n})\subset\mathbb{R}^{m}$ be a sequence of initial conditions
such that $(u_{0}^{n},\theta_{\mathcal{T}}^{n},\pi(\theta_{\mathcal{T}}^{n},\xi_{\mathcal{T}}^{n}))$
converges to $(u_{0},\theta_{\mathcal{T}},\eta_{\mathcal{T}})$ in
$\mathbb{R}^{m}\times B_{p,q}^{\alpha-1}\times B_{p/2,q}^{2\alpha-1}$.
Denote by $u^{n}$ the unique solution to the Cauchy problem~(\ref{eq:localRDE})
with $u_{0}^{n}$ and $\xi_{\mathcal{T}}^{n}$ for all $n\in\N$.
Then there exists $u\in\mathcal{S}^{\prime}$ such that $\psi u\in B_{p,q}^{\alpha}$
and $\psi u^{n}\to\psi u$ in $B_{p,q}^{\alpha}$. The limit $u$
depends only on $(u_{0},\theta_{\mathcal{T}},\eta_{\mathcal{T}})$
and not on the approximating family $(u_{0}^{n},\theta_{\mathcal{T}}^{n},\pi(\theta_{\mathcal{T}}^{n},\xi_{\mathcal{T}}^{n}))$.\end{lem}
\begin{proof}
In order to apply Lemma~\ref{lem:lipschitz}, we first need to ensure
that $\|F\|_{C^{3}}$ is small enough. Thus, as similarly done in
Step 2 of the proof of Theorem~\ref{thm:Young}, we scale $\theta_{\mathcal{T}}^{n}$:
For some fixed $\epsilon\in(0,\alpha-1/p)$ and for $\lambda\in(0,1)$
we set 
\[
\theta_{\mathcal{T}}^{n,\lambda}:=\lambda^{-\alpha+1/p+\epsilon}\Lambda_{\lambda}\theta_{\mathcal{T}}^{n}\quad\text{and}\quad\xi_{\mathcal{T}}^{n,\lambda}:=\lambda^{1-\alpha+1/p+\epsilon}\Lambda_{\lambda}\xi_{\mathcal{T}}^{n},
\]
where we recall the scaling operator $\Lambda_{\lambda}f=f(\lambda\cdot)$
for $f\in\mathcal{S}'$. Given this scaling, still $\xi_{\mathcal{T}}^{n,\lambda}=\dd\theta_{\mathcal{T}}^{n,\lambda}$
holds true and the corresponding norms of $\xi_{\mathcal{T}}^{n,\lambda}$
and $\theta_{\mathcal{T}}^{n,\lambda}$ can be controlled by the Lemmas~\ref{lem:antiderivative}
and \ref{lem:scaling}, i.e. 
\begin{align*}
\|\xi_{\mathcal{T}}^{n,\lambda}\|_{\alpha-1,p,q} & \lesssim\|\xi_{\mathcal{T}}^{n}\|_{\alpha-1,p,q}\quad\text{and}\quad\|\theta_{\mathcal{T}}^{n,\lambda}\|_{\alpha,p,q}\lesssim(1\vee\mathcal{T}^{2})\|\xi_{\mathcal{T}}^{n,\lambda}\|_{\alpha-1,p,q}\lesssim(1\vee\mathcal{T}^{2})\|\xi_{\mathcal{T}}^{n}\|_{\alpha-1,p,q}.
\end{align*}
Moreover, again using Lemma~\ref{lem:scaling} we can estimate 
\begin{align*}
\|\pi(\theta_{\mathcal{T}}^{n,\lambda},\xi_{\mathcal{T}}^{n,\lambda})\|_{2\alpha-1,p/2,q} & =\lambda^{1-2\alpha+2/p+2\epsilon}\|\pi(\Lambda_{\lambda}\theta_{\mathcal{T}}^{n},\Lambda_{\lambda}\xi_{\mathcal{T}}^{n})\|_{2\alpha-1,p/2,q}\\
 & \lesssim(\lambda^{2\epsilon}|\log\lambda|+\lambda^{1-2\alpha+2\epsilon})\|\pi(\theta_{\mathcal{T}}^{n},\xi_{\mathcal{T}}^{n})\|_{2\alpha-1,p/2,q}.
\end{align*}
Let us take once more the localization function $\phi$ from Assumption~\ref{ass:phi}
and noticing that $\phi_{2\mathcal{T}}\theta_{\mathcal{T}}^{n}=\theta_{\mathcal{T}}^{n}$
for all $n\in\mathbb{N}$. Therefore, Lemma~\ref{lem:lipschitz}
provides for $\lambda>0$ sufficiently small a unique global solution
$u^{n,\lambda}\in B_{p,q}^{\alpha}$ to 
\[
\dd u^{n,\lambda}=\lambda^{\alpha-1/p-\epsilon}F(u^{n,\lambda})\dd(\phi_{2\mathcal{T}}\theta_{\mathcal{T}}^{n,\lambda}),\quad u^{n,\lambda}(0)=u_{0}^{n}.
\]
Setting now $u^{n}:=\Lambda_{\lambda^{-1}}u^{n,\lambda}$, we have
constructed a unique global solution to 
\[
\dd u^{n}=F(u^{n})\dd(\phi_{2\lambda\mathcal{T}}\theta_{\mathcal{T}}^{n}),\quad u(0)=u_{0}^{n}.
\]
Since $(u_{0}^{n},\theta_{\mathcal{T}}^{n,\lambda},\pi(\theta_{\mathcal{T}}^{n,\lambda},\xi_{\mathcal{T}}^{n,\lambda}))$
converges to $(u_{0},\theta_{\mathcal{T}}^{\lambda},\pi(\theta_{\mathcal{T}}^{\lambda},\xi_{\mathcal{T}}^{\lambda}))$
in $\mathbb{R}^{d}\times B_{p,q}^{\alpha-1}\times B_{p/2,q}^{2\alpha-1}$,
the continuity of the Itô-Lyons map established in Lemma~\ref{lem:lipschitz}
implies that $u^{n,\lambda}$ converges to some $u^{\lambda}$ in
$B_{p,q}^{\alpha}$ weighted by $\psi$. Therefore, the solution $u^{n}$
converges to $u:=\Lambda_{\lambda^{-1}}u^{\lambda}$ in $B_{p,q}^{\alpha}$
weighted by $\psi$, due to Lemma~\ref{lem:scaling} and \ref{lem:scaling},
which can be seen analogously to Step 2 of the proof of Theorem~\ref{thm:lipschitz young}.
We note that $u|_{[-\lambda\mathcal{T},\lambda\mathcal{T}]}$ does
not depend on $\phi_{\lambda\mathcal{T}}$.

Following the same argumentation as in Step 3 of the proof of Theorem~\ref{thm:lipschitz young},
we can iterate this construction of $u^{n}$ and $u$ on intervals
of the length $2\lambda T$. In this way we end up with a continuous
function $u$ such that $\psi u\in B_{p,q}^{\alpha}$ and $\psi u^{n}$
converges to $\psi u$ in $B_{p,q}^{\alpha}$. Note that $u$ depends
only on $(u_{0},\theta_{\mathcal{T}},\pi(\theta_{\mathcal{T}},\xi_{\mathcal{T}}))$
but neither on approximating family $(u_{0}^{n},\theta_{\mathcal{T}}^{n},\pi(\theta_{\mathcal{T}}^{n},\xi_{\mathcal{T}}^{n}))$
nor on $\phi_{\lambda T}$. 
\end{proof}
While general Besov spaces contain functions with jumps, the paracontrolled
distribution approach to rough differential equations as explored
in the present section only studies continuous functions. Therefore,
we think a discussion is in order why the paracontrolled distribution
approach seems to be naturally restricted to continuous functions. 
\begin{rem}
\label{remark:jumps} The results in Section~\ref{sec:LinCom} apply
only to Besov spaces $B_{p,q}^{\alpha}$ for $p\ge1$. According to
(\ref{eq:boundResonant}), our estimates result in a bound of the
$B_{p/3,q}^{3\alpha-1}$-norm. Consequently, we require $p\ge3$ and
$\alpha>1/3$ in order to have positive regularity. In particular,
our main theorem applies only to the case $\alpha>1/p$ which implies
that $B_{p,q}^{\alpha}$ embeds into the space of continuous functions.

If we want to extend our results to discontinuous functions, corresponding
to $\alpha<1/p$, then we could hope that it helps to verify the previous
results for $p<1$. Let us sketch some details on this idea, where
we have to deal with the quasi-Banach space $B_{p,q}^{\alpha}$ for
$p<1$. In that case the triangle inequality only holds true up to
a multiplicative constant
\[
\|f+g\|_{\alpha,p,q}\le2^{1/p-1}\big(\|f\|_{\alpha,p,q}+\|g\|_{\alpha,p,q}\big)\quad\text{ for }f,g\in B_{p,q}^{\alpha}.
\]
Following the lines of the proof of Lemma~2.84 (or Lemma~2.49 respectively)
in \citet{Bahouri2011}, we obtain in the case $p\in(0,1)$, $q>1$,
$\alpha>1/p-1$, for $u:=\sum_{j}u_{j}$ with $\supp u_{j}\subset2^{j}\mathcal{B}$
for some ball $\mathcal{B}$ that
\[
\|u\|_{s-(1/p-1),p,q}\lesssim\big\|\big(2^{js}\|u_{j}\|_{L^{p}}\big)_{j}\big\|_{\ell^{q}},
\]
provided the right-hand side is finite. For the commutator lemma in
the case $p\in(0,1)$ we thus cannot hope for more than the following:
Replacing the assumption $p\ge1$ with $\alpha+\beta+\gamma>(\frac{1}{p}-1)\vee0$
in the situation of Lemma~\ref{lem:commutator}, we conjecture 
\[
\|\Gamma(f,g,h)\|_{\alpha+\beta+\gamma-(\frac{1}{p}-1)\vee0,p,q}\lesssim\|f\|_{\alpha,p_{1},q}\|g\|_{\beta,p_{2},q}\|h\|_{\gamma,p_{3},q}.
\]
Applying this bound to (\ref{eq:boundResonant}), we obtain for $p\in(0,1)$
\[
\|\Pi_{F}(u,\xi)\|_{3\alpha-1-(3/p-1),p/3,q}<\big(\|F'(u)\|_{\alpha,p,q}+\|u\|_{\alpha,p,q}\big)\|u\|_{\alpha,p,q}\|\xi\|_{\alpha-1,p,q}.
\]
However, $3\alpha-1-(3/p-1)>0$ is equivalent to $\alpha>1/p$, which
is the same condition as we had before, excluding discontinuous functions. 

Alternatively, a higher order expansion in the linearization Lemma~\ref{lem:linearization}
could be studied (corresponding to more additional information). If
such a second order expansion would succeed, we may have the condition
$4\alpha-1>0$, but with the price of imposing $p/4\ge1$. Consequently,
we would again obtain $\alpha>1/p$. 

In conclusion, it appears natural that this approach is restricted
to continuous functions. 
\end{rem}

\section{Stochastic differential equations\label{sec:SDE}}

The purely analytic results from the previous sections for rough differential
equations allow for treating a large class of stochastic differential
equations (SDEs) in a pathwise way. While we assumed so far that the
driving signal $\xi$ of the RDE~(\ref{eq:rde}) is given by a deterministic
function with a certain Besov regularity, we suppose from now on that
$\xi$ is the distributional derivative of some continuous stochastic
process $X$. Provided all involved stochastic objects live on a suitable
probability space $(\Omega,\mathcal{F},\mathbb{P})$ and setting $\xi:=\dd X$,
the RDE~(\ref{eq:rde}) becomes an SDE with the dynamic 
\begin{equation}
\dd u(t)=F(u(t))\dd X_{t},\quad u(0)=u_{0},\quad t\in[0,1],\label{eq:sde}
\end{equation}
where $u_{0}$ is a random variable in $\R^{m}$ and $X$ is some
$d$-dimensional stochastic process for simplicity on the interval
$[0,1]$. 

Instead of relying on classical stochastic integration in order to
give the SDE~(\ref{eq:sde}) a meaning, we shall demonstrate here
that the results of Section~\ref{sec:Young} and \ref{sec:paracontrolled ansatz}
are feasible for a wide class of SDEs. For this propose the present
section is devoted to show the required sample path properties of
a couple of stochastic processes. This allows for solving SDEs which
are beyond the scope of classical probability theory as well as for
recovering well-known examples. Let us emphasize that we present here
only a few exemplary stochastic processes to illustrate our results
and do not aim for the most general class of stochastic processes.

\subsection*{Gaussian processes}

A well-known but very common example for a stochastic driving signal
$X$ is the fractional Brownian motion, cf. \citep{Coutin2007,Mishura2008}.
A $d$-dimensional fractional Brownian motion $B^{H}=(B^{1},\dots,B^{d})$
with Hurst index $H\in(0,1)$ is a Gaussian process with zero mean,
independent components, and covariance function given by 
\[
\mathbb{E}[B_{s}^{i}B_{t}^{i}]=\frac{1}{2}\big(s^{2H}+t^{2H}-\vert t-s\vert^{2H}\big),\quad s,t\in[0,1],
\]
for $i=1,\dots,d$. The Besov regularity of (fractional) Brownian
motion is already known for a long time due to \citet{Roynette1993}
and \citet{Ciesielski1993}: it holds $(B_{t}^{H})_{t\in[0,1]}\in B_{p,\infty}^{H}([0,1],\mathbb{R}^{d})$
almost surely for any $p\in[1,\infty]$ and $(B_{t}^{H})_{t\in[0,1]}\notin B_{p,q}^{H}([0,1],\mathbb{R}^{d})$
almost surely if $q<\infty$, see for instance \citep[Corrollary 5.3]{Veraar2009}.
More recently, \citet{Veraar2009} investigated the Besov regularity
for more general Gaussian processes. The self-similar behavior of
fractional Brownian motion implies that $B^{H}$ has the same regularity
$H$ with respect to all $p$-scales of the Besov spaces. Therefore,
it suffices to focus on $p=\infty$ for this example. 

Even if one could still rely on results from rough path theory (\citet{Lyons1998}
or \citet{Gubinelli2012}) in the case $H>1/3$, the following lemma
shows how to recover the results for SDEs with our machinery. It in
particular covers the fractional Brownian motion.
\begin{lem}[{\citep[Cor. 3.10]{Gubinelli2012}}]
 Let $X$ be a centered $d$-dimensional Gaussian process with independent
components whose covariance function fulfills for some $H\in(1/4,1)$
the Coutin-Qian condition 
\begin{align}
\mathbb{E}[\vert X_{t}-X_{s}\vert^{2}] & \lesssim\vert t-s\vert^{2H}\quad\text{and}\nonumber \\
\vert\mathbb{E}[(X_{s+r}-X_{s})(X_{t+r}-X_{t})]\vert & \lesssim\vert t-s\vert^{2H-2}r^{2},\label{eq:gaussian}
\end{align}
for all $s,t\in\mathbb{R}$ and \textcolor{black}{all $r\in[0,\vert t-s\vert)$.}
For every $\alpha<H$ and any smooth function $\phi$ with compact
support we have $\phi X\in B_{\infty,\infty}^{\alpha}$. Moreover,
there exists an $\eta\in B_{\infty,\infty}^{2\alpha-1}$ such that
for every $\delta>0$ and every $\psi\in\mathcal{S}$ with $\int\psi(t)\d t=1$
it holds
\[
\lim_{n\to\infty}\mathbb{P}\big(\|\psi^{n}*(\phi X)-(\phi X))\|_{\alpha,\infty,\infty}+\|\pi(\psi^{n}*(\phi X),\dd(\psi^{n}*(\phi X))-\eta)\|_{2\alpha-1,\infty,\infty}>\delta\big)=0,
\]
where we denote $\psi^{n}:=n\psi(n\cdot)$. 
\end{lem}
In other words, every $d$-dimensional Gaussian process $X$ satisfying
the Coutin-Qian condition~(\ref{eq:gaussian}) for some $H\in(1/3,1/2)$
can be enhanced to a geometric Besov rough path and especially Theorem~\ref{thm:gobalsolution}
can be applied to solve the SDE~(\ref{eq:sde}), cf. \citet{Coutin2002}
or \citet{Friz2010c}.

\subsection*{Stochastic processes via Schauder expansions\label{sub:schauder}}

Instead of approximating stochastic processes by processes with smooth
sample paths, in probability theory it is often more convenient to
construct a process via an expansion with respect to a basis of $L^{2}$.
The presumably most famous construction of this type is the Karhunen-Loève
expansion of Gaussian processes.

A classical construction of a Brownian motion on the interval $[0,1]$
is the Lévy-Ciesielski construction based on Schauder functions. More
generally, Schauder functions are a very frequently applied tool in
stochastic analysis. Notably, they are used to investigate the Besov
regularity of stochastic processes, cf. for example \citet{Ciesielski1993}
and \citet{Rosenbaum2009}, and very recently \citet{Gubinelli2014}
constructed directly the rough path integral in terms of Schauder
expansions. 

The Schauder functions can be defined as the antiderivatives of the
Haar functions. More explicitly they are given by 
\[
G_{j,k}(t):=2^{-j/2}\psi\big(2^{j}t-(k-1)\big)\quad\mbox{with}\quad\psi(t):=t\1_{[0,1/2]}(t)-(t-\tfrac{1}{2})\1_{(1/2,1]}(t),\quad t\in\R,
\]
for $j\in\mathbb{N}$ and $1\leq k\leq2^{n}$, and $G_{0,0}(0):=1$.
The Haar functions form a basis of $L^{2}([0,1],\mathbb{R})$ and
it is obvious that $G_{n,k}\in B_{p,q}^{\beta}$ for $0<\beta<1$
and $p,q\in[1,\infty]$ with $\beta>1/p$, cf. \citep[Prop. 9]{Rosenbaum2009}.
The next lemma explains why an approximation of stochastic processes
in terms of Schauder expansions can also be used to show that a process
can be enhanced to a geometric Besov rough path. 
\begin{lem}
\label{lem:schauder approximation} Let $\alpha\in(1/3,1/2)$, $\beta\in(1/2,1]$,
$p\geq2$ and $q\geq1$. Suppose $(f^{n})\subset B_{p,q}^{\beta}$
is a sequence of functions such that $\supp f^{n}\subset[0,1]$ for
all $n\in\mathbb{N}$. If $(f^{n},\pi(f^{n},\dd f^{n}))$ converges
in $B_{p,q}^{\alpha}\times B_{p/2,q}^{2\alpha-1}$ to some $(f,\pi(f,\dd f))\in B_{p,q}^{\alpha}\times B_{p/2,q}^{2\alpha-1}$,
then \textup{$(f,\pi(f,\dd f))\in\mathcal{B}_{p,q}^{0,\alpha}$.}\end{lem}
\begin{proof}
Let us recall that $C_{1}^{\infty}$ is dense in $\{g\in B_{p,q}^{\beta}:\supp g\subset[0,1]\}$.
Hence, for every $n\in\mathbb{N}$ there exists a sequence of smooth
functions $(f^{n,m})_{m}\subset C_{1}^{\infty}$ such that $(f^{n,m},\dd f^{n,m})$
converges to $(f^{n},\dd f^{n})$ in $B_{p,q}^{\beta}\times B_{p,q}^{\beta-1}$
as $m$ goes to infinity, where the convergence of the second component
follows from the convergence of the first one using the lifting property
of Besov spaces. Since $\beta>1/2$, we also have by Lemma~\ref{lem:paraproduct}
that $\pi(f^{n,m},\dd f^{n,m})$ converges to $\pi(f^{n},\dd f^{n})$
as $m$ goes to infinity. Therefore, taking a diagonal sequence there
exists a sequence of smooth functions $(f^{n,m(n)})_{n}\subset C_{1}^{\infty}$
such that $(f,\pi(f,\dd f))=\lim_{n\to\infty}\pi(f^{n,m(n)},\dd f^{n,m(n)})$
where the limit is taken in $B_{p,q}^{\alpha}\times B_{p/2,q}^{2\alpha-1}$.
\end{proof}
Based on Lemma~\ref{lem:schauder approximation} it is now an immediate
consequence of Theorem 6.5 and 6.6. in \citep{Gubinelli2014} that
suitable hypercontractive processes and continuous martingales can
be lifted to geometric Besov rough paths since the Lévy area term
in \citep{Gubinelli2014} corresponds to our resonant term. Especially,
all examples from probability theory in \citep{Gubinelli2014} are
feasible with our results as well.

\subsection*{Random functions via wavelet expansions: a prototypical example}

Random Fourier series have been enhanced to rough paths by \citet{Friz2013a}.
Due to the localization of the trigonometric basis in Fourier domain,
it is quite convenient to use in their examples also the paracontrolled
approach. Working with Fourier series requires to localize the signal.
Motivated by the previous construction, we shall instead consider
stochastic processes which can be constructed as series expansion
with random coefficients and with respect to a wavelet basis. There
are several applications of such models, for instance, in non-parametric
Bayesian statistics to construct priors on function spaces. One advantage
is that the sample path regularity of such processes can be determined
precisely, cf. \citet{Abramovich1998}, \citet{Chioica2012} and \citet{Bochkina2013}.
Note that very similar calculations apply also to Fourier series,
requiring some extra technical effort for the localization function.

Wavelets can be taken to be localized in the time domain as well as
in the Fourier domain. The latter property is quite convenient when
working with Littlewood-Paley theory as we demonstrate in the following.
Let $\{\psi_{j,k}\,:\,j\in\mathbb{N},\,k\in\Z\}$ be an orthonormal
wavelet basis of $L^{2}(\mathbb{R})$, where $\psi_{j,k}(t):=2^{j/2}\psi(2^{j}t-k)$
for $j\ge1$, $k\in\Z$, $t\in\R$, and $\psi\in L^{2}(\mathbb{R})$.
Then, any function $f\in L^{2}(\mathbb{R})$ can be written as 
\[
f(t):=\sum_{j=0}^{\infty}\sum_{k\in\Z}\langle f,\psi_{j,k}\rangle\psi_{j,k}(t),\quad t\in\R,\quad\mbox{with}\quad\langle f,\psi_{j,k}\rangle:=\int_{\R}f(s)\psi_{j,k}(s)\d s.
\]
Replacing the deterministic wavelet coefficients with real valued
random variables $(Z_{j,k})_{j,k}$, we now study stochastic processes
of the type 

\begin{equation}
X_{t}:=\sum_{j\ge0}\sum_{k=-2^{j}}^{2^{j}}Z_{j,k}\psi_{j,k}(t),\quad t\in\R.\label{eq:random function}
\end{equation}
Without loss of generality, we truncated the series expansion in $k$
since we always have to localize the signal in order to apply our
results concerning RDEs, see the equations~(\ref{eq:lip-1}) and
(\ref{eq:localRDE}). Let us impose the following weak assumptions
on $(Z_{j,k})_{j,k}$ and $(\psi_{j,k})_{j,k}$:
\begin{assumption}
\label{ass:Z} Let $\{\psi_{j,k}\,:\,j\in\mathbb{N},\,k\in\Z\}$ be
an orthonormal and band limited wavelet basis of $L^{2}(\mathbb{R})$
and suppose $Z_{j,k}=A_{j,k}B_{j,k}$ for all $j\ge0$ and $k=-2^{j},\dots,2^{j}$
where
\begin{itemize}
\item $(A_{j,k})_{j,k}$ are random variables satisfying $\E[A_{j,k}^{p}]^{1/p}\lesssim2^{-js}$
for some $s>0$ and $p\in\{2,4\}$, 
\item \textup{$\mathbb{E}[A_{j,k}]=0$ for all $j,k$ and $\E[A_{j,k}A_{m,n}]=0$
for $j\neq m$ or $k\neq n$,}
\item $(B_{j,k})_{j,k}$ are Bernoulli random variables with\textup{ $\mathbb{P}(B_{j,k}=1)=2^{-jr}$
for some $r\in[0,1)$,}
\item $\E[A_{j,k}B_{j,k}A_{m,n}B_{m,n}]=\E[A_{j,k}A_{m,n}]\mathbb{E}[B_{j,k}B_{m,n}]$
for all $j,k,m,n$.
\end{itemize}
\end{assumption}
The assumption allows for a quite flexible class of stochastic processes
although it is chosen in a way to keep the required analysis simple.
Having in mind the construction of Brownian motion via Schauder functions,
as mentioned before, the process $X$ behaves like a Wiener process
if $(Z_{j,k})_{j,k}$ are i.i.d. standard normal distributed random
variables with $s=1$. In particular, the self-similar behavior of
Brownian motion is then achieved because all wavelet coefficients
at a level $j$ are of the same order of magnitude (especially $r=0$).
If instead $r\in(0,1)$, we expect only a number of $2\cdot2^{j(1-r)}$
non-zero wavelet coefficients at each level $j$ and we consequently
gain from measuring the regularity of $X$ in a $B_{p,q}^{\alpha}$-norm
for some finite $p$.

In order to profit from $(Z_{j,k})_{j,k}$ being uncorrelated we choose
an even number $p$. Together with the requirement $p\ge3$ in our
uniqueness and existence theorem for RDEs (Theorem~\ref{thm:gobalsolution}),
we thus take $p=4$. Keeping in mind that the Littlewood-Paley theory
relies on decomposing functions into blocks with compact support in
the Fourier domain, we postulate to take band limited wavelets, e.g.
Meyer wavelets. Note that $X$ then is not compactly supported, but
exponentially concentrated on a fixed interval for an appropriate
choice of $\psi$. We obtain the following sample path regularity
of $X$:
\begin{lem}
\label{lem:RegX} If $X$ is defined as in (\ref{eq:random function})
and satisfies Assumption~\ref{ass:Z}, then $X\in B_{p,1}^{\alpha}$
almost surely for any $\alpha<s+\frac{r}{p}-\frac{1}{2}$ and for
$p\in\{2,4\}$.\end{lem}
\begin{proof}
Applying formally the Littlewood-Paley decomposition, one has $X=\sum_{j\geq-1}\Delta_{j}X$
and for the sake of brevity we introduce the multi-indices $\lambda=(j,k)$
with $|\lambda|:=j$. Noting that by the assumption on the wavelet
basis $\supp\F\psi_{\lambda}\subset2^{|\lambda|}\mathcal{A}$ for
some annulus $\mathcal{A}$ independent of $\lambda,$ we obtain $\Delta_{j}\psi_{\lambda}=0$
if $|j-|\lambda||$ is larger than some fixed integer. Therefore,
the Littlewood-Paley blocks are well-defined and given by  
\[
\Delta_{j}X=\sum_{\lambda:|\lambda|\sim j}Z_{\lambda}\Delta_{j}\psi_{\lambda}\quad\text{for}\quad j\geq-1.
\]
Further, let us remark that $X$ as given in (\ref{eq:random function})
exists in $B_{p,1}^{\alpha}$ if $\sum_{j}\Delta_{j}X$ exists as
limit in $B_{p,1}^{\alpha}$.

In order to show the claimed Besov regularity, we have to verify 
\[
\|\Delta_{j}X\|_{L^{p}}\lesssim2^{-j(s+r/p-1/2)}\quad\text{for}\quad j\geq-1,\quad p\in\{2,4\}.
\]
Let us focus on $p=2$. The case $p=4$ can be proved similarly relying
on the estimates for the forth moments of $(Z_{\lambda})$, see also
Lemma~\ref{lem:exResonant} below. For $j\geq-1$ we have
\begin{align*}
\E[\|\Delta_{j}X\|_{L^{2}}^{2}] & =\int_{\R}\E\Big[\Big(\sum_{\lambda}Z_{\lambda}\Delta_{j}\psi_{j,k}(t)\Big)^{2}\Big]\dd t\\
 & =\sum_{\lambda,\lambda'}\E[Z_{\lambda}Z_{\lambda'}]\int\Delta_{j}\psi_{\lambda}(t)\Delta_{j}\psi_{\lambda'}(t)\d t\lesssim\sum_{\lambda}2^{-(2s+r)|\lambda|}\int(\Delta_{j}\psi_{\lambda})^{2}(t)\d t,
\end{align*}
where the last equality follows from $(Z_{\lambda})$ being mutually
uncorrelated. Hence, we further estimate
\begin{align*}
\E[\|\Delta_{j}X\|_{L^{2}}^{2}] & \lesssim\sum_{\lambda:|\lambda|\sim j}2^{-(2s+r)|\lambda|}\|\Delta_{j}\psi_{\lambda}\|_{L^{2}}^{2}\\
 & \lesssim\sum_{j'\sim j}2^{-(2s+r)j'}\sum_{k=-2^{j'}}^{2^{j'}}\|\psi_{j',k}\|_{L^{2}}^{2}=2\sum_{j'\sim j}2^{-2j'(s+r/2-1/2)}.
\end{align*}
By the Littlewood-Paley characterization of the Besov norm we conclude
\[
\E[\|X\|_{\alpha,p,1}]=\sum_{j\ge-1}2^{j\alpha}\E[\|\Delta_{j}X\|_{L^{p}}]\lesssim\sum_{j\ge-1}2^{j(\alpha-s-r/2+1/2)},
\]
which is finite whenever $\alpha<s+r/2-1/2$. \end{proof}
\begin{rem}
With analogous estimates as in Lemma~\ref{lem:RegX} it is easy to
show that $X\in B_{p,1}^{\alpha}$ a.s. for any $\alpha<s+\frac{r}{p}-\frac{1}{2}$
for any even $p\ge2$ provided $\E[A_{j,k}^{p}]^{1/p}\lesssim2^{-s}$
still holds for these higher powers.
\end{rem}
The derivative of $X$ is naturally given by $\dd X_{t}=\sum_{j,k}Z_{j,k}\psi_{j,k}^{\prime}(t)$
for $t\in\R$. The crucial point is now, that we can indeed verify
that the resonant term $\pi(X,\dd X)$ is in $B_{2,1}^{2\alpha-1}$
almost surely due to the probabilistic nature of $X$. The following
lemma highlights how the stochastic setting nicely complements the
analytical foundation.
\begin{lem}
\label{lem:exResonant} Suppose $X$ is given by (\ref{eq:random function})
and satisfies Assumption~\ref{ass:Z}, then
\[
X\in B_{4,1}^{\alpha}\quad\text{and\quad}\pi(X,\dd X)\in B_{2,1}^{2\alpha-1}
\]
 almost surely for any $\alpha<s+\frac{r}{4}-\frac{1}{2}.$\end{lem}
\begin{proof}
We start as in the classical proof of Bony's estimate (Lemma \ref{lem:paraproduct}~(iii),
cf. \citep[Thm. 2.85]{Bahouri2011}), and decompose 
\[
\pi(X,\dd X)=\sum_{j\geq-1}R_{j}\quad\mbox{with}\quad R_{j}:=\sum_{|\nu|\le1}(\Delta_{j-\nu}X)(\Delta_{j}\dd X).
\]
By the properties of the Littlewood-Paley blocks the Fourier transform
of $R_{j}$ is supported in $2^{j}$ times some fixed ball. Consequently,
$\Delta_{j'}R_{j}=0$ if $j'\gtrsim j$ and thus 
\[
\big\|\Delta_{j'}\pi(X,\dd X)\big\|_{L^{2}}=\Big\|\sum_{j\gtrsim j'}\Delta_{j'}R_{j}\Big\|_{L^{2}}\lesssim\sum_{j\gtrsim j'}\sum_{|\nu|\le1}\|(\Delta_{j-\nu}X)(\Delta_{j}\dd X)\|_{L^{2}}.
\]
Now we proceed similarly to Lemma~\ref{lem:RegX} (using again the
multi-indices $\lambda=(j,k)$):
\begin{align*}
 & \E\big[\|(\Delta_{j-\nu}X)(\Delta_{j}\dd X)\|_{L^{2}}^{2}\big]\\
 & \quad=\int_{\R}\E\Big[\Big(\sum_{\lambda_{1},\lambda_{2}}Z_{\lambda_{1}}Z_{\lambda_{2}}(\Delta_{j-\nu}\psi_{\lambda_{1}})(\Delta_{j}\psi'_{\lambda_{2}})\Big)^{2}\Big]\d t\\
 & \quad=\sum_{\substack{\lambda_{1},\dots,\lambda_{4}:\\
|\lambda_{\cdot}|\sim j
}
}\E\big[Z_{\lambda_{1}}Z_{\lambda_{2}}Z_{\lambda_{3}}Z_{\lambda_{3}}\big]\int_{\R}(\Delta_{j-\nu}\psi_{\lambda_{1}})(\Delta_{j}\psi'_{\lambda_{2}})(\Delta_{j-\nu}\psi_{\lambda_{3}})(\Delta_{j}\psi'_{\lambda_{4}})\d t\\
 & \quad\leq\sum_{\substack{\lambda_{1}\neq\lambda_{2}:\\
|\lambda_{\cdot}|\sim j
}
}\E\big[Z_{\lambda_{1}}^{2}Z_{\lambda_{2}}^{2}\big]\int_{\R}\Big((\Delta_{j-\nu}\psi_{\lambda_{1}})^{2}(\Delta_{j}\psi'_{\lambda_{2}})^{2}+(\Delta_{j-\nu}\psi_{\lambda_{1}})(\Delta_{j}\psi'_{\lambda_{1}})(\Delta_{j-\nu}\psi_{\lambda_{2}})(\Delta_{j}\psi'_{\lambda_{2}})\Big)\d t\\
 & \qquad+\sum_{\lambda:|\lambda|\sim j}\E\big[Z_{\lambda}^{4}\big]\int_{\R}(\Delta_{j-\nu}\psi_{\lambda})^{2}(\Delta_{j}\psi'_{\lambda})^{2}\d t\\
 & \quad\lesssim\sum_{\substack{\lambda_{1}\neq\lambda_{2}:\\
|\lambda_{\cdot}|\sim j
}
}2^{-(4s+2r)j}\|\psi_{\lambda_{1}}\|_{L^{4}}\|\psi'_{\lambda_{2}}\|_{L^{4}}\big(\|\psi_{\lambda_{1}}\|_{L^{4}}\|\psi'_{\lambda_{2}}\|_{L^{4}}+\|\psi_{\lambda_{2}}\|_{L^{4}}\|\psi'_{\lambda_{1}}\|_{L^{4}}\big)\\
 & \qquad+\sum_{\lambda:|\lambda|\sim j}2^{-(4s+r)j}\|\psi_{\lambda}\|_{L^{4}}^{2}\|\psi'_{\lambda}\|_{L^{4}}^{2}.
\end{align*}
Plugging in $\psi_{j,k}=2^{j/2}\psi(2^{j}\cdot-k)$, we obtain 
\[
\E\big[\|(\Delta_{j-\nu}X)(\Delta_{j}\dd X)\|_{L^{2}}\big]\lesssim2^{-j(2s+r/2-2)}.
\]
The assertion follows from Lemma~\ref{lem:BesovBall} by the compact
support of $\F R_{j}$ for $j\geq-1$.
\end{proof}
Combining the two previous lemmas, we conclude that stochastic models
of the form (\ref{eq:random function}) are prototypical examples
of geometric Besov rough paths, which were introduced in Definition~\ref{def:geometric rough path},
and thus Theorem~\ref{thm:gobalsolution} can be applied to the corresponding
stochastic differential equations. 
\begin{prop}
Let $\phi$ satisfy Assumption~\ref{ass:phi} and $X=(X^{1},\dots,X^{n})$
be an $n$-dimensional stochastic process. Suppose each component
$X^{d}$, $d=1,\dots n$, is of the form (\ref{eq:random function}),
fulfills Assumption~\ref{ass:Z} for $\frac{5}{6}<s+\frac{r}{4}$
and the corresponding coefficients $(Z_{j,k}^{d})$ and $(Z_{j,k}^{m})$
are independent for $d\ne m$ and all $j,k$. Then, the localized
process $\phi X$ can be enhanced to a geometric Besov rough path,
that is $\phi X\in\mathcal{B}_{4,1}^{0,\alpha}$ almost surely for
$\alpha\in(\frac{1}{3},s+\frac{r}{4}-\frac{1}{2})$. \end{prop}
\begin{proof}
The regularity for each component $X^{d}$, $d=1,\dots,n$, is determined
by Lemma~\ref{lem:RegX} and thus $X\in B_{4,1}^{\alpha}$ for $\alpha\in(\frac{1}{3},s+\frac{r}{4}-\frac{1}{2})$.
Furthermore, a smooth approximation is given by the projection of
$X$ onto the first $J\ge1$ Littlewood-Paley blocks as used in the
proof of Lemma~\ref{lem:RegX} or similarly by projecting on the
first $J\ge1$ wavelet resolution levels. 

The resonant terms $\pi(X^{d},\dd X^{d})$, $d=1,\dots,n$, are constructed
in Lemma~\ref{lem:exResonant} again by a smooth approximation in
terms of Littlewood-Paley blocks. Due to the independence of the corresponding
coefficients $(Z_{j,k}^{d})$ and $(Z_{j,k}^{m})$ for $d\ne m$,
an analogous calculation shows that the resonant terms $\pi(X^{d},\dd X^{m})$
for $d\neq m$ exists as limit of the same approximation in terms
of Littlewood-Paley blocks, too. 

It remains to deduce the above results for the localized process $\phi X$
as well. The regularity and approximation of $\phi X$ is implied
by Lemma~\ref{lem:antiderivative}. For the resonant term $\pi(\phi X,\dd(\phi X))$
we observe that 
\[
\pi(\phi X,\dd(\phi X))=\pi(\phi X,\phi^{\prime}X)+\pi(\phi X,\phi\dd X),
\]
where the first term turns out to be no issue thanks to Lemma~\ref{lem:paraproduct}.
For the second one we apply Bony's decomposition to $\phi X$ and
our commutator lemma (Lemma~\ref{lem:commutator}) to get 
\[
\pi(\phi X,\phi\dd X)=\phi\pi(X,\phi\dd X)+\phi\Gamma(\phi,X,\phi\dd X)+\pi(\pi(\phi,X),\phi\dd X)+X\pi(\phi,\phi\dd X)+\Gamma(X,\phi,\phi\dd X).
\]
Due to the regularity of $\phi$ and $X$ it remains to only handle
the first term. By another analogous application of the commutator
lemma, we finally see that the approximation of the resonant term
of the localized process can be deduced from the above approximation
of the non-localized process and therefore $\phi X\in\mathcal{B}_{4,1}^{0,\alpha}$.
\end{proof}

\appendix

\section{Appendix\label{sec:appendix}}

\subsection{Nonhomogeneous Besov spaces}

In this part of the appendix we collect for the reader's convenience
some results which allow to estimate the Besov norm of a function.
For a general introduction to Littlewood-Paley theory and Besov spaces
we recommend \citet{triebel2010} as well as \citet{Bahouri2011}.
\begin{lem}
\citep[Lem. 2.69]{Bahouri2011}\label{lem:BesovAnnulus} Let $\mathcal{A}\subset\mathbb{R}^{d}$
be an annulus, $\alpha\in\mathbb{R}$ and $p,q\in[1,\infty]$. Suppose
that $(f_{j})$ is a sequence of smooth functions such that 
\[
\supp\mathcal{F}f_{j}\subset2^{j}\mathcal{A}\quad\text{and}\quad\big\|\big(2^{\alpha j}\|f_{j}\|_{L^{p}}\big)_{j}\big\|_{\ell^{q}}<\infty.
\]
Then $f:=\sum_{j}f_{j}$ satisfies 
\[
f\in B_{p,q}^{\alpha}(\R^{d})\quad\text{and}\quad\|f\|_{\alpha,p,q}\lesssim\big\|\big(2^{\alpha j}\|f_{j}\|_{L^{p}}\big)_{j}\big\|_{\ell^{q}}.
\]

\begin{lem}
\citep[Lem. 2.84]{Bahouri2011}\label{lem:BesovBall} Let $\mathcal{B}\subset\mathbb{R}^{d}$
be a ball, $\alpha\in\mathbb{R}$ and $p,q\in[1,\infty]$. Suppose
that $(f_{j})$ is a sequence of smooth functions such that 
\[
\supp\mathcal{F}f_{j}\subset2^{j}\mathcal{B}\quad\text{and}\quad\big\|\big(2^{\alpha j}\|f_{j}\|_{L^{p}}\big)_{j}\big\|_{\ell^{q}}<\infty.
\]
Then $f:=\sum_{j}f_{j}$ satisfies 
\[
f\in B_{p,q}^{\alpha}(\R^{d})\quad\text{and}\quad\|f\|_{\alpha,p,q}\lesssim\big\|\big(2^{\alpha j}\|f_{j}\|_{L^{p}}\big)_{j}\big\|_{\ell^{q}}.
\]

\begin{lem}
\citep[Prop. 2.79]{Bahouri2011}\label{lem:bahouri} Let $p,q\in[1,\infty]$,
$\alpha<0$ and $f$ be a tempered distribution. Then, $f\in B_{p,q}^{\alpha}(\R^{d})$
if and only if 
\[
\big(2^{\alpha j}\|S_{j}f\|_{L^{p}}\big)_{j}\in\ell^{q},
\]
where we recall $S_{j}f:=\sum_{k=-1}^{j-1}\Delta_{k}f$. Furthermore,
there exists a constant $C>0$ such that 
\[
C^{-|\alpha|+1}\|f\|_{\alpha,p,q}\leq\big\|\big(2^{\alpha j}\|S_{j}f\|_{L^{p}}\big)_{j}\big\|_{\ell^{q}}\leq C\bigg(1+\frac{1}{|\alpha|}\bigg)\|f\|_{\alpha,p,q}.
\]

\end{lem}
\end{lem}
\end{lem}

\subsection{Proof of Lemma~\ref{lem:lipschitz}: Lipschitz continuity\label{app:proofLipschitz}}

This subsection is devoted to the proof of Lemma~\ref{lem:lipschitz}.
For $j=1,2$ let $u_{0}^{j}\in\mathbb{R}^{d}$ and $\theta_{\mathcal{T}}^{j}\in C_{\mathcal{T}}^{\infty}$
with derviative $\xi_{\mathcal{T}}^{j}=\dd\theta_{\mathcal{T}}^{j}$.
Denote by $u^{j}$, $j=1,2,$ the solutions to corresponding Cauchy
problems~(\ref{eq:localRDE}) and $\tilde{u}^{j}=\psi u^{j}$ for
a weight function $\psi$ satisfying Assumption~\ref{ass:weight}.
Then Lemma~\ref{lem:lipschitz} is proven if we can show that
\[
\|\tilde{u}^{1}-\tilde{u}^{2}\|_{\alpha,p,q}\le C\big(\|\theta_{\mathcal{T}}^{1}-\theta_{\mathcal{T}}^{2}\|_{\alpha,p,q}+\|\pi(\theta_{\mathcal{T}}^{1},\xi_{\mathcal{T}}^{1})-\pi(\theta_{\mathcal{T}}^{2},\xi_{\mathcal{T}}^{2})\|_{2\alpha-1,p/2,q}\big),
\]
for a constant $C$ which does not depend on $\tilde{u}$. Roughly
speaking, the verification of this bound follows the pattern of the
proofs of Proposition~\ref{prop:BoundResonant} and Corollary~\ref{cor:bound}.
However, since Lemma~\ref{lem:lipschitz} is essential for one of
our main results, we shall present it here in full length. 

Taking another weight function $\psi_{2}$ fulfilling Assumption~\ref{ass:weight}
and keeping Remark~\ref{rem:weights} in mind, we obtain
\begin{align}
 & \|\tilde{u}^{1}-\tilde{u}^{2}\|_{\alpha,p,q}\lesssim\|\psi_{2}(\tilde{u}^{1}-\tilde{u}^{2})\|_{\alpha,p,q}\nonumber \\
 & \quad\lesssim(\mathcal{T}^{2}\vee1)\big(|u^{1}(0)-u^{2}(0)|+\|\dd(\tilde{u}^{1}-\tilde{u}^{2})\|_{\alpha-1,p,q}\big)\nonumber \\
 & \quad\leq(\mathcal{T}^{2}\vee1)\big(|u^{1}(0)-u^{2}(0)|+\|\dd(T_{F(\tilde{u}^{1})}\theta_{\mathcal{T}}^{1}-T_{F(\tilde{u}^{2})}\theta_{\mathcal{T}}^{2})\|_{\alpha-1,p,q}+\|\dd(u^{\#,1}-u^{\#,2})\|_{\alpha-1,p,q}\big),\label{eq:lipschitzproof}
\end{align}
where Lemma~\ref{lem:antiderivative} is used in the second line
and the paracontrolled ansatz $\tilde{u}^{j}=T_{F(\tilde{u}^{j})}\theta_{\mathcal{T}}^{j}+u^{\#,j}$
in the third one. Let us continue by further estimating the term $\dd(T_{F(\tilde{u}^{1})}\theta_{\mathcal{T}}^{1}-T_{F(\tilde{u}^{2})}\theta_{\mathcal{T}}^{2})$.
Applying the Leibniz rule and the triangle inequality leads to 
\begin{align*}
 & \|\dd(T_{F(\tilde{u}^{1})}\theta_{\mathcal{T}}^{1}-T_{F(\tilde{u}^{2})}\theta_{\mathcal{T}}^{2})\|_{\alpha-1,p,q}\\
 & \quad\leq\|T_{\dd F(\tilde{u}^{1})}\theta_{\mathcal{T}}^{1}-T_{\dd F(\tilde{u}^{2})}\theta_{\mathcal{T}}^{2}\|_{\alpha-1,p,q}+\|T_{F(\tilde{u}^{1})}\xi_{\mathcal{T}}^{1}-T_{F(\tilde{u}^{2})}\xi_{\mathcal{T}}^{2}\|_{\alpha-1,p,q}\\
 & \quad\leq\|T_{\dd F(\tilde{u}^{1})}(\theta_{\mathcal{T}}^{1}-\theta_{\mathcal{T}}^{2})\|_{\alpha-1,p,q}+\|T_{\dd F(\tilde{u}^{1})-\dd F(\tilde{u}^{2})}\theta_{\mathcal{T}}^{2}\|_{\alpha-1,p,q}+\|T_{F(\tilde{u}^{1})}(\xi_{\mathcal{T}}^{1}-\xi_{\mathcal{T}}^{2})\|_{\alpha-1,p,q}\\
 & \qquad+\|T_{F(\tilde{u}^{1})-F(\tilde{u}^{2})}\xi_{\mathcal{T}}^{2}\|_{\alpha-1,p,q}.
\end{align*}
Based on Lemma~\ref{lem:paraproduct}, Besov embeddings, the lifting
property of Besov spaces \citep[Thm. 2.3.8]{triebel2010}, (\ref{eq:F(u)})
and (\ref{eq:estFu}), one has 
\begin{align}
 & \|\dd(T_{F(\tilde{u}^{1})}\theta_{\mathcal{T}}^{1}-T_{F(\tilde{u}^{2})}\theta_{\mathcal{T}}^{2})\|_{\alpha-1,p,q}\nonumber \\
 & \quad\lesssim\|\dd F(\tilde{u}^{1})\|_{\alpha-1,p,q}\|\theta_{\mathcal{T}}^{1}-\theta_{\mathcal{T}}^{2}\|_{0,\infty,\infty}+\|\dd F(\tilde{u}^{1})-\dd F(\tilde{u}^{2})\|_{\alpha-1,p,q}\|\theta_{\mathcal{T}}^{2}\|_{0,\infty,\infty}\nonumber \\
 & \qquad+\|F\|_{\infty}\|\xi_{\mathcal{T}}^{1}-\xi_{\mathcal{T}}^{2}\|_{\alpha-1,p,q}+\|F(\tilde{u}^{1})-F(\tilde{u}^{2})\|_{\infty}\|\xi_{\mathcal{T}}^{2}\|_{\alpha-1,p,q}\nonumber \\
 & \quad\lesssim\|F(\tilde{u}^{1})\|_{\alpha,p,q}\|\theta_{\mathcal{T}}^{1}-\theta_{\mathcal{T}}^{2}\|_{\alpha,p,q}+\|F(\tilde{u}^{1})-F(\tilde{u}^{2})\|_{\alpha,p,q}\|\theta_{\mathcal{T}}^{2}\|_{\alpha,p,q}+\|F\|_{\infty}\|\xi_{\mathcal{T}}^{1}-\xi_{\mathcal{T}}^{2}\|_{\alpha-1,p,q}\nonumber \\
 & \qquad+\|F^{\prime}\|_{\infty}\|\tilde{u}^{1}-\tilde{u}^{2}\|_{\infty}\|\xi_{\mathcal{T}}^{2}\|_{\alpha-1,p,q}\nonumber \\
 & \quad\lesssim\|F\|_{C^{1}}\|\tilde{u}^{1}\|_{\alpha,p,q}\|\theta_{\mathcal{T}}^{1}-\theta_{\mathcal{T}}^{2}\|_{\alpha,p,q}+\|F^{\prime}\|_{\infty}\|\theta_{\mathcal{T}}^{2}\|_{\alpha,p,q}\|\tilde{u}^{1}-\tilde{u}^{2}\|_{\alpha,p,q}+\|F\|_{\infty}\|\xi_{\mathcal{T}}^{1}-\xi_{\mathcal{T}}^{2}\|_{\alpha-1,p,q}\nonumber \\
 & \qquad+\|F^{\prime}\|_{\infty}\|\xi_{\mathcal{T}}^{2}\|_{\alpha-1,p,q}\|\tilde{u}^{1}-\tilde{u}^{2}\|_{\alpha,p,q}\nonumber \\
 & \quad\lesssim\|F\|_{C^{1}}\big(1+\|\tilde{u}^{1}\|_{\alpha,p,q}+\|\xi_{\mathcal{T}}^{2}\|_{\alpha-1,p,q}+\|\theta_{\mathcal{T}}^{2}\|_{\alpha,p,q}\big)\nonumber \\
 & \qquad\times\big(\|\xi_{\mathcal{T}}^{1}-\xi_{\mathcal{T}}^{2}\|_{\alpha-1,p,q}+\|\theta_{\mathcal{T}}^{1}-\theta_{\mathcal{T}}^{2}\|_{\alpha,p,q}+\|\tilde{u}^{1}-\tilde{u}^{2}\|_{\alpha,p,q}\big).\label{eq:lipschitzproof2}
\end{align}
It remains to consider the difference of derivatives $\dd\tilde{u}^{\#,j}$,
which can be decomposed (cf. (\ref{eq:phi1})) into 
\begin{align*}
\dd\tilde{u}^{\#,j} & =\pi(F(\tilde{u}^{j}),\xi_{\mathcal{T}}^{j})+T_{\xi_{\mathcal{T}}^{j}}(F(\tilde{u}^{j}))-T_{\dd F(\tilde{u}^{j})}\theta_{\mathcal{T}}^{j}+\frac{\psi^{\prime}}{\psi}\tilde{u}^{j}\quad\text{for }j=1,2.
\end{align*}
Applying Proposition~\ref{prop:reduction}, we can rewrite the resonant
term, differently than in the proof of Proposition~\ref{prop:BoundResonant},
as 
\begin{equation}
\pi(F(\tilde{u}^{j}),\xi_{\mathcal{T}}^{j})=F^{\prime}(\tilde{u}^{j})\pi(\tilde{u}^{j},\xi_{\mathcal{T}}^{j})+\Pi_{F}(\tilde{u}^{j},\xi_{\mathcal{T}}^{j})\label{eq:resDecomp-1}
\end{equation}
and, taking the ansatz $\tilde{u}^{j}=T_{F(\tilde{u}^{j})}\theta_{\mathcal{T}}^{j}+u^{\#,j}$
into account and applying the commutator Lemma~\ref{lem:commutator},
we have
\begin{align*}
\pi(\tilde{u}^{j},\xi_{\mathcal{T}}^{j}) & =\pi(T_{F(\tilde{u}^{j})}\theta_{\mathcal{T}}^{j},\xi_{\mathcal{T}}^{j})+\pi(u^{\#,j},\xi_{\mathcal{T}}^{j})\\
 & =F(\tilde{u}^{j})\pi(\theta_{\mathcal{T}}^{j},\xi_{\mathcal{T}}^{j})+\Gamma(F(\tilde{u}^{j}),\theta_{\mathcal{T}}^{j},\xi_{\mathcal{T}}^{j})+\pi(u^{\#,j},\xi_{\mathcal{T}}^{j}).
\end{align*}
Therefore, we decompose $\dd u^{\#,j}$ into the following seven terms
\begin{align*}
\dd\tilde{u}^{\#,j} & =F^{\prime}(\tilde{u}^{j})F(\tilde{u}^{j})\pi(\theta_{\mathcal{T}}^{j},\xi_{\mathcal{T}}^{j})+F^{\prime}(\tilde{u}^{j})\Gamma(F(\tilde{u}^{j}),\theta_{\mathcal{T}}^{j},\xi_{\mathcal{T}}^{j})+F^{\prime}(\tilde{u}^{j})\pi(u^{\#,j},\xi_{\mathcal{T}}^{j})+\Pi_{F}(\tilde{u}^{j},\xi_{\mathcal{T}}^{j})\\
 & \quad+T_{\xi_{\mathcal{T}}^{j}}(F(\tilde{u}^{j}))-T_{\dd F(\tilde{u}^{j})}\theta_{\mathcal{T}}^{j}+\frac{\psi^{\prime}}{\psi}\tilde{u}^{j}\\
 & =:D_{1}^{j}+\dots+D_{7}^{j}.
\end{align*}
Let us tackle the differences of these seven terms: The first term
is estimated as follows
\begin{align*}
 & \|D_{1}^{1}-D_{1}^{2}\|_{2\alpha-1,p/2,q}=\big\| F'(\tilde{u}^{1})F(\tilde{u}^{1})\pi(\theta_{\mathcal{T}}^{1},\xi_{\mathcal{T}}^{1})-F'(\tilde{u}^{2})F(\tilde{u}^{2})\pi(\theta_{\mathcal{T}}^{2},\xi_{\mathcal{T}}^{2})\big\|_{2\alpha-1,p/2,q}\\
 & \quad\lesssim\|F'(\tilde{u}^{1})F(\tilde{u}^{1})-F'(\tilde{u}^{2})F(\tilde{u}^{2})\|_{\alpha,p,q}\|\pi(\theta_{\mathcal{T}}^{1},\xi_{\mathcal{T}}^{1})\|_{2\alpha-1,p/2,q}\\
 & \qquad+\|F'(\tilde{u}^{2})F(\tilde{u}^{2})\|_{\alpha,p,q}\|\pi(\theta_{\mathcal{T}}^{1},\xi_{\mathcal{T}}^{1})-\pi(\theta_{\mathcal{T}}^{2},\xi_{\mathcal{T}}^{2})\|_{2\alpha-1,p/2,q}\\
 & \quad\lesssim\|\pi(\theta_{\mathcal{T}}^{1},\xi_{\mathcal{T}}^{1})\|_{2\alpha-1,p/2,q}\big(\|(F'(\tilde{u}^{1})-F'(\tilde{u}^{2}))F(\tilde{u}^{1})\|_{\alpha,p,q}+\|F'(\tilde{u}^{2})(F(\tilde{u}^{1})-F(\tilde{u}^{2}))\|_{\alpha,p,q}\big)\\
 & \qquad+\|F\|_{C^{2}}^{2}\|\tilde{u}^{2}\|_{\alpha,p,q}\|\pi(\theta_{\mathcal{T}}^{1},\xi_{\mathcal{T}}^{1})-\pi(\theta_{\mathcal{T}}^{2},\xi_{\mathcal{T}}^{2})\|_{2\alpha-1,p/2,q}\\
 & \quad\lesssim\|F\|_{C^{2}}^{2}\Big(\|\pi(\theta_{\mathcal{T}}^{1},\xi_{\mathcal{T}}^{1})\|_{2\alpha-1,p/2,q}\big(\|\tilde{u}^{1}\|_{\alpha,p,q}+\|\tilde{u}^{2}\|_{\alpha,p,q}\big)\|\tilde{u}^{1}-\tilde{u}^{2}\|_{\alpha,p,q}\\
 & \qquad+\|\tilde{u}^{2}\|_{\alpha,p,q}\|\pi(\theta_{\mathcal{T}}^{1},\xi_{\mathcal{T}}^{1})-\pi(\theta_{\mathcal{T}}^{2},\xi_{\mathcal{T}}^{2})\|_{2\alpha-1,p/2,q}\Big),
\end{align*}
where we refer to (\ref{eq:pointwiseMulti}), (\ref{eq:F(u)}), (\ref{eq:estFu})
and (\ref{eq:estimateProd}) for explanations to the above estimates.
Applying (\ref{eq:estimateProd}), Lemma~\ref{lem:commutator} and
Besov embeddings, we see for the next term that 
\begin{align*}
 & \|D_{2}^{1}-D_{2}^{2}\|_{2\alpha-1,p/2,q}=\big\| F'(\tilde{u}^{1})\Gamma(F(\tilde{u}^{1}),\theta_{\mathcal{T}}^{1},\xi_{\mathcal{T}}^{1})-F'(\tilde{u}^{2})\Gamma(F(\tilde{u}^{2}),\theta_{\mathcal{T}}^{2},\xi_{\mathcal{T}}^{2})\big\|_{2\alpha-1,p/2,q}\\
 & \quad\lesssim\|F^{\prime}\|_{\infty}\big(\|\Gamma(F(\tilde{u}^{1})-F(\tilde{u}^{2}),\theta_{\mathcal{T}}^{1},\xi_{\mathcal{T}}^{1})\|_{3\alpha-1,p/3,q}\\
 & \qquad+\|\Gamma(F(\tilde{u}^{2}),\theta_{\mathcal{T}}^{1}-\theta_{\mathcal{T}}^{2},\xi_{\mathcal{T}}^{1})\|_{3\alpha-1,p/3,q}+\|\Gamma(F(\tilde{u}^{2}),\theta_{\mathcal{T}}^{2},\xi_{\mathcal{T}}^{1}-\xi_{\mathcal{T}}^{2})\|_{3\alpha-1,p/3,q}\big)\\
 & \qquad+\|F'(\tilde{u}^{1})-F^{\prime}(\tilde{u}^{2})\|_{\infty}\|\Gamma(F(\tilde{u}^{2}),\theta_{\mathcal{T}}^{2},\xi_{\mathcal{T}}^{2})\|_{3\alpha-1,p/3,q}\\
 & \quad\lesssim\|F\|_{C^{1}}^{2}\Big(\|\theta_{\mathcal{T}}^{1}\|_{\alpha,p,q}\|\xi_{\mathcal{T}}^{1}\|_{\alpha-1,p,q}\|\tilde{u}^{1}-\tilde{u}^{2}\|_{\alpha,p,q}\\
 & \qquad+\|\xi_{\mathcal{T}}^{1}\|_{\alpha-1,p,q}\|\tilde{u}^{2}\|_{\alpha,p,q}\|\theta_{\mathcal{T}}^{1}-\theta_{\mathcal{T}}^{2}\|_{\alpha,p,q}+\|\tilde{u}^{2}\|_{\alpha,p,q}\|\theta_{\mathcal{T}}^{2}\|_{\alpha,p,q}\|\xi_{\mathcal{T}}^{1}-\xi_{\mathcal{T}}^{2}\|_{\alpha-1,p,q}\Big)\\
 & \qquad+\|F^{\prime\prime}\|_{\infty}\|F\|_{C^{1}}\|\tilde{u}^{2}\|_{\alpha,p,q}\|\theta_{\mathcal{T}}^{2}\|_{\alpha,p,q}\|\xi_{\mathcal{T}}^{2}\|_{\alpha-1,p,q}\|\tilde{u}^{1}-\tilde{u}^{2}\|_{\alpha,p,q}.
\end{align*}
For the third term, again due to (\ref{eq:estimateProd}) as well
as Lemma~\ref{lem:paraproduct} and Besov embeddings, we obtain 
\begin{align*}
 & \|D_{3}^{1}-D_{3}^{2}\|_{2\alpha-1,p/2,q}=\|F'(\tilde{u}^{1})\pi(u^{\#,1},\xi_{\mathcal{T}}^{1})-F'(\tilde{u}^{2})\pi(u^{\#,2},\xi_{\mathcal{T}}^{2})\|_{2\alpha-1,p/2,q}\\
 & \quad\lesssim\|F'(\tilde{u}^{1})\pi(u^{\#,1}-u^{\#,2},\xi_{\mathcal{T}}^{1})\|_{2\alpha-1,p/2,q}+\|F'(\tilde{u}^{1})\pi(u^{\#,2},\xi_{\mathcal{T}}^{1}-\xi_{\mathcal{T}}^{2})\|_{2\alpha-1,p/2,q}\\
 & \qquad+\|(F'(\tilde{u}^{1})-F'(\tilde{u}^{2}))\pi(u^{\#,2},\xi_{\mathcal{T}}^{2})\|_{2\alpha-1,p/2,q}\\
 & \quad\lesssim\|F'(\tilde{u}^{1})\|_{\infty}\|\pi(u^{\#,1}-u^{\#,2},\xi_{\mathcal{T}}^{1})\|_{3\alpha-1,p/3,q}+\|F'(\tilde{u}^{1})\|_{\infty}\|\pi(u^{\#,2},\xi_{\mathcal{T}}^{1}-\xi_{\mathcal{T}}^{2})\|_{3\alpha-1,p/3,q}\\
 & \qquad+\|\pi(u^{\#,2},\xi_{\mathcal{T}}^{2})\|_{3\alpha-1,p/3,q}\|F^{\prime}(\tilde{u}^{1})-F^{\prime}(\tilde{u}^{2})\|_{\infty}\\
 & \quad\lesssim\|F\|_{C^{1}}\|\xi_{\mathcal{T}}^{1}\|_{\alpha-1,p,q}\|u^{\#,1}-u^{\#,2}\|_{2\alpha,p/2,q}+\|F\|_{C^{1}}\|u^{\#,2}\|_{2\alpha,p/2,q}\|\xi_{\mathcal{T}}^{1}-\xi_{\mathcal{T}}^{2}\|_{\alpha-1,p,q}\\
 & \qquad+\|F^{\prime\prime}\|_{\infty}\|u^{\#,2}\|_{2\alpha,p/2,q}\|\xi_{\mathcal{T}}^{2}\|_{\alpha-1,p,q}\|\tilde{u}^{1}-\tilde{u}^{2}\|_{\alpha,p,q}.
\end{align*}
Proposition~\ref{prop:reduction} and the embedding $B_{p/3,\infty}^{3\alpha-1}\subset B_{p/2,q}^{2\alpha-1}$
yield for the fourth term 
\begin{align*}
\|D_{4}^{1}-D_{4}^{2}\|_{2\alpha-1,p/2,q} & =\|\Pi_{F}(\tilde{u}^{1},\xi_{\mathcal{T}}^{1})-\Pi_{F}(\tilde{u}^{2},\xi_{\mathcal{T}}^{2})\|_{2\alpha-1,p/2,q}\\
 & \lesssim\|\Pi_{F}(\tilde{u}^{1},\xi_{\mathcal{T}}^{1})-\Pi_{F}(\tilde{u}^{2},\xi_{\mathcal{T}}^{2})\|_{3\alpha-1,p/3,\infty}\\
 & \lesssim\|F\|_{C^{3}}C(\tilde{u}^{1},\tilde{u}^{2},\xi_{\mathcal{T}}^{1},\xi_{\mathcal{T}}^{2})\Big(\|\tilde{u}^{1}-\tilde{u}^{2}\|_{\alpha,p,q}+\|\xi_{\mathcal{T}}^{1}-\xi_{\mathcal{T}}^{2}\|_{\alpha-1,p,q}\Big),
\end{align*}
where the constant $C(\tilde{u}^{1},\tilde{u}^{2},\xi_{\mathcal{T}}^{1},\xi_{\mathcal{T}}^{2})$
is given in Proposition~\ref{prop:reduction}. The fifth term can
be bounded by 
\begin{align*}
 & \|D_{5}^{1}-D_{5}^{2}\|_{2\alpha-1,p/2,q}=\|T_{\xi_{\mathcal{T}}^{1}}(F(\tilde{u}^{1}))-T_{\xi_{\mathcal{T}}^{2}}(F(\tilde{u}^{2}))\|_{2\alpha-1,p/2,q}\\
 & \quad\lesssim\|T_{\xi_{\mathcal{T}}^{1}-\xi_{\mathcal{T}}^{2}}(F(\tilde{u}^{1}))\|_{2\alpha-1,p/2,q}+\|T_{\xi_{\mathcal{T}}^{2}}(F(\tilde{u}^{1})-F(\tilde{u}^{2}))\|_{2\alpha-1,p/2,q}\\
 & \quad\lesssim\|F\|_{C^{1}}\|\tilde{u}^{1}\|_{\alpha,p,2q}\|\xi_{\mathcal{T}}^{1}-\xi_{\mathcal{T}}^{2}\|_{\alpha-1,p,2q}+\|\xi_{\mathcal{T}}^{2}\|_{\alpha-1,p,2q}\|F(\tilde{u}^{1})-F(\tilde{u}^{2})\|_{\alpha,p,2q}\\
 & \quad\lesssim\|F\|_{C^{1}}\|\tilde{u}^{1}\|_{\alpha,p,q}\|\xi_{\mathcal{T}}^{1}-\xi_{\mathcal{T}}^{2}\|_{\alpha-1,p,q}+\|F^{\prime}\|_{\infty}\|\xi_{\mathcal{T}}^{2}\|_{\alpha-1,p,q}\|\tilde{u}^{1}-\tilde{u}^{2}\|_{\alpha,p,q}
\end{align*}
because of Lemma~\ref{lem:paraproduct} and (\ref{eq:estFu}). For
the sixth term, the lifting property \citep[Thm. 2.3.8]{triebel2010},
an analog to (\ref{eq:estFu}) and (\ref{eq:F(u)}) yield 
\begin{align*}
 & \|D_{6}^{1}-D_{6}^{2}\|_{2\alpha-1,p/2,q}=\|T_{\dd F(\tilde{u}^{1})}\theta_{\mathcal{T}}^{1}-T_{\dd F(\tilde{u}^{2})}\theta_{\mathcal{T}}^{2}\|_{2\alpha-1,p/2,q}\\
 & \quad\lesssim\|T_{\dd F(\tilde{u}^{1})-\dd F(\tilde{u}^{2})}\theta_{\mathcal{T}}^{1}\|_{2\alpha-1,p/2,q}+\|T_{\dd F(\tilde{u}^{2})}(\theta_{\mathcal{T}}^{1}-\theta_{\mathcal{T}}^{2})\|_{2\alpha-1,p/2,q}\\
 & \quad\lesssim\|\dd F(\tilde{u}^{1})-\dd F(\tilde{u}^{2})\|_{\alpha-1,p,q}\|\theta_{\mathcal{T}}^{1}\|_{\alpha,p,q}+\|\dd F(\tilde{u}^{2})\|_{\alpha-1,p,q}\|\theta_{\mathcal{T}}^{1}-\theta_{\mathcal{T}}^{2}\|_{\alpha,p,q}\\
 & \quad\lesssim\|F(\tilde{u}^{1})-F(\tilde{u}^{2})\|_{\alpha,p,q}\|\theta_{\mathcal{T}}^{1}\|_{\alpha,p,q}+\|F(\tilde{u}^{2})\|_{\alpha,p,q}\|\theta_{\mathcal{T}}^{1}-\theta_{\mathcal{T}}^{2}\|_{\alpha,p,q}\\
 & \quad\lesssim\|F'\|_{\infty}\|\theta_{\mathcal{T}}^{1}\|_{\alpha,p,q}\|\tilde{u}^{1}-\tilde{u}^{2}\|_{\alpha,p,q}+\|F\|_{C^{1}}\|\tilde{u}^{2}\|_{\alpha,p,q}\|\theta_{\mathcal{T}}^{1}-\theta_{\mathcal{T}}^{2}\|_{\alpha,p,q}.
\end{align*}
Since $2\alpha-1<0$, the last difference $D_{7}^{1}-D_{7}^{2}$ can
be easily estimated by
\[
\|\frac{\psi^{\prime}}{\psi}(\tilde{u}^{1}-\tilde{u}^{2})\|_{2\alpha-1,p/2,q}\lesssim\|\frac{\psi^{\prime}}{\psi}(\tilde{u}^{1}-\tilde{u}^{2})\|_{L^{p/2}}\leq\|\frac{\psi^{\prime}}{\psi}\|_{\infty}\|\tilde{u}^{1}-\tilde{u}^{2}\|_{L^{p/2}}\lesssim(\mathcal{T}\vee1)\|\frac{\psi^{\prime}}{\psi}\|_{\infty}\|\tilde{u}^{1}-\tilde{u}^{2}\|_{\alpha,p,q}.
\]
 Defining the constants
\begin{align*}
 & \tilde{C}_{\tilde{u},u^{\#}}:=1+\sum_{i=1,2}\big(\|\tilde{u}^{j}\|_{\alpha,p,q}+\|\tilde{u}^{j}\|_{\alpha,p,q}^{2}+\|u^{\#,j}\|_{2\alpha,p/2,q}\big),\\
 & C_{\xi^{j},\theta^{j}}:=\|\theta_{\mathcal{T}}^{j}\|_{\alpha,p,q}+\|\theta_{\mathcal{T}}^{j}\|_{\alpha,p,q}^{2}+\|\pi(\theta_{\mathcal{T}}^{j},\xi_{\mathcal{T}}^{j})\|_{2\alpha-1,p/2,q},\quad j=1,2,\\
 & \tilde{C}{}_{\xi,\theta}:=1+C_{\xi^{1},\theta^{1}}+C_{\xi^{2},\theta^{2}},
\end{align*}
we altogether obtain
\begin{align*}
 & \|\dd u^{\#,1}-\dd u^{\#,2}\|_{2\alpha-1,p/2,q}\\
 & \quad\lesssim\tilde{C}{}_{\xi,\theta}\tilde{C}_{\tilde{u},u^{\#}}\big(\|F\|_{C^{3}}+\|F\|_{C^{2}}^{2}\big)\\
 & \qquad\times\Big(\|\tilde{u}^{1}-\tilde{u}^{2}\|_{\alpha,p,q}+\|u^{\#,1}-u^{\#,2}\|_{2\alpha,p/2,q}\\
 & \qquad\qquad+\|\xi_{\mathcal{T}}^{1}-\xi_{\mathcal{T}}^{2}\|_{\alpha-1,p,q}+\|\theta_{\mathcal{T}}^{1}-\theta_{\mathcal{T}}^{2}\|_{\alpha,p,q}+\|\pi(\theta_{\mathcal{T}}^{1},\xi_{\mathcal{T}}^{1})-\pi(\theta_{\mathcal{T}}^{2},\xi_{\mathcal{T}}^{2})\|_{2\alpha-1,p/2,q}\Big)\\
 & \qquad+(\mathcal{T}\vee1)\|\frac{\psi^{\prime}}{\psi}\|_{\infty}\|\tilde{u}^{1}-\tilde{u}^{2}\|_{\alpha,p,q}.
\end{align*}
The factor $\tilde{C}_{\tilde{u},u^{\#}}$ is (locally) bounded since
$\|\tilde{u}^{1}\|_{\alpha,p,q}$ and $\|\tilde{u}^{2}\|_{\alpha,p,q}$
can be bounded by Corollary~\ref{cor:bound} and $\|u^{\#,j}\|_{2\alpha,p/2,q}$,
for $j=1,2$, can be bounded analogously to (\ref{eq:u raute 1})
and (\ref{eq:u raute 2}) by 
\begin{align*}
\|u^{\#,j}\|_{2\alpha,p/2,q} & \lesssim(\mathcal{T}\vee1)\Big(\big(\|F\|_{\infty}\|\theta_{\mathcal{T}}^{j}\|_{L^{p}}+\|\tilde{u}^{j}\|_{L^{p}}\big)\\
 & \quad+C_{\xi^{j},\mathcal{\theta}^{j}}(\|F\|_{C^{2}}\vee\|F\|_{C^{2}}^{2})\big(\|\tilde{u}^{j}\|_{\alpha,p,q}+\|F\|_{\infty}\|\theta_{\mathcal{T}}^{j}\|_{\alpha,p,q}\big)+\|\frac{\psi^{\prime}}{\psi}\|_{\infty}\|\tilde{u}^{j}\|_{\alpha,p,q}\Big)\\
 & \lesssim(\mathcal{T}\vee1)\big(1+(\|F\|_{C^{2}}\vee\|F\|_{C^{2}}^{3})(1+\|\theta_{\mathcal{T}}^{j}\|)C_{\xi^{j},\theta^{j}}+\|\frac{\psi^{\prime}}{\psi}\|_{\infty}\big)(1+\|\tilde{u}^{j}\|_{\alpha,p,q}).
\end{align*}
Relying on the lifting property of Besov spaces together with the
definition of $u^{\#}$, $\|\tilde{u}^{1}-\tilde{u}^{2}\|_{L^{p/2}}\lesssim(\mathcal{T}\vee1)\|\tilde{u}^{1}-\tilde{u}^{2}\|_{L^{p}}$
and the compact support of $\theta_{\mathcal{T}}^{j}$, we have 
\begin{align*}
 & \|u^{\#,1}-u^{\#,2}\|_{2\alpha,p/2,q}\\
 & \quad\lesssim\|u^{\#,1}-u^{\#,2}\|_{L^{p/2}}+\|\dd u^{\#,1}-\dd u^{\#,2}\|_{2\alpha-1,p/2,q}\\
 & \quad\le\|T_{F(\tilde{u}^{1})}\theta_{\mathcal{T}}^{1}-T_{F(\tilde{u}^{2})}\theta_{\mathcal{T}}^{2}\|_{L^{p/2}}+\|\tilde{u}{}^{1}-\tilde{u}^{2}\|_{L^{p/2}}+\|\dd u^{\#,1}-\dd u^{\#,2}\|_{2\alpha-1,p/2,q}\\
 & \quad\le\|T_{F(\tilde{u}^{1})-F(\tilde{u}^{2})}\theta_{\mathcal{T}}^{1}\|_{L^{p/2}}+\|T_{F(\tilde{u}^{2})}(\theta_{\mathcal{T}}^{1}-\theta_{\mathcal{T}}^{2})\|_{L^{p/2}}+\|\tilde{u}{}^{1}-\tilde{u}^{2}\|_{L^{p/2}}\\
 & \qquad+\|\dd u^{\#,1}-\dd u^{\#,2}\|_{2\alpha-1,p/2,q}\\
 & \quad\lesssim(\mathcal{T}\vee1)\big(\|F(\tilde{u}^{1})-F(\tilde{u}^{2})\|_{\infty}\|\theta_{\mathcal{T}}^{1}\|_{L^{p}}+\|F\|_{\infty}\|\theta_{\mathcal{T}}^{1}-\theta_{\mathcal{T}}^{2}\|_{L^{p}}+\|\tilde{u}{}^{1}-\tilde{u}^{2}\|_{L^{p}}\big)\\
 & \qquad+\|\dd u^{\#,1}-\dd u^{\#,2}\|_{2\alpha-1,p/2,q}\\
 & \quad\lesssim(\mathcal{T}\vee1)\big(\|F^{\prime}\|_{\infty}\|\theta_{\mathcal{T}}^{1}\|_{\alpha,p,q}\|\tilde{u}^{1}-\tilde{u}^{2}\|_{\alpha,p,q}+\|F\|_{\infty}\|\theta_{\mathcal{T}}^{1}-\theta_{\mathcal{T}}^{2}\|_{\alpha,p,q}+\|\tilde{u}{}^{1}-\tilde{u}^{2}\|_{\alpha,p,q}\big)\\
 & \qquad+\|\dd u^{\#,1}-\dd u^{\#,2}\|_{2\alpha-1,p/2,q}.
\end{align*}
Therefore, if $\|F\|_{C^{3}}+\|F\|_{C^{2}}^{2}$ is sufficiently small,
depending on $\tilde{C}_{\xi,\theta}$, $\tilde{C}_{\tilde{u},u^{\#}}$
and $\mathcal{T}$, then

\begin{align*}
 & \|\dd u^{\#,1}-\dd u^{\#,2}\|_{2\alpha-1,p/2,q}\\
 & \quad\lesssim(1+\|\theta_{\mathcal{T}}^{1}\|_{\alpha,p,q})\tilde{C}_{\xi,\theta}\tilde{C}{}_{\tilde{u},u^{\#}}(\mathcal{T}\vee1)(\|F\|_{C^{3}}+\|F\|_{C^{2}}^{3})\\
 & \qquad\times\Big(\|\tilde{u}^{1}-\tilde{u}^{2}\|_{\alpha,p,q}+\|\xi_{\mathcal{T}}^{1}-\xi_{\mathcal{T}}^{2}\|_{\alpha-1,p,q}+\|\theta_{\mathcal{T}}^{1}-\theta_{\mathcal{T}}^{2}\|_{\alpha,p,q}\\
 & \qquad+\|\pi(\theta_{\mathcal{T}}^{1},\xi_{\mathcal{T}}^{1})-\pi(\theta_{\mathcal{T}}^{2},\xi_{\mathcal{T}}^{2})\|_{2\alpha-1,p/2,q}\Big)+(\mathcal{T}\vee1)\|\frac{\psi^{\prime}}{\psi}\|_{\infty}\|\tilde{u}^{1}-\tilde{u}^{2}\|_{\alpha,p,q}.
\end{align*}
Plugging this estimate and (\ref{eq:lipschitzproof2}) into (\ref{eq:lipschitzproof}),
we obtain
\begin{align*}
 & \|\tilde{u}^{1}-\tilde{u}^{2}\|_{\alpha,p,q}\\
 & \quad\lesssim(\mathcal{T}^{2}\vee1)|u^{1}(0)-u^{2}(0)|+(1+\|\theta_{\mathcal{T}}^{1}\|_{\alpha,p,q})\tilde{C}_{\xi,\theta}\tilde{C}_{\tilde{u},u^{\#}}(\mathcal{T}\vee1)(\|F\|_{C^{3}}+\|F\|_{C^{2}}^{3})\\
 & \qquad\times\Big(\|\tilde{u}^{1}-\tilde{u}^{2}\|_{\alpha,p,q}+\|\xi_{\mathcal{T}}^{1}-\xi_{\mathcal{T}}^{2}\|_{\alpha-1,p,q}+\|\theta_{\mathcal{T}}^{1}-\theta_{\mathcal{T}}^{2}\|_{\alpha,p,q}\\
 & \qquad+\|\pi(\theta_{\mathcal{T}}^{1},\xi_{\mathcal{T}}^{1})-\pi(\theta_{\mathcal{T}}^{2},\xi_{\mathcal{T}}^{2})\|_{2\alpha-1,p/2,q}\Big)+(\mathcal{T}^{2}\vee1)\|\frac{\psi^{\prime}}{\psi}\|_{\infty}\|\tilde{u}^{1}-\tilde{u}^{2}\|_{\alpha,p,q}.
\end{align*}
For a possibly smaller $\|F\|_{C^{3}}+\|F\|_{C^{2}}^{3}$ and a sufficiently
small $\|\frac{\psi^{\prime}}{\psi}\|_{\infty}$, we conclude
\begin{align*}
\|\tilde{u}^{1}-\tilde{u}^{2}\|_{\alpha,p,q} & \lesssim(\mathcal{T}^{2}\vee1)|u^{1}(0)-u^{2}(0)|+(1+\|\theta_{\mathcal{T}}^{1}\|_{\alpha,p,q})\tilde{C}_{\xi,\theta}\tilde{C}_{\tilde{u},u^{\#}}(\mathcal{T}\vee1)(\|F\|_{C^{3}}+\|F\|_{C^{2}}^{3})\\
 & \quad\times\Big(\|\theta_{\mathcal{T}}^{1}-\theta_{\mathcal{T}}^{2}\|_{\alpha,p,q}+\|\pi(\theta_{\mathcal{T}}^{1},\xi_{\mathcal{T}}^{1})-\pi(\theta_{\mathcal{T}}^{2},\xi_{\mathcal{T}}^{2})\|_{2\alpha-1,p/2,q}\Big).
\end{align*}
Finally, note again that $\tilde{C}_{\tilde{u},u^{\#}}$ is (locally)
bounded by Corollary~\ref{cor:bound}. \qed

\bibliographystyle{chicago}
\bibliography{quellenN}

\end{document}